\documentclass[11pt, oneside]{amsart}
\usepackage{amsmath, amsthm, amsfonts, amssymb}
\usepackage{amscd, amsxtra, latexsym}
\usepackage{mathtools}
\usepackage{mathrsfs}
\usepackage{tikz-cd}
\usepackage[shortlabels]{enumitem}
\usepackage{adjustbox}
\usepackage{tabularx}
\usepackage{microtype}
\usepackage{xcolor}
\usepackage{hyperref}
\hypersetup{
  colorlinks=true,
  linkcolor=black,
  citecolor=black,
  urlcolor=black
}
\usepackage{cleveref}

\newtheorem{prop}{Proposition}
\newtheorem{thm}[prop]{Theorem}

\newtheorem{lemma}{Lemma}[section]
\newtheorem{theorem}[lemma]{Theorem}
\newtheorem{proposition}[lemma]{Proposition}
\newtheorem{corollary}[lemma]{Corollary}
\theoremstyle{definition}
\newtheorem{example}[lemma]{Example}
\newtheorem{remark}[lemma]{Remark}

\theoremstyle{definition}
\newtheorem{claim}[lemma]{Claim}

\usepackage[
  top=0.99in,
  textheight=8.23in,
  textwidth=5.89in
]{geometry}

\begin{document}

\title[]{Geometric Kernels of Proper Maps Between Non-Compact Surfaces}

\author[Sumanta Das]{Sumanta Das}
\address{Department of Mathematics, Indian Institute of Technology Bombay, India}
\email{sumanta@math.iitb.ac.in}

\begin{abstract}
A map between connected \(2\)-manifolds has a \emph{geometric kernel} if it sends a non-contractible simple loop to a null-homotopic loop. While every non-\(\pi_1\)-injective map between compact surfaces admits a geometric kernel, this generally fails for compact bordered or non-compact surfaces. In this paper, we use Brown’s proper fundamental group to give a sufficient condition under which a degree-one map between non-compact surfaces admits a geometric kernel. Furthermore, we characterize conjugacy classes in the proper fundamental group and use this characterization to establish sufficient conditions for the existence of geometric kernels.
\end{abstract}
\maketitle
\section{Introduction}
Throughout, all manifolds are assumed to be orientable; they may be non-compact, have boundary, or be disconnected. A \emph{surface} (resp.\ a \emph{bordered surface}) is a connected \(2\)-dimensional manifold with empty (resp.\ non-empty) boundary. A connected \(2\)-manifold is said to be of \emph{finite type} if its fundamental group is finitely generated; otherwise, it is of \emph{infinite type}.

A map \( f\colon M \to N \) between connected $2$-manifolds is said to have a \emph{geometric kernel} if there exists a simple loop in \( M \) representing a non-trivial element of \( \ker \pi_1(f) \) \cite[Chapter 1]{MR2634047}. The existence of geometric kernels has long been a topic of interest. It traces back to a lemma of Tucker \cite[Lemma 1]{MR334218}, which asserts that every non-\(\pi_1\)-injective map between two  connected, compact \(2\)-manifolds that maps the boundary into the boundary admits a geometric kernel. However, in an unpublished correction to his original paper, Tucker later showed that this claim is not true in general, at least when the boundaries are non-empty.

Further progress on the existence of geometric kernels for maps between compact \(2\)-manifolds was made by Edmonds, who proved that a map \( f\colon M \to N \) between oriented, connected, closed \(2\)-manifolds admits a geometric kernel whenever its degree \( d \) satisfies \( \chi(M) \leq |d|(\chi(N) - 1) \) \cite[Theorem 4.5]{MR541331}.
 Skora’s thesis provided an alternative proof of this result and additionally established the existence of a geometric kernel for \( f \) when \( d = 3 \) or \( 4 \) \cite[Chapter 3]{MR2634047}. Later, Gabai proved the existence of a geometric kernel for any non-\(\pi_1\)-injective map between two connected, closed \(2\)-manifolds \cite[Theorem 2.1]{MR806708}.

However, when the domain and codomain are non-compact, a non-\( \pi_1 \)-injective map, even if proper, may not have a geometric kernel. This can be shown by modifying Tucker’s counterexample so that the role of boundaries is replaced by punctures; see \Cref{ex}. We aim to identify algebraic conditions on a proper map \( f\colon S' \to S \) between non-compact surfaces that guarantee the existence of a geometric kernel for \( f \). One strategy is to find conditions that ensure \( f \) is \( \pi_1 \)-injective outside a compact subset of \( S' \), thereby reducing the problem to the compact case. Our first theorem provides such a condition using Brown’s proper fundamental group, a tool that encodes information about all loops that lie near an end and are based at points along a given proper ray.

 \begin{thm}\label{cutoff}
Let \( f\colon S' \to S \) be a proper map between non-compact surfaces, where neither \( S' \) nor \( S \) is homeomorphic to \( \mathbb{R}^2 \). Suppose for every end $e$ of $S'$, there exists a representative $a\colon [0,\infty)\to S'$ of $e$ such that $\underline\pi_1(f)\colon \underline\pi_1(S',\underline a)\to \underline \pi_1(S,\underline{fa})$ is a monomorphism. Then there exist compact $2$-dimensional submanifolds \( D' \) of \( S' \) and \( D \) of \( S \) such that, after a proper homotopy, the restriction of \( f \) to the closure of each unbounded component of \( S' \setminus D' \) is a finite-sheeted covering map onto the closure of some unbounded component of \( S \setminus D \). 

%Moreover, under the additional assumption that $\underline\pi_1(f)$ is an isomorphism, the restriction $f\vert A'\to A$ becomes a homeomorphism.
\end{thm}
An analogue of \Cref{cutoff} also holds for connected, irreducible, end-irreducible open \(3\)-manifolds and can be proved using similar ideas; see \Cref{cutoff3}. In the special case where such \(3\)-manifolds are restricted to having exactly one end, the result was previously proved by Brown \cite[Theorem~2.4]{MR511426}. Note that in \Cref{cutoff}, if \(\pi_1(f)\) is a monomorphism (resp. isomorphism), then \(\underline{\pi}_1(f)\) is as well; however, the converse does not hold in general, as can be seen from the examples following \Cref{finitepinch}. The proof of \Cref{cutoff} relies on \Cref{classify} below, which classifies \(\pi_1\)-injective proper maps between non-compact bordered surfaces.

\begin{thm}\label{classify}
    Let $M$ and $N$ be non-compact bordered surfaces such that boundary components of $N$ are compact. Suppose there is a proper map $f\colon (M, \partial M) \to (N, \partial N)$ so that $\pi_1(f)$ is a monomorphism and $f$ restricted to each component of $\partial M$ is a covering map onto a component of $\partial N$. Then $f$ can be properly homotoped rel  $\partial M$ to a finite-sheeted covering map.
\end{thm}
The analogue of \Cref{classify} for compact bordered surfaces is due to Nielsen. Waldhausen proved an analogue for compact  Haken $3$-manifolds \cite[Theorem 6.1]{MR224099}, and later Brown and Tucker extended it to non-compact end-irreducible $3$-manifolds \cite[Theorem 4.2]{MR334225}. Our proof of \Cref{classify} makes essential use of ideas developed by Brown and Tucker, even though \Cref{classify} addresses a problem in one lower dimension. As an application of \Cref{cutoff}, we now state the following theorem.

 \begin{thm}\label{finitepinch}
 Let $f\colon S' \to S$ be a proper map of degree one between oriented non-compact surfaces, where \( S \) is not homeomorphic to \( \mathbb{R}^2 \). Suppose the induced map $\underline\pi_0(f)\colon \underline \pi_0(S')\to \underline\pi_0(S)$ between the spaces of ends is injective,  and for every end $e$ of $S'$, there exists a representative $a\colon [0,\infty)\to S'$ of $e$ such that $\underline\pi_1(f)\colon \underline\pi_1(S',\underline a)\to \underline \pi_1(S,\underline{fa})$ is a monomorphism. Then either $f$ has a geometric kernel or $f$ is properly homotopic to a homeomorphism. \end{thm}
 
 An example of a non-\(\pi_1\)-injective proper map of degree one satisfying the hypotheses of \Cref{finitepinch} is given by any quotient map \( q\colon S \to S \) from an infinite-genus surface to itself that collapses an essential compact bordered subsurface of genus \( g \geq 1 \) with one boundary component to a point. In fact, the proof of \Cref{finitepinch} shows that any non-\(\pi_1\)-injective proper map \( f \colon S' \to S \) of degree one satisfying the hypotheses of \Cref{finitepinch} can be properly homotoped so that an essential compact bordered subsurface of genus \( g \geq 1 \) with one boundary component is sent to a point. This conclusion follows from \Cref{cutoff}, combined with Edmonds’ classification theorem of allowable degree-one maps (see \Cref{allow}).

Let us examine the necessity of the various hypotheses in \Cref{finitepinch} for ensuring the existence of a geometric kernel. The hypothesis that \( f \) is a proper map of degree one is necessary. For example, there exists a non-\( \pi_1 \)-injective degree-two self-map of the Loch Ness monster surface that has no geometric kernel, yet induces a monomorphism on \( \underline{\pi}_1 \); see \Cref{LNM}. Moreover, the injectivity of \( \underline{\pi}_0(f) \) is also necessary, as demonstrated by \Cref{exnext}. On the other hand, the injectivity of \( \underline{\pi}_1(f) \) is not necessary. In fact, it is too strong a condition: as observed in the previous paragraph, it restricts \( f \) to pinching only finitely many handles. To illustrate this, consider the following example: Let \( S \) be an infinite-genus surface, and let \( \{ h_i : i \in \mathbb{N} \} \) be a pairwise disjoint collection of essential subsurfaces of \( S \), where each \( h_i \) is a compact bordered surface of genus one with one boundary component. The quotient map \( q' \colon S \to S \), which collapses each \( h_{2i} \) to a point, provides an example of a proper map of degree one with a geometric kernel, but it is not \( \underline{\pi}_1 \)-injective. Note that \( q' \) is \( \underline{\pi}_0 \)-injective, however.

The \( \underline{\pi}_1 \)-injectivity hypothesis is not only stronger than necessary to ensure the existence of a geometric kernel, but it also introduces technical complications: it requires choosing a base ray for each end. To avoid this dependence on representative rays, we seek a weaker condition that depends solely on the ends themselves. This motivates us to consider conjugacy classes in \( \underline{\pi}_1 \), analogous to the classical setting where conjugacy classes in \( \pi_1 \) are used to eliminate dependence on a basepoint. To make this precise, we define a set \( \widehat{\underline{\pi}_1}(S, e) \), where \( S \) is a surface and \( e \) is an end of \( S \), consisting of germ homotopy classes of sequences of loops in $S$ converging to \( e \); see \Cref{freehomotopy} for details. We show that this set coincides with the set of conjugacy classes in \( \underline{\pi}_1(S, \underline{a}) \), where \( a \) is any representative of the end \( e \); see \Cref{conjugacy}. Using this framework, we formulate a sufficient condition under which a degree-one map between infinite-genus surfaces may pinch infinitely many handles.

\begin{thm}\label{infinitepinch}
     Let \( S' \) and \( S \) be non-planar surfaces, each with finitely many ends, and let \( f \colon S' \to S \) be a proper map that is not \( \pi_1 \)-injective. Suppose that for every end \( e \) of \( S \), there exists a unique end \( e' \) of \( S' \) such that \( \underline{\pi}_0(f)(e') = e \), and that there exists a sequence \( \alpha \) of separating circles bounding \( e \) such that the preimage of the germ homotopy class of every power of \( \alpha \) under the induced map \( \widehat{\underline{\pi}_1}(f) \colon \widehat{\underline{\pi}_1}(S', e') \to \widehat{\underline{\pi}_1}(S, e) \) is a singleton. Then there exists a pairwise disjoint collection \( \{ h_i' \colon i \in \mathscr{A} \} \) of essential handles in \( S' \), with \( 1 \leq |\mathscr{A}| \leq \aleph_0 \), and a proper map \( g \colon S' \to S \) properly homotopic to \( f \) such that \( g(h_i') \) is a point for every \( i \in \mathscr{A} \), and this collection is infinite if and only if there exists an end \( e' \) of \( S' \) such that \( \widehat{\underline{\pi}_1}(f) \colon \widehat{\underline{\pi}_1}(S', e') \to \widehat{\underline{\pi}_1}(S, e) \) is not injective, where \( e = \underline{\pi}_0(f)(e') \).
\end{thm}

Here is an example of such a map. Let \( S \) be the Loch Ness monster surface, that is, the unique infinite-genus surface with exactly one end. Consider a pairwise disjoint infinite collection \( \{S_{1,2}^i \colon i \in \mathbb{N}\} \) of subsurfaces of \( S \), where each \( S_{1,2}^i \) is a compact genus-one surface with two boundary components. Define a quotient map \( q \colon S \to S \) that collapses a handle lying in the interior of \( S_{1,2}^i \) to a point in the interior of \( S_{1,2}^i \) for each even \( i \). The map \( q \) satisfies all the hypotheses of \Cref{infinitepinch}, where any infinite subcollection of the components of \( \bigcup_{i} \partial S_{1,2}^i \), possibly after reindexing, can be taken as the sequence \( \alpha \). Note that the key technical condition in \Cref{infinitepinch}, expressed in terms of \( \widehat{\underline{\pi}_1} \), ensures that \( f \) can be properly homotoped so that, for all but finitely many components \( \alpha_i \) of \( \alpha \), the restriction of \( f \) to each component of \( f^{-1}(\operatorname{im}(\alpha_i)) \) is a homeomorphism onto \( \operatorname{im}(\alpha_i) \).

In analogy with Edmonds' notion of allowability, we assume in Theorems \ref{finitepinch} and \ref{infinitepinch} that \( f \) is \emph{end-allowable}, meaning that it induces a bijection between the spaces of ends. However, when studying geometric kernels of proper maps between planar surfaces, the end-allowability condition turns out to be too strong, as it restricts attention to the special case where the map has degree zero.
\begin{thm}\label{deg0}
   Let \( S' \) and \( S \) be oriented planar surfaces, each having at least three ends. Suppose 
\( f \colon S' \to S \) is a proper map such that \( \underline{\pi}_0(f) \) is injective. Then \( f \) has a geometric kernel if and only if  \( \ker \pi_1(f) \neq 0 \), which holds if and only if \( \deg(f) = 0 \).
\end{thm}

On the other hand, if the end-allowability assumption is dropped from the hypotheses, then—even under the strongest possible assumptions on the maps induced at the level of \(\widehat{\underline{\pi}_1}\)—one can, to the best of the author's knowledge, expect at most a kernel appearing only at the homology level.

\begin{thm}\label{withoutallowable}
    Let \( S' \) and \( S \) be oriented planar surfaces, each with at least three ends, and let \( f \colon S' \to S \) be a proper map of degree one. Suppose there exists an end \( e \) of \( S \) such that the preimage \( \underline{\pi}_0(f)^{-1}(e) \) is a finite set of cardinality at least two, and that for each \( e' \in \underline{\pi}_0(f)^{-1}(e) \), the induced map \(\widehat{\underline{\pi}_1}(f) \colon \widehat{\underline{\pi}_1}(S', e') \to \widehat{\underline{\pi}_1}(S, e)\) is an isomorphism. Then there exists a simple loop in \( S' \) representing a non-trivial element of \( \ker  H_1(f) \).
\end{thm}

\subsection*{Outline of the paper} \Cref{S2} introduces general notation and definitions, and provides background on proper fundamental groups. In \Cref{S3}, we present examples demonstrating that a non-\(\pi_1\)-injective proper map between non-compact surfaces need not admit a geometric kernel. \Cref{S4} proves \Cref{cutoff} (\Cref{cut}) using \Cref{classify} (\Cref{specialcase}), and applies it to establish \Cref{finitepinch} (\Cref{deg}). Additionally, a three-dimensional analogue of \Cref{cutoff} is proved in this section (see \Cref{cutoff3}). Finally, \Cref{freehomotopy} characterizes conjugacy classes in the proper fundamental group (\Cref{conjugacy}) and uses this to prove \Cref{infinitepinch} (\Cref{infinitepinchproof}) and \Cref{withoutallowable} (\Cref{geokerplanar}). The proof of \Cref{deg0}, although it does not require this characterization, can also be found in this section (see \Cref{randomproposition}).

\section{Preliminaries}\label{S2}
\subsection{Notation and definitions}
Let \( X \) be a space, and let \( Y \) be a subspace of \( X \). The \emph{closure} (resp. \emph{interior}) of \( Y \) in \( X \) is denoted by \( \operatorname{cl}_X(Y) \) (resp., \( \operatorname{int}_X(Y) \)). The \emph{frontier} of \( Y \) in \( X \), denoted \( \operatorname{fr}_X(Y) \), is defined by \( \operatorname{fr}(Y) \coloneqq \operatorname{cl}_X(Y) \setminus \operatorname{int}_X(Y) \). We say that \( Y \) is \emph{unbounded} in \( X \) if \( \operatorname{cl}_X(Y) \) is not compact. When the ambient space \( X \) is clear from context, we omit the subscript in the notation for closure, interior, and frontier.

An \emph{exhausting sequence} for \( X \) is a sequence \( \{C_n\} \) of compact subsets of \( X \) such that \( C_n \subset \operatorname{int}(C_{n+1}) \) and \( \bigcup_n C_n = X \).

Let \( N \) be a manifold. We denote its \emph{boundary} by \( \partial N \) and its \emph{interior} by \( \iota N \), where \( \iota N \coloneqq N \setminus \partial N \). For integers \( g \geq 0 \), \( b \geq 0 \), and \( p \geq 0 \), let \( S_{g,b} \) denote the connected \( 2 \)-manifold of genus \( g \) with \( b \) boundary components, and let \( S_{g,b,p} \) denote the surface obtained by removing \( p \) points from \( \iota S_{g,b} \). For convenience, we will occasionally refer to \( S_{0,1} \), \( S_{0,2} \), \( S_{0,3} \), \( S_{1,1} \), \( S_{1,2} \), and \( S_{0,1,1} \) as a \emph{disk}, an \emph{annulus}, a \emph{pair of pants}, a \emph{handle}, a \emph{two-holed torus}, and a \emph{punctured disk}, respectively.

Let \( M \) be a connected two-dimensional manifold, possibly with boundary. A \emph{circle} in \( M \) is the image of an embedding of \( \mathbb{S}^1 \) into \( M \). A circle in \( M \) is said to be \emph{trivial} if it bounds an embedded disk in \( M \); otherwise, it is called \emph{non-trivial}. 
Note that if the image of an embedding \( \gamma \colon \mathbb{S}^1 \hookrightarrow M \) is a non-trivial circle in \( M \), then there does not exist \( g \in \pi_1(M) \) such that \( [\gamma] = g^k \) with \( |k| > 1 \); see \cite[Theorems 1.7 and 4.2]{MR214087}. A \emph{subsurface} of \( M \) is a connected two-dimensional submanifold—possibly with non-empty boundary—of \( M \). A subsurface \( S \subseteq M \) is said to be \emph{essential} (resp. \emph{properly embedded}) if the inclusion map \( S \hookrightarrow M \) is \( \pi_1 \)-injective (resp. proper).

A \emph{standard circle} \(c\) in \(\mathbb{R}^2\) is a set of the form \(\{z \in \mathbb{R}^2 : |z - a| = r\}\), where \(a \in \mathbb{R}^2\) is the \emph{center} of \(c\), and \(r \in (0, \infty)\) is its \emph{radius}. The \emph{interior} of a standard circle \(c\) in \(\mathbb{R}^2\), denoted by \(\operatorname{interior}(c)\), is defined as the bounded component of \(\mathbb{R}^2 \setminus c\).

Let \(P\) be a \(3\)-manifold, and let \(F\) be a smoothly embedded \(2\)-dimensional submanifold of \(P\) such that the inclusion \(F \hookrightarrow P\) is a proper map, and either \(F \cap \partial P = \partial F\) or \(F \subseteq \partial P\). We say that \(F\) is \emph{compressible} in \(P\) if one of the following holds:  (1) there exists a smoothly embedded copy \(E\) of \(\{x \in \mathbb{R}^3 \colon |x| \leq 1\}\) in \(P\) such that \(E \cap F = \partial E\); or  (2) there exists a smoothly embedded non-trivial circle \(c\) in \(\iota F\), and a smoothly embedded disk \(D\) in \(P\), with \(\iota D \subset \iota P\), such that \(D \cap F = \partial D = c\). We say that \(F\) is \emph{incompressible} in \(P\) if it is not compressible in \(P\). The manifold \(P\) is called \emph{irreducible} if every smoothly embedded copy of \(\{x \in \mathbb{R}^3 \colon |x| = 1\}\) in \(P\) is compressible, and \emph{boundary irreducible} if \(\partial P\) is incompressible.

It can be shown that \(F\) is incompressible in \(P\) if and only if every component of \(F\) is incompressible in \(P\)~\cite[Proposition 1.4]{MR2619608}. Moreover, if \(F\) is connected and not homeomorphic to \(S_{0,0}\), then by the loop theorem, \(F\) is incompressible if and only if the inclusion \(F \hookrightarrow P\) is \(\pi_1\)-injective~\cite[Proposition 1.5]{MR2619608}.

\subsection{Brown's proper fundamental group of an end}
For compact manifolds of the same dimension, the homotopy type can sometimes determine the homeomorphism type, though not always. In contrast, for non-compact manifolds of the same dimension, the homotopy type is often too weak to determine the homeomorphism type. To address this limitation, one can consider the \emph{proper homotopy type}, which captures the behavior of a space outside increasingly larger compact subsets, or equivalently, its behavior ``at infinity.''

The proper homotopy type is the analogue of the usual homotopy type, but it is formulated in the \emph{proper category}, whose objects are topological spaces and whose morphisms are proper maps. In classical homotopy theory, one seeks algebraic conditions---such as those in Whitehead's theorem---under which a map between suitably nice spaces is a homotopy equivalence. Similarly, in the proper category, identifying conditions that ensure a proper map between such spaces is a proper homotopy equivalence is one of the central themes of proper homotopy theory.

In 1974, Brown \cite{MR356041} introduced the notion of \emph{proper homotopy groups}, which complement the classical homotopy groups and, taken together, provide a version of Whitehead's theorem in the proper category. We focus on the one-dimensional case, the proper fundamental group, which suffices to classify proper maps between non-compact surfaces. This is consistent with the fact that the fundamental group alone determines the homotopy types of maps between aspherical manifolds. Moreover, there are instances in the $3$-manifold setting where the proper fundamental group alone can be used to distinguish homeomorphism types. For example, it distinguishes \(\mathbb{R}^3\) among irreducible contractible open \(3\)-manifolds with one end \cite[Corollary 3.3]{MR334225}. More generally, it can be used to show that, among irreducible contractible open \(3\)-manifolds of finite genus at infinity, the proper homotopy type determines the homeomorphism type \cite[Corollary 2.6]{MR511426}.

Let \( X \) and \( Y \) be spaces. We define an equivalence relation on the set of proper maps \( X\to Y \) as follows: Two proper maps \( f, g \colon X \to Y \) are said to be equivalent if there exists a compact subset \( C \) of \( X \) such that \( f(x) = g(x) \) for all \( x \in X \setminus C \). The equivalence class of a proper map \( f\colon X\to Y \) under this relation is denoted by \( \underline{f} \) and is called the \emph{germ} of \( f \).  

Two proper maps \( f_0, f_1 \colon X \to Y \) are \emph{germ homotopic} if there exists a proper map \( H\colon X\times [0,1]\to Y \) such that \( \underline{H_0} = \underline{f_0} \) and \( \underline{H_1} = \underline{f_1} \).

From now on, and throughout this section, unless otherwise stated, we assume that \( X \) and \( Y \) denote connected manifolds. While most of the theory can be developed for more general spaces, we restrict ourselves to connected manifolds for simplicity.
 
First, we recall the definition of the space of ends, a concept introduced by Freudenthal. To define this space, we consider an equivalence relation on the set of all proper maps from \( \underline * \coloneqq [0,\infty) \) to \( X \). Two proper maps \( a, b \colon \underline * \to X \) are considered equivalent if, for every compact subset \( C \subseteq X \), there exists \( t_C \geq 0 \) such that \( a(t) \) and \( b(t) \) lie in the same component of \( X \setminus C \) for all \( t \geq t_C \). The equivalence class of a proper map \( a \colon \underline * \to X \) is called an \emph{end} of \( X \), and is denoted by \( [a] \).

The set of all ends of \( X \), denoted \( \underline{\pi}_0(X) \), carries a natural topology whose basis consists of open sets of the form  
\(
\underline{A} \coloneqq \{e\in \underline \pi_0(X) \mid \exists a\in e \text{ such that }a([t,\infty)]) \subseteq A \text{ for some } t\geq 0\},
\)  
where \( A \) is an unbounded component of the complement of some compact subset of \( X \). Equipped with this topology, \( \underline{\pi}_0(X) \) is homeomorphic to a closed subset of the Cantor set.  

The definition of convergence to an end is now straightforward: we say that a sequence \( \{Z_n\} \) of subsets of \( X \) \emph{converges to the end} \( e \) of $X$ if, for every basic open neighborhood \( \underline A \) of \( e \) (as defined above), we have \( Z_n \subseteq A \) for all but finitely many \( n \). 

If \( f \colon X \to Y \) is a proper map, then it induces a continuous map $\underline{\pi}_0(f) \colon$ $\underline{\pi}_0(X) \to \underline{\pi}_0(Y)$, given by $[a] \longmapsto [fa] $. Moreover, this induced map depends only on the germ homotopy class of \( f \).

 We next associate to each end of $X$ a group analogous to the fundamental group $\pi_1(X)$. Let $\underline{\mathbb  S^1}$ be the space $\underline *$ together with a distinct circles attached at each integer point. Let $e$ be an end of $X$, and select as base point the germ $\underline a$ of a representative $a \in e$. A \emph{proper map of pairs} $\alpha\colon (\underline{\mathbb  S^1}, \underline *)\to (X, \underline a)$ means a proper map $\alpha\colon \underline{\mathbb  S^1}\to X$ so that the germ of $\alpha \vert \underline *$ is the germ of $a$. If $\beta\colon (\underline{\mathbb  S^1}, \underline *)\to (X, \underline a)$ is another proper map of pairs, we say that $\alpha$ and $\beta$ are \emph{germ homotopic rel $\underline *$} if there is a proper homotopy $H\colon\underline{\mathbb  S^1} \times [0,1] \to X$ so that the germ of $H_0$ is the germ of  $\alpha$, the germ of $H_1$ is the germ of $\beta$, and the germ of $H\vert\underline * \times [0,1]$ agrees with the germ of the composition $\underline * \times [0,1] \xrightarrow{p} \underline * \xrightarrow{a} X$, where $p$ is the natural projection. The equivalence class of $\alpha$ is denoted $[\alpha]$, and the set of all equivalence classes is denoted $\underline \pi_1(X, \underline a)$.

%, that is, $$\underline{\mathbb S^1}\coloneqq\frac{[0,\infty)\cup \bigcup_{k=0}^\infty \mathbb S^1\times \{k\}}{k\sim (1,k)}.$$ 

We now define a group structure on \( \underline{\pi}_1(X, \underline{a}) \). Let \( g, g' \in \underline{\pi}_1(X, \underline{a}) \), and choose representatives \( \alpha \in g \) and \( \alpha' \in g' \) such that \( \alpha|\underline{*} = \alpha'|\underline{*} \). Define a proper map of pairs \( \alpha \cdot \alpha' \colon  (\underline{\mathbb{S}^1}, \underline{*}) \to (X, \underline{a}) \) as follows: On \( \underline{*} \), define \( \alpha \cdot \alpha' \) by \( \alpha|\underline{*} \), and on \( \mathbb{S}_k^1 \), define \( \alpha \cdot \alpha' \) to be the concatenation of \( \alpha|\mathbb{S}_k^1 \) and \( \alpha'|\mathbb{S}_k^1 \), taken with respect to the basepoint \( \alpha(k) \), where \( \mathbb{S}_k^1 \) denotes the circle attached at the integer point \( k \). The product \( g \cdot g' \) is then defined as the equivalence class \( [ \alpha \cdot \alpha' ] \). This product is independent of the choice of representatives \( \alpha \in g \) and \( \alpha' \in g' \), and hence defines a group operation on \( \underline{\pi}_1(X, \underline{a}) \). We call \( \underline{\pi}_1(X, \underline{a}) \) the \emph{proper fundamental group of the end $e$ of \( X \) based at \( \underline{a} \)}.

The dependence of the proper fundamental group on the base germ can be seen as follows. Suppose \(b \colon [0, \infty) \to X\) is another representative of \(e\), so that \([a] = e = [b]\). Then there exists a sequence \(p = \{p_k \colon [0,1] \to X : k \geq 0\}\) of maps, called a \emph{path in \(X\) from \(\underline{a}\) to \(\underline{b}\)}, such that \(\{\operatorname{im}(p_k) : k \geq 0\}\) converges to \(e\), and for all but finitely many \(k\), we have \(p_k(0) = a(k)\) and \(p_k(1) = b(k)\). Now, if \(\alpha \colon (\underline{\mathbb{S}^1}, \underline{*}) \to (X, \underline{a})\) is a proper map of pairs, then there exists a proper map of pairs \(\alpha_p \colon (\underline{\mathbb{S}^1}, \underline{*}) \to (X, \underline{b})\) such that \(\alpha_p\vert\underline{*} = b\) and \(\alpha_p\vert\mathbb{S}^1_k = \overline{p_k} * (\alpha\vert\mathbb{S}^1_k) * p_k\) for all but finitely many \(k\), where \(\overline{p_k}\) denotes the inverse of the path \(p_k\). This defines a well-defined isomorphism \(p_* \colon \underline{\pi}_1(X, \underline{a}) \to \underline{\pi}_1(X, \underline{b})\), given by \(p_*([\alpha]) \coloneqq [\alpha_p]\).

The functoriality of \(\underline{\pi}_1\) follows from the way it assigns homomorphisms to proper maps: given a proper map \(f \colon X \to Y\), the induced homomorphism \(\underline{\pi}_1(f) \colon \underline{\pi}_1(X, \underline{a}) \to \underline{\pi}_1(Y, \underline{f a})\) is defined by \([\alpha] \mapsto [f \alpha]\). Moreover, for two proper maps \( f_0, f_1 \colon X \to Y \), if there exists a proper homotopy \( H \colon X \times [0,1] \to Y \) such that \( \underline{H_0} = \underline{f_0} \), \( \underline{H_1} = \underline{f_1} \), and the germ of \( H|\operatorname{im}(a) \times [0,1] \) agrees with the germ of the composition $\operatorname{im}(a) \times [0,1] \xrightarrow{p} \operatorname{im}(a) \xrightarrow{f} Y,$ where $p$ is the natural projection, then \( \underline{\pi}_1(f_0) = \underline{\pi}_1(f_1) \).

We conclude this section by introducing a few notations specific to \(2\)-manifolds. Let \(M\) be a \(2\)-manifold, and let \(e\) be an end of \(M\). The end \(e\) is said to be \emph{planar} if there exists a basic open neighborhood \(\underline{A}\) of \(e\) such that \(A\) embeds in \(\mathbb{R}^2\); otherwise, \(e\) is called \emph{non-planar}. We denote the set of all non-planar ends of \(M\) by \(\underline{\pi}_0^{\mathrm{np}}(M)\). The end \(e\) is said to be \emph{isolated} if it is an isolated point of the space of ends \(\underline{\pi}_0(M)\).

\section{Examples of Proper Maps Without Geometric Kernels}\label{S3}We present three types of examples illustrating that a non-\(\pi_1\)-injective proper map between two non-compact surfaces—possibly of infinite type—need not, in general, admit a geometric kernel. Each proper map constructed in these examples has degree either one or two. Here, \emph{degree} refers to the integral cohomological degree; that is, if \(f\colon (M_1, \partial M_1) \to (M_2, \partial M_2)\) is a proper map between connected, oriented, topological \(n\)-manifolds, then the \emph{(integral cohomological) degree} of \(f\) is the unique integer \(\deg(f)\) satisfying \(H^n_c(f)\big([M_2]\big) = \deg(f) \cdot [M_1]\), where \([M_j]\) denotes the preferred generator of the \(n\)th singular cohomology with compact support \(H^n_c(M_j, \partial M_j; \mathbb{Z}) \cong \mathbb{Z}\), compatible with the orientation of \(M_j\), for each \(j = 1, 2\). Several properties of the degree—such as proper homotopy invariance, multiplicativity, and geometric realization—are discussed in \cite{MR192475}.

Our first example is a modification of a counterexample due to Tucker.

\begin{example}\label{ex}
   There exist non-compact planar surfaces \( S' \) and \( S \), and a proper map \( f\colon S' \to S \) of degree \( 2 \), such that \( \ker \pi_1(f) \neq 0 \), but \( f \) has no geometric kernel. To construct such an example, we first recall Tucker's counterexample. Let \( c \) be a standard circle in \( \mathbb R^2 \), centered at the origin. Choose standard circles \( c_+ \) and \( c_- \) in \( \operatorname{interior}(c) \) such that neither \( c_+ \) nor \( c_- \) intersects the \( \mathrm{Y} \)-axis, and \( c_- \) is the reflection of \( c_+ \) across the \( \mathrm{Y} \)-axis. Let \( P \) be the compact bordered subsurface of $\mathbb R^2$ with three boundary components \( c \), \( c_+ \), and \( c_- \). Consider the quotient space $A$ obtained from \( P \) by identifying \( (x, y) \in P \) with \( (-x, y) \in P \) for \( x \neq 0 \), and identifying \( (0, y) \in P \) with \( (0, -y) \in P \). Notice that $P$ is a pair of pants and $A$ is an annulus. Moreover, the quotient map \( q\colon P \to A \) is a two-fold branched cover, with the branch point at \( q(0,0) \). In particular, $q\vert c_+\sqcup c_-\to q(c_+)=q(c_-)\subset\partial A$ and $q\vert c\to q(c)\subset\partial A$ are two-fold coverings. Since $\pi_1(A)$ is abelian and $\pi_1(P)$ is non-abelian, $\ker \pi_1(q)\neq 0$. But \( q \) has no geometric kernel, because in a complete hyperbolic pair of pants \( Y \) with three closed geodesic boundaries, any non-trivial simple loop in \( Y \) is freely homotopic to a simple loop whose image is a component of \( \partial Y \) \cite[Theorem 1.6.6]{MR2742784}.

  Let \( M \) be a non-compact bordered planar surface with one boundary component, possibly of infinite type, and let \( D_* \) be a punctured disk. Attach two copies of \( M \) and one copy of \( D_* \) to \( P \) by identifying a copy of \( \partial M \) with each of \( c_+ \) and \( c_- \), and identifying a copy of \( \partial D_* \) with \( c \). Let \( S' \) denote the resulting non-compact surface. Similarly, define \( S \) to be the non-compact surface obtained by attaching one copy of \( M \) and one copy of \( D_* \) to \( A \), by identifying a copy of \( \partial M \) with \( q(c_+) \) and a copy of \( \partial D_* \) with \( q(c) \). Then the map \( q\colon P \to A \) extends to a two-fold branched covering \( f\colon S' \to S \) such that \( f \) restricts to a homeomorphism from each copy of \( M \subset S' \) onto \( M \subset S \), and restricts to a two-fold covering map from \( D_* \subset S' \) onto \( D_* \subset S \). Moreover, \( \ker \pi_1(f)\neq0 \), since \( P \) is an essential subsurface of \( S' \) and \( \ker \pi_1(q)\neq 0 \).

We claim that \( q \) has no geometric kernel. Suppose, for contradiction, that there exists a non-trivial simple loop \( \gamma' \subset S' \) such that \( f(\gamma') \) is null-homotopic. Since \( \gamma' \) is simple, we may assume that \( \gamma' \cap D_* = \varnothing \). Choose an essential compact bordered subsurface \( N \) of \(M \subset S \) such that \( \partial M \) is a component of \( \partial N \) and \( \gamma' \) is contained in the interior of the essential compact bordered subsurface \( P \cup f^{-1}(N) \subset S' \). Thus, \( f \) restricts to a homeomorphism from each component of \( f^{-1}(N) \) onto \( N \).

Pick a tubular neighborhood \( \gamma' \times [-1,1] \) with \( \gamma' \times \{0\} \equiv \gamma' \), lying in the interior of \( P \cup f^{-1}(N) \). Let \( X' \) denote the $2$-manifold obtained by removing \( \gamma' \times (-1,1) \) from \( P \cup f^{-1}(N) \) and gluing a disk along each of \( \gamma' \times \{i\} \) for \( i = \pm 1 \). Since  \( \gamma' \) separates the planar surface \( S' \), it follows that \( X' \) has exactly two components. Denote the component of \( X' \) containing \( c \) by \( X'_c \). Notice that the map \( f\vert (P \cup f^{-1}(N)) \setminus \gamma' \times (-1,1) \) extends to a map \( \overline{f} \colon X'_c \to A \cup N \) because \( f(\gamma') \) is null-homotopic.

Since \( \overline f\vert f^{-1}(q(c)) = c \to q(c) \) is a two-fold covering, and \( \overline f \) restricts to a homeomorphism on each component of \( \partial X'_c \setminus c \) onto a component of \( \partial N \setminus q(c_+) \), the following commutative diagram
\[
\begin{tikzcd}
H_2(X'_c,\partial X'_c) \arrow[d, "H_2(\overline f)"'] \arrow[r] & H_1(\partial X'_c) \arrow[d, "H_1(\overline f\vert \partial X'_c)"] \\
H_2(A \cup N, \partial(A \cup N)) \arrow[r] & H_1(\partial(A \cup N))
\end{tikzcd}
\]where the horizontal maps are of the form \( \mathbb{Z} \ni 1 \mapsto \oplus_{i=1}^n 1 \in \oplus_{i=1}^n \mathbb{Z} \), implies that the left vertical map is multiplication by \( \pm 2 \). Hence, the preimage under \( \overline f \) of each component of \( \partial N \setminus q(c_+) \) has exactly two components. It follows that \( \partial X'_c = \partial (P \cup f^{-1}(N)) \), which is possible if and only if \( \gamma' \) bounds a disk in \( P \cup f^{-1}(N) \)—a contradiction, since \( \gamma' \) is non-trivial. \hfill$\square$
\end{example}

Our next example demonstrates the necessity of \( \underline{\pi}_0 \)-injectivity in the hypothesis of \Cref{finitepinch}. For this, we require \Cref{tool} stated below. In this result, the case of compact surfaces is due to Kneser, while the case of compact bordered surfaces follows from a combination of the following three facts: (1) the absolute degree and the integral cohomological degree coincide up to sign~\cite[Theorem 3.1]{MR192475}; (2) the absolute degree equals the geometric degree (Hopf’s Theorem)~\cite[Theorem 4.1]{MR192475},~\cite[Theorem 2.4]{MR875337}; and (3) Skora’s generalization of the Kneser theorem, which asserts that the Euler characteristic of the domain is at most the geometric degree times that of the codomain~\cite[Theorem 4.1]{MR875337}.

\begin{theorem}[Kneser-Epstein-Skora]\label{tool}
    Let \( F \) and \( G \) be connected, oriented, compact \( 2 \)-manifolds, and let \( f\colon (F,\partial F) \to (G,\partial G) \) be a map. If \( \deg(f)\neq0 \), then \(\chi(F) \leq |\deg(f)| \cdot \chi(G)\).
\end{theorem}Most of \Cref{tool} also follows as a special case of the degree estimate for Gromov's simplicial volume ~\cite[p.~8]{MR686042}. Indeed, if \( \varphi\colon M \to N \) is a proper map between finite-type surfaces with \( \chi(M), \chi(N) \leq 0 \), then \( |\deg(\varphi)| \cdot (-2\chi(N)) = |\deg(\varphi)| \cdot \|N\| \leq \|M\| = -2\chi(M) \). 

We now return to the setup of the previous example, where \( P \) denotes the same pair of pants with boundary components \( c \), \( c_+ \), and \( c_- \).\begin{example}\label{exnext}
    There exist non-compact surfaces \( S' \) and \( S \), and a proper map \( f\colon S' \to S \) of degree one, such that \( \ker \pi_1(f) \neq 0 \), but \( f \) has no geometric kernel. We first construct a non-\(\pi_1\)-injective map \(\varphi\) from \(P\) to the compact bordered surface \(A' \coloneqq P \cup \operatorname{interior}(c_-)\), satisfying \(\varphi(\partial P) \subseteq \partial A'\), such that \(\varphi\) has degree one but no geometric kernel. The construction of \(\varphi\) is based on a cell-by-cell extension process.

    Join \(c\) to \(c_+\) by a simple arc \(\lambda_1\) in \(P\), and \(c_-\) to \(c_+\) by a simple arc \(\lambda_2\) in \(P\), such that \(\lambda_i \cap \partial P\) is a two-point set for each \(i\). Moreover, we assume that \(\lambda_1 \cap \lambda_2 = \varnothing\). Thus, \(P^{(1)} \coloneqq \partial P \cup \lambda_1 \cup \lambda_2\) forms the 1-skeleton of a CW-structure of \(P\).  \begin{figure}[h]
        \centering
         \def\svgwidth{0.5\linewidth} 
         %% Creator: Inkscape 1.1.2 (b8e25be833, 2022-02-05), www.inkscape.org
%% Accompanies image file 'paper.pdf' (pdf, eps, ps)
\begingroup%
  \makeatletter%
  \providecommand\color[2][]{%
    \errmessage{(Inkscape) Color is used for the text in Inkscape, but the package 'color.sty' is not loaded}%
    \renewcommand\color[2][]{}%
  }%
  \providecommand\transparent[1]{%
    \errmessage{(Inkscape) Transparency is used (non-zero) for the text in Inkscape, but the package 'transparent.sty' is not loaded}%
    \renewcommand\transparent[1]{}%
  }%
  \providecommand\rotatebox[2]{#2}%
  \newcommand*\fsize{\dimexpr\f@size pt\relax}%
  \newcommand*\lineheight[1]{\fontsize{\fsize}{#1\fsize}\selectfont}%
  \ifx\svgwidth\undefined%
    \setlength{\unitlength}{467.26104973bp}%
    \ifx\svgscale\undefined%
      \relax%
    \else%
      \setlength{\unitlength}{\unitlength * \real{\svgscale}}%
    \fi%
  \else%
    \setlength{\unitlength}{\svgwidth}%
  \fi%
  \global\let\svgwidth\undefined%
  \global\let\svgscale\undefined%
  \makeatother%
  \begin{picture}(1,0.86084194)%
    \lineheight{1}%
    \setlength\tabcolsep{0pt}%
    \put(0,0){\includegraphics[width=\unitlength,page=1]{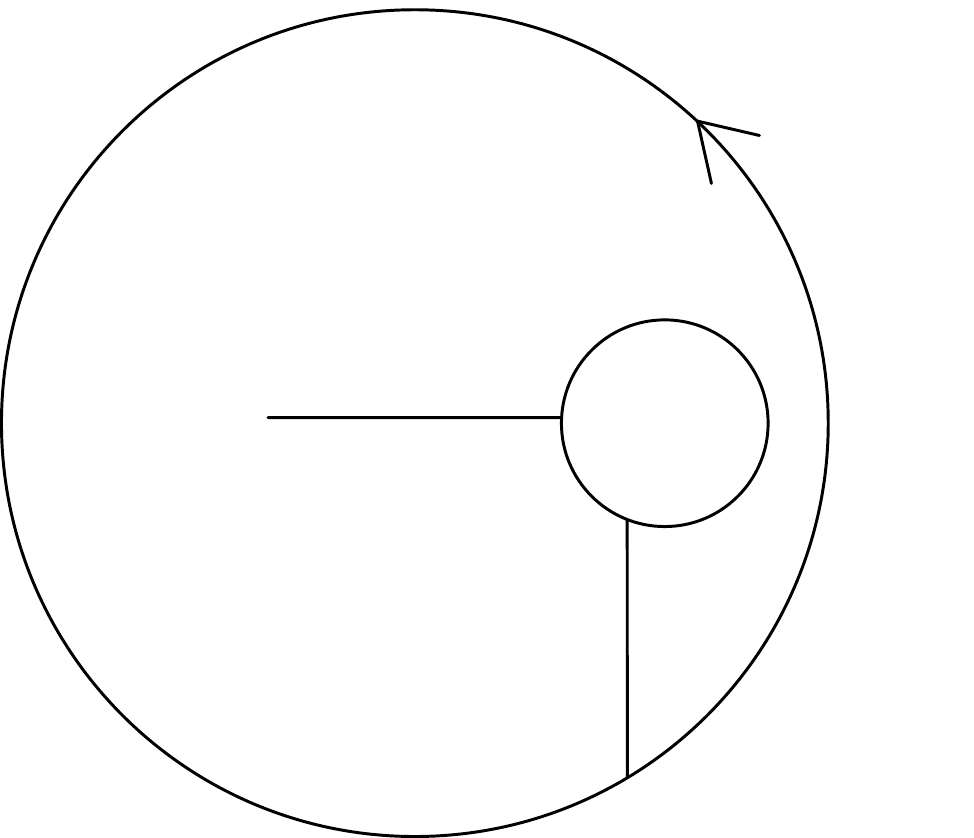}}%
    \put(0.58668659,0.40590075){\makebox(0,0)[lt]{\lineheight{0}\smash{\begin{tabular}[t]{l}$c_+$\end{tabular}}}}%
    \put(0.40075769,0.81143516){\makebox(0,0)[lt]{\lineheight{0}\smash{\begin{tabular}[t]{l}$c$\end{tabular}}}}%
    \put(0.37470415,0.45259465){\makebox(0,0)[lt]{\lineheight{0}\smash{\begin{tabular}[t]{l}$\lambda_2$\end{tabular}}}}%
    \put(0.65910587,0.17982406){\makebox(0,0)[lt]{\lineheight{0}\smash{\begin{tabular}[t]{l}$\lambda_1$\end{tabular}}}}%
    \put(0,0){\includegraphics[width=\unitlength,page=2]{paper.pdf}}%
    \put(0.07458745,0.40600239){\makebox(0,0)[lt]{\lineheight{0}\smash{\begin{tabular}[t]{l}$c_-$\end{tabular}}}}%
    \put(0,0){\includegraphics[width=\unitlength,page=3]{paper.pdf}}%
    \put(0.55695247,0.32778531){\makebox(0,0)[lt]{\lineheight{0}\smash{\begin{tabular}[t]{l}$\mu$\end{tabular}}}}%
    \put(0.74760611,0.52755367){\makebox(0,0)[lt]{\lineheight{0}\smash{\begin{tabular}[t]{l}$\nu$\end{tabular}}}}%
    \put(0,0){\includegraphics[width=\unitlength,page=4]{paper.pdf}}%
  \end{picture}%
\endgroup%

         \caption{$P^{(1)}$}
         \label{fig1}
    \end{figure}Orient the circle \(c\) counter-clockwise, and the circles \(c_+\) and \(c_-\) clockwise. Also, orient \(\lambda_1\) and \(\lambda_2\) so that their starting points lie on \(c_+\) (see \Cref{fig1}). Now, define \(\varphi\) on \(P^{(1)}\) as follows: map \(c \cup \lambda_1\) onto \(c \cup \lambda_1\) by the identity map. Next, map \(c_+\) onto \(c_+\) by an orientation-preserving two-fold covering. Then, map \(c_-\) onto \(c_+\) by an orientation-reversing three-fold covering. We further assume that \(\varphi(\lambda_1 \cap c_+) = \varphi(\lambda_2 \cap c_+) = \varphi(\lambda_2 \cap c_-) = \lambda_1 \cap c_+\), and that \(\varphi\) restricted to either component of \(c_+ \setminus (\lambda_1 \cup \lambda_2)\) is a homeomorphism onto \(c_+ \setminus \lambda_1\). Finally, map the entire \(\lambda_2\) to \(\lambda_1 \cap c_+\). Let \(\gamma \colon \mathbb{S}^1 \to P^{(1)}\) be the loop described by \(\lambda_1 *  c * \overline{\lambda_1} * \mu * \lambda_2 * c_- * \overline{\lambda_2} * \nu\), where \(\mu\) and \(\nu\) are the arcs of \(c_+\) determined by the points \(\lambda_1 \cap c_+\) and \(\lambda_2 \cap c_+\). Since \(\varphi \circ \gamma\) is null-homotopic, \(\varphi\) can be extended to a map from \(P \cong P^{(1)} \cup_\gamma \mathbb{D}^2\) to \(A'\). Thus, we obtain a map \( \varphi \colon P \to A' \) such that \( \varphi\vert c \to c \) is a homeomorphism, \( \varphi\vert c_+ \to c_+ \) is a two-fold covering, and \( \varphi\vert c_- \to c_+ \) is a three-fold covering. Using an external collar, we may further assume that \( \varphi^{-1}(\partial A') = \partial P \). An argument similar to that in \Cref{ex} then shows that \(\ker \pi_1(\varphi) \neq 0\), but \( \varphi \) has no geometric kernel.

    Let \( M \) be a non-compact bordered surface with a single boundary component, which may be of infinite type, non-planar, or both, and let \( D_* \) denote a punctured disk. Attach two copies of \( D_* \) and one copy of \( M \) to \( P \) by identifying a copy of \( \partial D_* \) with each of \( c_+ \) and \( c_- \), and identifying a copy of \( \partial M \) with \( c \). Let \( S' \) denote the resulting non-compact surface. Similarly, define \( S \) to be the non-compact surface obtained by attaching one copy of \( M \) and one copy of \( D_* \) to \( A' \), by identifying a copy of \( \partial M \) with \( c_+ \) and a copy of \( \partial D_* \) with \( c \). Then the map \( \varphi\colon P \to A' \) extends to a degree one map \( f\colon S' \to S \) such that \( f \) restricts to a homeomorphism from \( M = f^{-1}(M) \subset S' \) onto \( M \subset S \), and to a two- or three-sheeted covering map from each copy of \( D_* \subset S' \) onto the unique copy of \( D_* \subset S \). As usual, \( \ker \pi_1(f)\neq 0 \). 

   We claim that \( f \) has no geometric kernel. On the contrary, assume that \( \gamma' \) is a non-trivial simple loop in \( S' \) such that \( f(\gamma') \) is null-homotopic. Depending on whether \( \gamma' \) separates \( S' \) or not, we consider two cases. 

If \( \gamma' \) separates \( S' \), then an argument using the naturality of the homology long exact sequence, similar to that given in \Cref{ex}, shows that \( \gamma' \) bounds an essential compact bordered subsurface \( S_{\gamma'}' \subset S' \) with \( \partial S_{\gamma'}' = \gamma' \). If the genus of \( S_{\gamma'}' \) is zero, then \( \gamma' \) must be contractible—a contradiction. On the other hand, if the genus of \( S_{\gamma'}' \) is positive, then \( \gamma' \) can be freely homotoped into \( M \), and hence remains non-trivial when projected to \( S \) by \( f \), since \( f\vert f^{-1}(M) \to M \) is a homeomorphism—again a contradiction.

Now consider the remaining case, namely that \(\gamma'\) does not separate \(S'\). Since \(\gamma'\) is simple, we may assume it is disjoint from both copies of \(D_*\) in \(S'\). Choose an essential compact bordered subsurface \(N \) of \(M \subset S\) such that \(\partial M\) is a component of \(\partial N\), and \(\gamma'\) lies in the interior of the essential compact bordered subsurface \(P \cup f^{-1}(N) \subset S'\). Recall that \(f\) restricts to a homeomorphism from \(f^{-1}(N)\) onto \(N\). Let \(\gamma' \times [-1,1]\) be a tubular neighborhood of \(\gamma'\) in the interior of \(P \cup f^{-1}(N)\), with \(\gamma' \times \{0\} \equiv \gamma'\). Define \(X'\) to be the surface obtained from \(P \cup f^{-1}(N)\) by removing \(\gamma' \times (-1,1)\) and gluing a disk along each boundary component \(\gamma' \times \{i\}\) for \(i = \pm 1\). Since \(\gamma'\) does not separate \(S'\), the resulting surface \(X'\) is connected. The restriction of \(f\) to \((P \cup f^{-1}(N)) \setminus \gamma' \times (-1,1)\) extends to a map \(\overline{f} \colon X' \to A' \cup N\), since \(f(\gamma')\) is null-homotopic. Moreover, \(\deg(\overline{f}) = \pm 1\), as \(\overline{f}\) maps \(f^{-1}(\partial N \setminus \partial A') = \partial X' \cap \overline{f}^{-1}(\partial N \setminus \partial A')\) homeomorphically onto \(\partial N \setminus \partial A'\).

Applying the inclusion--exclusion formula for Euler characteristic, we have $\chi(X') = \chi((P \cup f^{-1}(N)) \setminus \gamma' \times (-1,1)) + 2 = \chi(P \cup f^{-1}(N)) + 2.$ Since \(\chi(P \cup f^{-1}(N)) = \chi(P) + \chi(f^{-1}(N)) - \chi(P \cap f^{-1}(N)) = -1 + \chi(N)\), it follows that \(\chi(X') = 1 + \chi(N)\). By \Cref{tool}, $\chi(X') \leq |\deg(\overline{f})| \cdot \chi(A' \cup N).$ This gives \(1 + \chi(N) \leq\chi(A'\cup N)= \chi(A') + \chi(N) - \chi(A' \cap N)=\chi(N)\), a contradiction.

Thus, in both cases, we conclude that \(f\) has no geometric kernel. Moreover, observe that \( \underline{\pi}_0(f) \) is not injective, since \( f^{-1}(D_*) \) is the union of two disjoint properly embedded punctured disks in \( S' \). On the other hand, \( \underline{\pi}_1(f) \) is injective, as for each component \( V' \) of \( S' \setminus P \), the restriction \( f| V' \to f(V') \) is \( \pi_1 \)-injective.
 \hfill$\square$
\end{example}

The final example of this section demonstrates the necessity of the hypothesis in \Cref{finitepinch} that the proper map must have degree one. We continue to follow the notation established in \Cref{ex}.

\begin{example}\label{LNM}Let \( S \) denote the Loch Ness monster surface, i.e., the unique infinite genus surface with exactly one end. There exists a proper self-map \( f\colon S \to S \) of degree $2$ such that \( \ker \pi_1(f) \neq 0 \), yet \( f \) has no geometric kernel. To construct such an example, we modify Tucker’s counterexample once more. This time, we attach handles \( h_+ \) and \( h_- \) to \( P \) along \( c_+ \) and \( c_- \), respectively, and a handle \( h \) to \( A \) along \( q(c_+) = q(c_-) \). We then extend the map \( q \) by sending each of \( h_+ \) and \( h_- \) homeomorphically onto \( h \). This yields a two-fold branched covering \( q'\colon S_{2,1} \to S_{1,1} \), with branch point \( q'(0,0) = q(0,0) \). In particular, the restriction \( q'\vert\partial S_{2,1} \to \partial S_{1,1} \) is a two-fold covering map. 

Let \( \gamma_- \) (resp.\ \( \gamma_+ \)) be a circle in \( P \), based at \( (0,0) \), such that \( \gamma_- \) (resp.\ \( \gamma_+ \)) bounds an annulus in \( P \) with \( c_- \) (resp.\ \( c_+ \)). Assume that \( \gamma_+ \cap \gamma_- = \{(0,0)\} \), and that \( \gamma_- \) is the reflection of \( \gamma_+ \) across the \( \mathrm Y \)-axis. Orient both \( \gamma_- \) and \( \gamma_+ \) counterclockwise. Note that \( \gamma \coloneqq \gamma_- * \gamma_+ \) is a non-trivial loop in \( P \subset S_{2,1} \), while its image \( q'(\gamma) = q(\gamma) \) is null-homotopic. Thus, \( \ker \pi_1(q') \neq 0 \). 

We now extend \(q'\) to a map \(f\colon S \to S\). Consider the surfaces obtained by attaching a punctured disk along \(\partial S_{2,1}\) and another along \(\partial S_{1,1}\). This gives a map \(Q'\colon S_{2,0,1} \to S_{1,0,1}\) such that \(Q'\vert S_{2,1} = q'\), and the restriction of \(Q'\) to \(\operatorname{cl}(S_{2,0,1} \setminus S_{2,1})\) is a two-fold covering map onto \(\operatorname{cl}(S_{1,0,1} \setminus S_{1,1})\). Choose a sequence \(\{D_n\}\) of disks in \(S_{1,0,1} \setminus S_{1,1}\) converging to the end of \(S_{1,0,1}\), such that \(Q'\vert Q'^{-1}(D_n) \to D_n\) is a two-fold covering for each \(n\). Then \(Q'^{-1}(D_n)\) is a disjoint union of two disks in \(S_{2,0,1} \setminus S_{2,1}\), and \(Q'\) maps each component of \(Q'^{-1}(D_n)\) homeomorphically onto \(D_n\). Attach handles along each boundary component of \(S_{2,0,1} \setminus \bigcup_n \operatorname{int}(Q'^{-1}(D_n))\), and similarly along each boundary component of \(S_{1,0,1} \setminus \bigcup_n \operatorname{int}(D_n)\), and extend \(Q'\) homeomorphically over these handles to obtain a map \(f\colon S \to S\). Note that \(f\vert f^{-1}(S_{1,1}) = q'\colon S_{2,1} \to S_{1,1}\), and the restriction of \(f\) to \(\operatorname{cl}(S \setminus S_{2,1})\) is a two-fold covering onto \(\operatorname{cl}(S \setminus S_{1,1})\). Thus, \(\deg(f) = \pm 2\) and \(\ker \pi_1(f) \neq 0\). 

Finally, an argument similar to that in \Cref{exnext} shows that \(f\) has no geometric kernel. To prove this by contradiction, we consider two cases: one where a simple closed curve representing a geometric kernel separates \(S\), and one where it does not separate \(S\). In both cases, we apply \Cref{tool} to reach a contradiction, unlike in \Cref{exnext}, where \Cref{tool} are applied only in the non-separating case.

Observe that \( \underline{\pi}_0(f) \) is injective, since \( S \) has exactly one end. Moreover, \( \underline{\pi}_1(f) \) is injective because, for any component \( V' \) of \( S' \setminus P \), the restriction \( f| V' \to f(V') \) is \( \pi_1 \)-injective. \hfill$\square$
\end{example}

\begin{remark}
The author is not aware of any example of a non-$\pi_1$-injective self-map of the Loch Ness monster surface $S$ of degree one that lacks a geometric kernel. However, a related hypothesis may be proposed: there exists a non-$\pi_1$-injective, non-$\pi_1$-surjective self-map $f$ of $S$ of prime degree that also lacks a geometric kernel. If this hypothesis holds, such an example could then be constructed, since, in that case, any lift of $f$ with respect to the covering corresponding to the subgroup $\operatorname{im} \pi_1(f) \subseteq \pi_1(S)$ would be a non-$\pi_1$-injective self-map of $S$ of degree one that lacks a geometric kernel, by \cite[Corollary 3.4 and Proof of Theorem 3.1]{MR192475}. The truth of this hypothesis remains unknown to the author.
\end{remark}

\section{Cut-off Technique for Detecting Geometric Kernel}\label{S4}The main goal of this section is to analyze the behavior of a \(\underline{\pi}_1\)-injective proper map between non-compact surfaces by using the classification of \(\pi_1\)-injective proper maps. As an application, we then give a sufficient condition under which a degree-one map between non-compact surfaces has a geometric kernel.

\subsection{\texorpdfstring{$\pi_1$-injective proper maps}{π₁-injective proper maps}}Let \( M \) and \( N \) be two connected, non-compact $2$-manifolds. The aim of this section is to provide a proper homotopy classification of all \(\pi_1\)-injective proper maps from \(M\) to \(N\), subject to certain restrictions on such maps or on the spaces \(M\) and \(N\). This classification will be used extensively in the next section, in particular to prove \Cref{cut}. 

When \( \partial M = \varnothing = \partial N \), every \(\pi_1\)-injective proper map from \( M \) to \( N \) is properly homotopic to a finite-sheeted covering map, provided \( M \) is neither the plane nor the punctured plane; see \cite{Das2024pi1injective}. However, when \( \partial N \neq \varnothing \), an analogous classification requires the additional assumption that every component of \( \partial N \) is compact. For instance, let \( E \) be any compact, totally disconnected subset of \( (-\infty, 2] \times (0, \infty) \), and consider the map \( \varphi \) from \( S' \coloneqq \{(x,y) \in \mathbb{R}^2 \colon y \geq 0\} \setminus (E \cup \{(-1/2, 0), (1/2, 0), (3/2, 0)\}) \) to \( S \coloneqq \{(x,y) \in \mathbb{R}^2 \colon y \geq 0\} \setminus (E \cup \{(-1/2, 0)\}) \), defined by \( \varphi(x,y) \coloneqq (x,y) \) if \( x \leq 0 \), \( \varphi(x,y) \coloneqq (-x,y) \) if \( 0 \leq x < 1 \), and \( \varphi(x,y) \coloneqq (x - 2, y) \) if \( x \geq 1 \). Then \( \varphi \) is a \(\pi_1\)-bijective proper map between two non-homeomorphic bordered surfaces, with \( \varphi(\partial S') \subseteq \partial S \). On the other hand, the restriction \( \psi \coloneqq \varphi\vert S' \setminus (\{(-3,0)\}\cup\{(-3 + \frac 1n, 0) \colon n \in \mathbb{N}\}) \to S \setminus (\{(-3,0)\}\cup\{(-3 + \frac 1n, 0) \colon n \in \mathbb{N}\}) \) is a \(\pi_1\)-bijective proper map between two homeomorphic\footnote{We may apply \cite[Theorem~2.2]{MR542887} to show that the domain and codomain of \(\psi\) are homeomorphic. To this end, observe that the map \( -\infty \mapsto -\infty \), \( -3 \mapsto -3 \), \( -3 + \tfrac{1}{n} \mapsto -3 + \tfrac{1}{n+2} \) for all \( n \in \mathbb{N} \), \( -\tfrac{1}{2} \mapsto -3 + \tfrac{1}{2} \), \( \tfrac{1}{2} \mapsto -3 + 1 \), \( \tfrac{3}{2} \mapsto -\tfrac{1}{2} \), and \( +\infty \mapsto +\infty \) induces an isomorphism from the diagram of the domain of \(\psi\) onto the diagram of the codomain of \(\psi\).
} bordered surfaces that sends the boundary into the boundary, but still \( \psi \) is not properly homotopic to a homeomorphism, since \( \underline{\pi}_0(\psi) \) sends each of the ends \( (-1/2, 0) \), \( (1/2, 0) \), and \( (3/2, 0) \) to the same end \( (-1/2, 0) \). Accordingly, we henceforth assume that every component of \(\partial N\) is compact.

Moreover, following the theory of compact bordered surfaces \cite[Theorem 2.1]{MR2619608}, we restrict our attention to those \(\pi_1\)-injective proper maps \(f \colon M \to N\) that send \(\partial M\) into \(\partial N\). To illustrate the importance of controlling the behavior of a \(\pi_1\)-injective proper map on the boundary, consider the following situation. Suppose \(M'\) is a non-compact bordered surface such that every component of \(\partial M'\) is compact and \(M'\) has no planar end, and let \(f' \colon M' \to M'\) be a \(\pi_1\)-bijective proper self-map. If we want \(f'\) to be properly homotopic to a homeomorphism, then it is necessary to assume that for each loop \(\gamma' \subset \partial M'\), the image \(f'(\gamma')\) is freely homotopic to a loop in \(\partial M'\). This is because of the following two facts (where \(i\) denotes the geometric intersection number and \([\cdot]\) denotes free homotopy class): (1) if \(\alpha'\) is a loop in \(M'\) such that \(i([\alpha'], [\beta']) = 0\) for every loop \(\beta'\) in \(M'\), then \(\alpha'\) is freely homotopic to a loop in \(\partial M'\); (2) if \(h' \colon M' \to M'\) is homotopic to a homeomorphism, then for any two closed curves \(\alpha'\) and \(\beta'\) in \(M'\), we have \(i([h'(\alpha')], [h'(\beta')]) = i([\alpha'], [\beta'])\). Thus, combining the above observations, we propose the following theorem.

%This additional assumption suffices for our purposes—namely, proving \Cref{classify}, which in turn plays a key role in the proof of \Cref{cutoff}.

\begin{theorem}\label{classification}
    Let $M$ and $N$ be non-compact bordered surfaces such that boundary components of $N$ are compact. Suppose there is a proper map $f\colon (M, \partial M) \to (N, \partial N)$ so that $\pi_1(f)$ is a monomorphism. Then, there is a proper homotopy $H\colon (M\times[0,1],\partial M\times [0,1])\to (N,\partial N)$ from $f$ to a finite-sheeted covering map. 
\end{theorem}

Now, there are two potential approaches to proving \Cref{classification}. The first is to modify the proof of the analogous result in the boundaryless case \cite{Das2024pi1injective}, and the second is to adapt the proof of the corresponding statement for irreducible, boundary-irreducible, end-irreducible \(3\)-manifolds \cite[Theorem 4.2]{MR334225}. The former approach proceeds in two steps, described as follows. Let \(f\colon S' \to S\) be a \(\pi_1\)-injective proper map between two surfaces, with \(S'\) neither the plane nor the punctured plane. The first step is to assume that \(\pi_1(f)\) is an isomorphism and aim to show that \(f\) is properly homotopic to a homeomorphism. The second step handles the general case—where \(\pi_1(f)\) is merely a monomorphism—by lifting \(f\) to the covering space corresponding to the subgroup \(\operatorname{im}(\pi_1(f)) \subseteq \pi_1(S)\). Both steps crucially rely on establishing that either \(f\) or its lift has nonzero degree.

The latter approach, due to Brown and Tucker, is more direct in that it avoids reducing to the \(\pi_1\)-bijective case before addressing the general \(\pi_1\)-injective case, and it does not require showing that the given \(\pi_1\)-injective proper map has nonzero degree. To prove \Cref{classification}, we follow this approach and adapt it to dimension two. However, this method relies heavily on the end-irreducibility of both the domain and the codomain. A connected \(3\)-manifold \(P\) is said to be \emph{end-irreducible} if \(P \ne \mathbb{R}^3\) and, for each end \(e\) of \(P\), there exists a representative \(a \in e\) such that the inclusion-induced homomorphism \(\underline{\pi}_1(P, \underline{a}) \to \pi_1(P, \underline{a})\) is injective, where \(\pi_1(P, \underline{a})\) denotes the \emph{repeated fundamental group} of the end \(e\) based at \(\underline{a}\). The repeated fundamental group is defined in the same way as the proper fundamental group, except that the term “proper” is omitted from all relevant notions—namely, germ of a proper map, proper maps of pairs, and germ homotopy rel \(\underline{*}\) between proper maps of pairs; see \cite[p.~109]{MR334225}. To clarify these technical terms involved in the definition of end-irreducibility, consider the following equivalent statements about a connected manifold \(X\) (not necessarily $3$-dimensional); see \cite[Proposition~1.5]{MR2621405}: (1) for each end \(e\) of \(X\), there exists a representative \(a \in e\) such that the inclusion-induced homomorphism \(\underline{\pi}_1(X, \underline{a}) \to \pi_1(X, \underline{a})\) is injective; and (2) for every compact subset \(K \subseteq X\), there exists a compact subset \(K' \subseteq X\) with \(K \subset \operatorname{int}(K')\) such that any loop \(\ell \colon \mathbb{S}^1 \to X \setminus K'\) that is null-homotopic in \(X\) is also null-homotopic in \(X \setminus K\).

Now, the role of end-irreducibility is as follows. Suppose \(P\) and \(Q\) are connected, irreducible, boundary-irreducible, end-irreducible \(3\)-manifolds such that each component of \(\partial Q \sqcup \partial P\) is compact. Brown and Tucker showed that the end-irreducibility of \(P\) yields an exhausting sequence \(\{C_n\}\) for \(P\), satisfying the following conditions for each \(n\): \(C_n\) is a compact \(3\)-dimensional submanifold of \(P\); \(\operatorname{fr}(C_n) \subseteq P \setminus \partial P\); each component of \(\operatorname{fr}(C_n)\) is an incompressible surface in \(P\); \(C_n\) is connected; and each component of \(P \setminus C_n\) is unbounded \cite[Lemma~3.1]{MR334225}. Therefore, if \( f \colon (Q, \partial Q) \to (P, \partial P) \) is a proper map (not necessarily \(\pi_1\)-injective), then \( f \) can be properly homotoped through maps \((Q, \partial Q) \to (P, \partial P)\) so that, possibly after passing to a subsequence of \(\{C_n\}\), the sequence \(\{D_n \coloneqq f^{-1}(C_n)\}\) forms an exhausting sequence for \(Q\), satisfying all the properties of \(\{C_n\}\) except possibly the last two; see \cite[Proof of Theorem~4.2]{MR334225}. This conclusion follows from applying a theorem of Heil \cite[Corollary~3.2]{MR2619608} at each stage \(n\) to ``simplify'' the preimage of \(\operatorname{fr}(C_n)\), noting that each component of \(\operatorname{fr}(C_n)\) is incompressible. Thus, if we further assume that \(f\) is \(\pi_1\)-injective, Waldhausen’s theory \cite[Theorem 6.1]{MR224099} can be applied to each restriction \(f|(D_n, \partial D_n) \to (C_n, \partial C_n)\).
Now, to adapt the Brown--Tucker theory to dimension two—specifically to construct such an exhausting sequence—end-irreducibility is unnecessary.
\begin{theorem}\label{exhaustion}
    Let \( M \) be a connected, non-compact \( 2 \)-manifold such that each component of \(\partial M\), if any, is compact. Suppose \( M \neq \mathbb{R}^2 \). Then there exists an exhausting sequence \(\{C_n: n\geq 1\}\) for \( M \) such that, for each \(n\), 
\begin{enumerate}[(i)]
    \item\label{(i)} $C_n$ is a compact $2$-dimensional submanifold of $M$,
    \item\label{(ii)} each component of $\operatorname{fr}(C_n)$ is a circle in $M\setminus \partial M$,
    \item\label{(iii)} each component of $\operatorname{fr}(C_n)$ does not bound a disk in $M$,
    \item\label{(iv)} $C_n$ is connected,
    \item\label{(v)} all components of $M \setminus C_n$ are unbounded.
\end{enumerate}
\end{theorem}
\begin{proof}
  First, assume that \(\partial M = \varnothing\). Then Goldman’s inductive procedure \cite[Chapter 8]{MR2616858} yields an exhausting sequence \(C_0 \subset C_1 \subset C_2 \subset \cdots\) of compact bordered subsurfaces of \(M\) such that \(C_0\) is a disk; \(\operatorname{cl}(C_{k+1} \setminus C_k)\) is, for each \(k \geq 0\), either an annulus, a pair of pants, or a two-holed torus; and \(\operatorname{cl}(C_{k+1} \setminus C_k) \cap C_k\) is a single circle for each \(k \geq 0\). Since \(M \neq \mathbb{R}^2\), not all \(\operatorname{cl}(C_{k+1} \setminus C_k)\) can be annuli. Therefore, after discarding the first few terms and re-indexing the sequence, we may assume that \(\operatorname{cl}(C_1 \setminus C_0)\) is either a pair of pants or a two-holed torus. The sequence \(\{C_n : n \geq 1\}\) now satisfies all five conditions \ref{(i)}--\ref{(v)}.

  From now on, we consider the other case; namely, we assume that \(\partial M \neq \varnothing\). By hypothesis, every component of \(\partial M\) is a circle. An end \(e \in \underline{\pi}_0(M)\) is said to be a \emph{rim point of \(M\)} if, for every compact subset \(K \subset M\), the unbounded component \(A\) of \(M \setminus K\), corresponding to \(e\) (i.e., \(e \in \underline{A}\)), contains infinitely many components of \(\partial M\). Denote the set of all rim points of \(M\) by \(\operatorname{rim}(M)\). Let \(\widehat{M}\) be the surface obtained by gluing a disk along each component of \(\partial M\). Then \(M\) is a properly embedded subsurface of \(\widehat M\), and each component of \(\partial M\) bounds a disk in \(\widehat M\). Moreover, if \(i \colon M \hookrightarrow \widehat M\) denotes the inclusion, then the induced map \(\underline\pi_0(i) \colon (\underline\pi_0(M), \underline\pi_0^\mathrm{np}(M)) \to (\underline\pi_0(\widehat M), \underline\pi_0^\mathrm{np}(\widehat M))\) is a homeomorphism. Let \(\{\widehat{C}_n: n\geq 1\}\) be an exhausting sequence for the surface \(\widehat{M}\) such that if \(\widehat{M} \neq \mathbb{R}^2\), then \(\{\widehat{C}_n: n\geq 1\}\) is constructed similarly to the sequence described in the previous paragraph, and if \(\widehat{M} = \mathbb{R}^2\), then each \(\widehat{C}_n\) is a disk. In particular, if \(\widehat{M} \neq \mathbb{R}^2\), then for each \(n \geq 1\), \(\operatorname{cl}(\widehat{C}_{n+1} \setminus \widehat{C}_n)\) is either an annulus, a pair of pants, or a two-holed torus, and \(\widehat{C}_1\) is either an annulus or a handle. It may happen that \(\partial M \cap \bigcup_n \operatorname{fr}(\widehat{C}_n) \neq \varnothing\). For this reason, we aim to construct a model \(M'\) for \(M\) inside \(\widehat{M}\) such that \(\partial M'\) does not intersect \(\bigcup_n \operatorname{fr}(\widehat{C}_n)\). Let \(\{h_l : l \in \mathscr{A}\}\) be a pairwise disjoint collection of essentially embedded handles in \(\widehat{M}\) such that \(\widehat{M} \setminus \bigcup_{l \in \mathscr{A}} h_l\) is embeddable in \(\mathbb{R}^2\). Without loss of generality, we may assume that each \(h_l\) is contained either in \(\operatorname{int}(\operatorname{cl}(\widehat{C}_{n+1} \setminus \widehat{C}_n))\) for some \(n\), or in \(\operatorname{int}(\widehat{C}_1)\). Pick a pairwise disjoint, locally finite collection \(\{\widehat{D}_j : j \in \mathscr{B}\}\) of embedded disks in \(\widehat{M}\) such that the following hold: the cardinality of \(\mathscr{B}\) equals the number of components of \(\partial M\); \(\bigcup_{j \in \mathscr{B}} \widehat{D}_j\) is disjoint from \(\bigcup_{l \in \mathscr{A}} h_l\); and a subsequence of \(\{\widehat{D}_j : j \in \mathscr{B}\}\) converges to an end \(e \in \underline{\pi}_0(\widehat{M})\) if and only if \(e \in \underline{\pi}_0(i)(\operatorname{rim}(M))\). For a similar construction, see \cite[Corollary A.8]{MR561583}. Moreover, we may assume that each \(\widehat{D}_j\) is contained in \(\operatorname{int}(\operatorname{cl}(\widehat{C}_{n+1} \setminus \widehat{C}_n))\) for some \(n\). In particular, the sets \(\bigcup_n \operatorname{fr}(\widehat{C}_n)\), \(\bigcup_{l \in \mathscr{A}} h_l\), and \(\bigcup_{j \in \mathscr{B}} \widehat{D}_j\) are three pairwise disjoint subsets of \(\widehat{M}\). Define \(M' \coloneqq \widehat{M} \setminus \operatorname{int}(\bigcup_{j \in \mathscr{B}} \widehat{D}_j)\). Thus, \(M'\) is a properly embedded subsurface of \(M\). Furthermore, if \(i' \colon M' \hookrightarrow M\) denotes the inclusion, then the induced map \(\underline{\pi}_0(i') \colon (\underline{\pi}_0(M'), \underline{\pi}_0^\mathrm{np}(M')) \to (\underline{\pi}_0(\widehat{M}), \underline{\pi}_0^\mathrm{np}(\widehat{M}))\) is a homeomorphism such that \(\underline{\pi}_0(i)(\operatorname{rim}(M)) = \underline{\pi}_0(i')(\operatorname{rim}(M'))\). In particular, we have a homeomorphism from \(\underline{\pi}_0(M)\) onto \(\underline{\pi}_0(M')\) that sends \(\underline{\pi}_0^\mathrm{np}(M)\) and \(\operatorname{rim}(M)\) onto \(\underline{\pi}_0^\mathrm{np}(M')\) and \(\operatorname{rim}(M')\), respectively. By \cite[Theorem A.7]{MR561583}, there exists a homeomorphism \(\varphi\) from \(M'\) onto \(M\). The required exhausting sequence for \(M\) is then given by \(\{\varphi(\widehat{C}_n \cap M') : n \geq 1\}\).
\end{proof}

As an application, we show that every connected \(2\)-manifold \(M\) satisfying the hypotheses of \Cref{exhaustion} is end-irreducible. %Note, however, that none of our subsequent results in dimension two require end-irreducibility.

\begin{proposition}\label{endirr}
     Let \( M \) be a connected, non-compact \( 2 \)-manifold such that each component of \(\partial M\), if any, is compact. Suppose \( M \neq \mathbb{R}^2 \). Then for each end \(e\) of \(M\), there exists a representative \(a \in e\) such that the inclusion-induced homomorphism \(\underline{\pi}_1(M, \underline{a}) \to \pi_1(M, \underline{a})\) is injective.
\end{proposition}
   
\begin{proof}
    By \cite[Proposition~1.5]{MR2621405}, it is enough to show that for every compact subset \(K \subseteq M\), there exists a compact subset \(K' \subseteq M\) with \(K \subset \operatorname{int}(K')\) such that any loop \(\ell \colon \mathbb{S}^1 \to M \setminus K'\) that is null-homotopic in \(M\) is also null-homotopic in \(M \setminus K\). So, let \(K\) be a compact subset of \(M\), and pick an exhausting sequence \(\{C_n\}\) for \(M\) satisfying conditions \ref{(i)}--\ref{(v)} of \Cref{exhaustion}. Choose \(n\) large enough so that \(K \subset \operatorname{int}(C_n)\). Suppose \(\gamma \colon \mathbb{S}^1 \to M \setminus C_n\) is a loop that extends to a map \(\widetilde{\gamma} \colon \mathbb{D}^2 \to M\), where \(\mathbb{D}^2 \coloneqq \{z \in \mathbb{C} : |z| \leq 1\}\). Choose some large \(m > n\) such that \(\widetilde{\gamma}(\mathbb{D}^2) \subseteq C_m\). Since each component of \(\operatorname{fr}(C_n)\) is a non-trivial circle in \(\operatorname{int}(C_m)\), and \(\widetilde{\gamma}(\partial \mathbb{D}^2) \cap \operatorname{fr}(C_n) = \varnothing\), the transversality homotopy theorem in dimension two \cite[Lemma~2.2]{MR2619608} implies that \(\widetilde{\gamma} \colon \mathbb{D}^2 \to C_m\) can be homotoped rel\ \(\partial \mathbb{D}^2\) to a map \(\Gamma \colon \mathbb{D}^2 \to C_m\) such that \(\Gamma^{-1}(\operatorname{fr}(C_n)) = \varnothing\). Thus, the connected set \(\Gamma(\mathbb{D}^2)\) must be contained entirely in exactly one of the following: \(\operatorname{int}(C_n)\) or \(C_m \setminus C_n\). The former is not possible, since \(\Gamma(\partial \mathbb{D}^2) = \gamma(\mathbb{S}^1) \subseteq M \setminus C_n\). It follows that \(\gamma \colon \mathbb{S}^1 \to M \setminus C_n\) extends to a map \(\mathbb{D}^2 \to C_m \setminus C_n\). Since \(C_m \setminus C_n \subset M \setminus K\), we are done.
\end{proof}
By applying the transversality homotopy theorem in dimension two \cite[Lemma~2.2]{MR2619608} in a similar manner, one can prove the following result.
\begin{lemma}\label{rhofunctor}
    Assume \( M \) satisfies the conditions of \Cref{endirr}. Let \(B\) be a properly embedded subsurface of \(M\) such that \(\operatorname{fr}(B)\) is a finite, pairwise disjoint collection of non-trivial circles in \(M\setminus \partial M\). Then the inclusion \(B \hookrightarrow M\) is \(\pi_1\)-injective.
\end{lemma}
\begin{remark} 
     One may use Brown’s \( \wp \)-functor \cite[§2]{MR356041} to give an alternative proof of \Cref{endirr}. Here is an outline of this alternative argument. Let \(\{C_n \colon n \ge 0\}\) be an exhausting sequence for \(M\) satisfying properties \ref{(i)}--\ref{(v)} of \Cref{exhaustion}. For each \(n \ge 0\), let \(B_n\) denote the component of \(M \setminus C_n\) such that \(e \in \underline{B_n}\). Choose a proper map \(b \colon [0, \infty) \to M\) representing \(e\) such that \(b([n, \infty)) \subseteq B_n\) for each \(n \ge 0\). Then, by \Cref{rhofunctor}, the inclusion-induced homomorphism \(\pi_1(B_n, b(n)) \to \pi_1(M, b(n))\) is injective for every \(n \ge 0\). Consider the inverse sequence of groups \(\{\pi_1(B_n, b(n)) \colon n \ge 0\}\) determined by the inclusion maps followed by change of base point along the path \(b\vert[n, n+1]\). Similarly, we have the inverse sequence of groups \(\{\pi_1(M, b(n)) \colon n \ge 0\}\). Applying \(\wp\)-functor, we obtain a monomorphism \(\wp(\{\pi_1(B_n, b(n))\}) \to \wp(\{\pi_1(M, b(n))\})\). There are natural isomorphisms \(\wp(\{\pi_1(B_n, b(n))\}) \cong \underline{\pi}_1(M, \underline{b})\) \cite[p.~45]{MR356041} and \(\wp(\{\pi_1(M, b(n))\}) \cong \pi_1(M, \underline{b})\) \cite[Proof of Lemma~2.1]{MR511426}. It follows that the inclusion-induced homomorphism \(\underline{\pi}_1(M, \underline{b}) \to \pi_1(M, \underline{b})\) is injective. Hence, the same holds for \(\underline{\pi}_1(M, \underline{a}) \to \pi_1(M, \underline{a})\) for any representative \(a\) of \(e\).
\end{remark}
We are now ready to prove \Cref{classification}. To this end, we first reduce \Cref{classification} to a special case (see \Cref{specialcase}), namely when \(f\vert\partial M \to \partial N\) is a local homeomorphism. We then handle the general case, where no assumption is made on the map \(f\vert\partial M \to \partial N\). The following lemmas will be used to establish the special case.

\begin{lemma}\label{aux}
    Let \( \varphi \) and \( \psi \) be two homotopic self-maps of \( \mathbb{S}^1 \). Then \( \varphi \times \operatorname{Id}_{[-1,1]} \) is homotopic rel \( \mathbb{S}^1 \times \{\pm 1\} \) to a self-map \( \Phi \) of \( \mathbb{S}^1 \times [-1,1] \) such that \( \Phi^{-1}(\mathbb{S}^1 \times [-1/2,1/2]) = \mathbb{S}^1 \times [-1/2,1/2] \) and \( \Phi(z,t) = (\psi(z), t) \) for all \( (z,t) \in \mathbb{S}^1 \times [-1/2,1/2] \).
\end{lemma}
\begin{proof}
    Let \(a\in (1/2,1)\). By the homotopy extension theorem, there exist a self-map \(\Phi_l\) of \(\mathbb{S}^1 \times [-1, -a]\) homotopic rel \(\mathbb{S}^1\times\{-1\}\) to \(\varphi\times\operatorname{Id}_{[-1,-a]}\) such that \(\Phi_l(z, -a) = (\psi(z), -a)\) for all \(z \in \mathbb{S}^1\). Similarly, there exists a self-map \(\Phi_r\) of \(\mathbb{S}^1 \times [a, 1]\) homotopic rel \(\mathbb{S}^1\times\{1\}\) to \(\varphi\times\operatorname{Id}_{[a,1]}\) such that \(\Phi_r(z, a) = (\psi(z), a)\) for all \(z \in \mathbb{S}^1\). Pasting \(\Phi_l\), \(\psi \times \operatorname{Id}_{[-a, a]}\), and \(\Phi_r\), we obtain the desired \(\Phi\).
\end{proof}

\begin{lemma}\label{auxnext}
    Let \( g \colon \mathbb{S}^1 \times [0,3] \to \mathbb{S}^1 \times [0,2] \) be a map such that  \( g^{-1}(\mathbb{S}^1 \times \{0\}) = \mathbb{S}^1 \times [1,2] \).  Then \( g \) can be homotoped, rel \( \mathbb{S}^1 \times \{0,3\} \), to a map that sends  \( \mathbb{S}^1 \times [0,3] \) into \( \mathbb{S}^1 \times (0,2] \).
\end{lemma}
\begin{proof}
   Let \( r \in (0,2) \) be such that \( g(\mathbb{S}^1 \times \{0,3\}) \) is disjoint from \( \mathbb{S}^1 \times [0,r] \). If  \( g_1 \colon \mathbb{S}^1 \times [0,3] \to \mathbb{S}^1 \) and \( g_2 \colon \mathbb{S}^1 \times [0,3] \to [0,2] \)  are the component functions of \( g \), then consider the homotopy  \( H \colon \mathbb{S}^1 \times [0,3] \times [0,1] \to \mathbb{S}^1 \times [0,2] \)  defined by  \( H(z,s,t) \coloneqq \big(g_1(z,s),\, (1 - t)g_2(z,s) + t \max\{r,\, g_2(z,s)\}\big) \).
\end{proof}
%As discussed, we now prove the following special case of \Cref{classification}.
\begin{theorem}\label{specialcase}
Let $M$ and $N$ be non-compact bordered surfaces such that boundary components of $N$ are compact. Suppose there is a proper map $f\colon (M, \partial M) \to (N, \partial N)$ so that $\pi_1(f)$ is a monomorphism and $f$ restricted to each component of $\partial M$ is a covering map onto a component of $\partial N$. Then $f$ can be properly homotoped rel  $\partial M$ to a finite-sheeted covering map.
\end{theorem}
\begin{proof}
     Note that the restriction of a proper map to a closed subset is proper. Thus, the restriction of \(f\) to each component of \(\partial M\) is a proper map into a component of \(\partial N\). Since each component of \(\partial N\) is compact, it follows that each component of \(\partial M\) is also compact. Choose exhausting sequences \(\{C_n\}\) for \(N\) and \(\{C_n'\}\) for \(M\), both satisfying conditions \ref{(i)}--\ref{(v)} of \Cref{exhaustion}. Let \(f(C_1') \subset C_{i_1} \setminus \operatorname{fr}(C_{i_1})\) for some \(i_1\). Choose \(i_2 > i_1\) such that \(f^{-1}(C_{i_1}) \subset C_{i_2}' \setminus \operatorname{fr}(C_{i_2}')\). Next, choose \(i_3 > i_2\) such that \(f(C_{i_2}') \subset C_{i_3} \setminus \operatorname{fr}(C_{i_3})\). Then choose \(i_4 > i_3\) such that \(f^{-1}(C_{i_3}) \subset C_{i_4}' \setminus \operatorname{fr}(C_{i_4}')\), and so on.

Let $C_0\coloneqq\varnothing\eqqcolon C'_0$. Thus, after passing to subsequences, we henceforth assume that our original exhausting sequences satisfy $C_n' \subset f^{-1}(C_n \setminus \operatorname{fr}(C_n))$ and $f^{-1}(C_n) \subset C_{n+1}' \setminus \operatorname{fr}(C_{n+1}')$ for all \(n\). As a consequence, we have $f(C_{n+1}'\setminus C_n')\subset C_{n+1}\setminus C_{n-1}$ since $f(C_{n+1}')\subset C_{n+1}$ and $f^{-1}(C_{n-1})\subset C_n'$ for all $n$. Without loss of generality, we may also assume that $C_1\cap f(\partial M)\neq \varnothing$.

Let $n$ be a positive integer. Define \(X_n \coloneqq \operatorname{cl}(C_{n+1}' \setminus C_n')\) and \(Y_n \coloneqq \operatorname{cl}(C_{n+1} \setminus C_{n-1})\). Since \(k_n \coloneqq \operatorname{fr}(C_n)\) is a pairwise disjoint finite collection of non-trivial circles in \(Y_n \setminus (\partial Y_n \cup f(\partial X_n))\), there exists a homotopy $H_n\colon X_n\times [0,1]\to Y_n$ rel \(\partial X_n\) from \(f\vert X_n \to Y_n\) to a map \(f_n \colon X_n \to Y_n\) such that $f_n$ is transverse with respect to $k_n$ and \(k_n' \coloneqq f_n^{-1}(k_n)\) is a pairwise disjoint finite collection of non-trivial circles in \(X_n \setminus \partial X_n\). We may further assume that there exists a tubular neighborhood \(U_n \coloneqq k_n \times [-1,1] \subset Y_n \setminus \partial Y_n\) of $k_n$, where \(k_n \times \{0\} \equiv k_n\), such that \(f_n^{-1}(U_n)\) can be identified with a tubular neighborhood \(U_n' \coloneqq k_n' \times [-1,1] \subset X_n \setminus \partial X_n\) of \(k_n'\), where \(k_n' \times \{0\} \equiv k_n'\), and \(f_n(z, t) = (f_n(z), t)\) for all \((z, t) \in U_n'\) \cite[Lemma 2.2]{MR2619608}. Notice that \( k_n' \neq \varnothing \), since \( f \) is proper and \( C_n \cap f(\partial M) \neq \varnothing \). Moreover, since \( f(X_n) \cap \operatorname{fr}(C_{n-1}) = \varnothing \), the homotopy \( H_n \) can be thought of as being performed in \( Y_n \setminus U_n \), where \( U_n \) is an open collar neighborhood of \( \operatorname{fr}(C_{n-1}) \) in \( Y_n \).

Therefore, after a proper homotopy rel \( \partial M \cup C_1' \), we may assume that \( f \) agrees with \( f_n \) on \( \operatorname{cl}(C_{n+1}' \setminus C_n') \) for all \( n  \). In particular, $f^{-1}(k_n)=k_n'$, since $\operatorname{im}(H_n)\cap k_{n-1}=\varnothing$ for all $n$.

Let \( n \) be a positive integer. Denote by \( D_n' \) the union of the closures of all components of \( M \setminus k_n' = M \setminus f^{-1}(k_n) \) that are mapped into \( C_n \) by \( f \). Thus, \( D_n' \) is a compact, possibly disconnected, \(2\)-dimensional submanifold of \( M \). Notice that \( D_n' \setminus \operatorname{fr}(D_n') = f^{-1}(C_n \setminus \operatorname{fr}(C_n)) \), and each component of \( \operatorname{fr}(D_n') \) is a component of \( k_n' \). In fact, \( \operatorname{fr}(D_n') = k_n' \), since \( f \) sends a small (two-sided) tubular neighborhood of each component of \( k_n' \) into a small (two-sided) tubular neighborhood of a component of \( \operatorname{fr}(C_n) \), preserving the fibers. Therefore, \( D_n' = f^{-1}(C_n) \). For future use, note also that each component of \( \partial D_n' \) is a non-trivial circle in \( M \), and \( \operatorname{fr}(D_n') \cap \partial M = \varnothing \).

Now \( D_1' = f^{-1}(C_1) \supseteq C_1' \), and \( f(X_n) \subseteq Y_n \) for all \( n \). Thus, an inductive argument shows that \( C_n' \subseteq D_n' \) for all \( n \). Hence, \( \bigcup_n D_n' = M \). Since \( \operatorname{int}(C_n) \subset \operatorname{int}(C_{n+1}) \), it follows that \( \operatorname{int}(D_n') \subset \operatorname{int}(D_{n+1}') \) for all \( n \). Therefore, \( \{D_n'\} \) is an exhausting sequence for \( M \).

Now, \Cref{aux} allows us to properly homotope \(f\) rel \(\partial M\) to a map \(f'\) such that, for all \(n\), the following hold: \(f'^{-1}(C_n) = D_n'\), and if \(F'\) is a component of \(k_n'=\operatorname{fr}(D_n')\), then \(f'|F'\) is a covering map onto a component of \(k_n=\operatorname{fr}(C_n)\). Moreover, \Cref{aux} says that there exist tubular neighborhoods \(k_n' \times [-1/2,1/2] \subset M \setminus \partial M\) and \(k_n \times [-1/2,1/2] \subset N \setminus \partial N\), where \(k_n' \times \{0\} \equiv k_n'\) and \(k_n \times \{0\} \equiv k_n\), such that $f'^{-1}(k_n\times[-1/2,1/2])=k_n'\times[-1/2,1/2]$ and  \(f'(z,t) = (f'(z), t)\) for all \((z,t) \in k_n' \times [-1/2,1/2]\). In particular, for all \( n \), we have \( f'^{-1}(k_n) = k_n' \), and hence \( D_n' \setminus \operatorname{fr}(D_n') = f'^{-1}(C_n \setminus \operatorname{fr}(C_n)) \).

Let \(D'\) be a component of some \(D_n'\). Then \(D'\) is an essential compact bordered subsurface of \(M\), since each component of $\partial D'$ is a non-trivial circle in $M$ (see \Cref{rhofunctor}). Hence the restriction \(f'|D' \to C_n\) is \(\pi_1\)-injective. Moreover, since \(f'|\partial D'  \to \partial C_n\) is a local homeomorphism, it follows from \cite[Lemma 1.4.3]{MR224099} that exactly one of the following holds:
\begin{enumerate}[(a)]
    \item\label{(a)} \(f'|D'\) is homotopic rel \(\partial D'\) to a covering map onto \(C_n\), or
    \item\label{(b)} \(D'\) is an annulus and \(f'|D'\) is homotopic rel \(\partial D'\) to a map into \(\partial C_n\).
\end{enumerate}
Suppose in \ref{(a)} that \(D'\) is an annulus. Then, by the multiplicativity of Euler characteristic under finite-sheeted coverings, \(C_n\) must also be an annulus. Moreover, one component of \(\partial C_n\) must be a component of \(\partial N\), and the other component must be \(\operatorname{fr}(C_n)\), since \(C_n \setminus \operatorname{fr}(C_n)\) intersects the subset \(f'(\partial M) = f(\partial M)\) of \(\partial N\). Consequently, one component of \(\partial D'\) must be a component of \(\partial M\), and the other component must be \(\operatorname{fr}(D')\), because \(f'\) maps \(\operatorname{fr}(D')\) into \(\operatorname{fr}(C_n)\), and a covering map from \(D'\) to \(C_n\) sends \(D' \setminus \partial D'\) onto \(C_n \setminus \partial C_n\).

Now, consider the case \ref{(b)}. By the connectedness of continuous image, $f'(\partial D')$ must contained in a single component of $\partial C_n$.  Notice that at least one of the components $\partial D'$ must be a component of $k_n'\subset M\setminus \partial M$, i.e., $\varnothing\neq\operatorname{fr}(D')\subseteq \partial D'$. In fact, $\partial D'=\operatorname{fr}(D')$ because \(f'(\operatorname{fr}(D'))\subseteq\operatorname{fr}(C_n)\) does not  intersect with $f'(\partial M)\subseteq\partial N$. Thus, \ref{(b)} can be rewritten as follows:
\begin{enumerate}[(b.1)]
    \item\label{(b')} \(D'\) is an annulus and \(f'|D'\) is homotopic rel $\partial D'=\operatorname{fr}(D')$ to a map into \(\operatorname{fr}(C_n)\).
\end{enumerate}
Now, recall that \( f' \) maps a tubular neighborhood of each component of \( \partial D' = \operatorname{fr}(D') \) into a tubular neighborhood of a component of \( \operatorname{fr}(C_n) \) in a fiber-preserving manner. Therefore, by \Cref{auxnext}, there exists a homotopy of \( f' \), constant outside an arbitrarily small neighborhood of \( D' \), to a map \( f_{D'}' \) such that \( f_{D'}'^{-1}(C_n) = D_n' \setminus D' \) and \( f_{D'}'^{-1}(C_{n+1}) = D_{n+1}' \).

Observe that it may happen that for some \( n \), \( D_n' \) has a component \( A_n' \) that falls under case~\ref{(b')}. Then the component \( A_{n+1}' \) of \( D_{n+1}' \), which contains \( A_n' \) in its interior, also falls under case~\ref{(b')}, and similarly, the component \( A_{n+2}' \) of \( D_{n+2}' \), which contains \( A_{n+1}' \) in its interior, again falls under case~\ref{(b')}, and so on. However, this process cannot continue indefinitely; otherwise, by the connectedness of \( M \), the \(2\)-manifold \( M \) with non-empty boundary would be homeomorphic to \( \mathbb{S}^1 \times \mathbb{R} \). Thus, we have a pairwise disjoint collection \( \{A_\alpha'\} \) of annuli with the following properties: each \( A_\alpha' \) is a component of some \( D_{n_\alpha}' \); each \( A_\alpha' \) falls under case~\ref{(b')}; and if \( A' \) is a component of some \( D_n' \) such that \( A' \) falls under case~\ref{(b')}, then \( A' \) is contained in the interior of some \( A_\alpha' \). In particular, if \( D' \) is the component of \( D_{n_\alpha+1}' \) for which \( A_\alpha' \subset \operatorname{int}(D') \), then \( D' \) must fall under case~\ref{(a)}. Similar to the previous paragraph, for each \( \alpha \), applying \Cref{auxnext}, we can remove \( A_\alpha' \) from \( D_{n_\alpha}' \). More precisely, for each \( \alpha \), there exists a small neighborhood \( U_\alpha' \) of \( A_\alpha' \) in \( M \setminus \partial M \), and a homotopy \( H_\alpha \), rel \( M \setminus U_\alpha' \), from \( f' \) to a map \( f_{A_\alpha'}' \) such that \( f_{A_\alpha'}'^{-1}(C_{n_\alpha}) = D_{n_\alpha}' \setminus A_\alpha' \) and \( f_{A_\alpha'}'^{-1}(C_{n_\alpha+1}) = D_{n_\alpha+1}' \). We may also assume that \( U_\alpha' \cap U_\beta' = \varnothing \) if \( \alpha \neq \beta \).

%By choosing a subsequence of $\{C_n\}$ if necessary, the above leads to two cases:
%\begin{enumerate}[(a.{1})]
    %\item There is an exhausting sequence for $M$ whose $n$th term is a component of $D_n'$, which falls under case \ref{(a)} above, or
    %\item Given any component of $D_n'$, there exists $m \geq n$ so that the component of $D_m'$ containing it falls under case \ref{(b)}.
%\end{enumerate}

%Applying the homotopy mentioned after case \ref{(b)} above and choosing a subsequence if necessary, we see there is a proper homotopy rel $\partial M$ of $f'$ to a map $f_1'$ so that if $D_n = f_1'^{-1}(C_n)$, then $D_n$ consists of some components of $D_n'$ and either
%\begin{enumerate}[(1)]
    %\item\label{a2} For any $n$ and any component $D$ of $D_n$, $f_1'|D$ is homotopic rel $\partial D$ to a covering map onto $C_n$, or
    %\item\label{b2} $D_n$ is an annulus for every $n$, and $f_1'|D_n$ is homotopic rel $\partial D_n$ to a map into $\partial C_n$.
%\end{enumerate}

Thus, choosing a subsequence if necessary, and pasting all $H_\alpha$, we see there is a proper homotopy rel $\partial M$ of $f'$ to a map $f_1'$ so that if $D_n = f_1'^{-1}(C_n)$, then $D_n$ consists of some components of $D_n'$, and 
\begin{enumerate}[(a.1)]
    \item \label{a2}for any $n$ and any component $D$ of $D_n$, $f_1'|D$ is homotopic rel $\partial D$ to a covering map onto $C_n$.
\end{enumerate}

%Note that since $\partial M\neq \varnothing$, \ref{b2} will not appear. Now suppose case \ref{a2} above holds.

Notice that $D_n\setminus\operatorname{fr}(D_n)=f_1'^{-1}(C_n\setminus\operatorname{fr}(C_n))$ and $\operatorname{fr}(D_n)= f_1'^{-1}(\operatorname{fr}(C_n))$, and hence, $f_1'(\operatorname{cl}(D_{n+1} \setminus D_n))\subseteq\operatorname{cl}(C_{n+1} \setminus C_n)$ for all $n$.  If, for every \( n \) and every component \( S \) of \( \operatorname{cl}(D_{n+1} \setminus D_n) \), the restriction \( f_1'\vert S \) is homotopic rel \( \partial S \) to a covering map, then these homotopies fit together to give a proper homotopy rel \( \partial M \) from \( f_1' \) to a covering map \( M \to N \), thereby proving the theorem. The aim of the next couple of paragraphs is to properly homotope \( f_1' \) rel \( \partial M \) so that this condition is satisfied.

%From now on, we use \( S \) to denote a component of \( \operatorname{cl}(D_{n+1} \setminus D_n) \), where \( n \) is arbitrary. 

Suppose \( n_0 > 1 \) is the smallest integer such that \( \operatorname{cl}(D_{n_0+1} \setminus D_{n_0}) \) has an annular component \( S_0 \) for which there exists a homotopy $H_0\colon S_0\times [0,1]\to N$ rel \( \partial S_0 \) from \( f_1'\vert S_0 \) to a map \( S_0 \to f_1'(\partial S_0) \). Observe that \( \partial S_0 \subseteq \partial M \cup \operatorname{fr}(D_{n_0}) \cup \operatorname{fr}(D_{n_0+1}) \). However, \( \partial S_0 \) cannot be contained in \( \partial M \) since \( M \) is connected. Now, the connectedness of \( H_0(S_0,1) = f_1'(\partial S_0) \) implies that \( \partial S_0 \) is entirely contained in either \( \operatorname{fr}(D_{n_0}) \) or \( \operatorname{fr}(D_{n_0+1}) \), because the images under \( f_1' \) of \( \partial M \), \( \operatorname{fr}(D_{n_0}) \), and \( \operatorname{fr}(D_{n_0+1}) \) lie in three pairwise disjoint sets: \( \partial N \), \( \operatorname{fr}(C_{n_0}) \), and \( \operatorname{fr}(C_{n_0+1}) \), respectively. In particular, we have \( \partial S_0 = \operatorname{fr}(S_0) \). Moreover, we can conclude that \( \partial S_0 \subset \operatorname{fr}(D_{n_0}) \), since if \( \partial S_0 \subset \operatorname{fr}(D_{n_0+1}) \), then \( S_0 \) would be a component of \( D_{n_0+1} \), which is excluded by \ref{a2}. Thus, by \Cref{auxnext}, there exists a homotopy of \( f_1' \), relative to the complement (taken in  \(M\)) of a small neighborhood of \( S_0 \), to a map \( g_1 \) such that \(g_1^{-1}(C_{n_0+1}) = D_{n_0+1}\)  and \(g_1^{-1}(C_{n_0}) = D_{n_0} \cup S_0.\) Note that \(\operatorname{cl}(D_{n_0+1} \setminus (D_{n_0} \cup S_0)) = \operatorname{cl}(D_{n_0+1} \setminus D_{n_0}) \setminus \operatorname{int}(S_0).\)
In particular, the number of components of \( \operatorname{cl}(g_1^{-1}(C_{n_0+1}) \setminus g_1^{-1}(C_{n_0})) \) is strictly less than the number of components of \( \operatorname{cl}(f_1'^{-1}(C_{n_0+1}) \setminus f_1'^{-1}(C_{n_0})) \).

Now, we may need to properly homotope \( g_1 \) rel \( \partial M \) if either condition~\ref{a2} is not satisfied or if there exists an integer smaller than \( n \) for which \( g \) satisfies the property analogous to that satisfied by \( f \), as described in the first line of the previous paragraph. Based on these two cases, we will consider two types of proper homotopies in the following two paragraphs.

It may now happen that the restriction \( g_1 \vert D_0 \) of \( g_1 \) to the component \( D_0 \) of \( g_1^{-1}(C_{n_0}) = D_{n_0} \cup S_0 \) containing \( S_0 \) is no longer homotopic rel \( \partial D_0 \) to a covering map. That is, \( D_0 \) is an annulus with \( \operatorname{fr}(D_0) = \partial D_0 \) such that \( S_0 \subset \operatorname{int}(D_0) \), and \( g_1 \vert D_0 \) is homotopic rel \( \partial D_0 \) to a map \( D_0 \to g_1(\partial D_0) \subseteq \operatorname{fr}(C_{n_0}) \). By \Cref{auxnext}, we may now perform a homotopy of \( g_1 \), relative to the complement (taken in \( M \)) of a small neighborhood of \( D_0 \), to a map \( g_1' \) such that \( g_1'(D_0) \subset C_{n_0+1} \setminus C_{n_0} \). We are thus back in case~\ref{a2}, since \( g_1'^{-1}(C_{n_0+1}) = D_{n_0+1} \). Moreover, the number of components of \( \operatorname{cl}(g_1'^{-1}(C_{n_0}) \setminus g_1'^{-1}(C_{n_0-1})) \) is strictly less than the number of components of \( \operatorname{cl}(g_1^{-1}(C_{n_0}) \setminus g_1^{-1}(C_{n_0-1})) \). Note that \( \operatorname{cl}(g_1^{-1}(C_{n_0}) \setminus g_1^{-1}(C_{n_0-1})) \) has at most as many components as \( \operatorname{cl}(D_{n_0} \setminus D_{n_0-1}) \).

On the other hand, it may happen that $g_1$ already satisfies \ref{a2}, but the component $S_0'$ of $\operatorname{cl}(g_1^{-1}(C_{n_0}) \setminus D_{n_0-1})=\operatorname{cl}(g_1^{-1}(C_{n_0}) \setminus g_1^{-1}(C_{n_0-1}))$ for which $S_0\subset \operatorname{int}(S_0')$, may be an annulus with $\partial S_0'\subseteq \operatorname{fr}(D_{n_0-1})=\operatorname{fr}(g_1^{-1}(C_{n_0-1}))$ such that  $g_1\vert S_0'$ is homotopic rel  $\partial S_0'$ to a map $S_0'\to\operatorname{fr}(C_{n_0-1})$ (that is, $n_0$ is no longer the smallest for $g_1$). We perform a homotopy as we did on $f_1'$ this time decreasing the number of components in $\operatorname{cl}(g_1^{-1}(C_{n_0}) \setminus g_1^{-1}(C_{n_0-1}))$. More precisely, by \Cref{auxnext}, there exists a homotopy of \( g_1\), relative to the complement (taken in  \(M\)) of a small neighborhood of \( S_0' \), to a map \( g_1' \) such that \(g_1'^{-1}(C_{n_0}) = g_1^{-1}(C_{n_0})\) and \(g_1'^{-1}(C_{n_0-1}) = D_{n_0-1} \cup S_0'.\) 
In particular, the number of components of \( \operatorname{cl}(g_1'^{-1}(C_{n_0}) \setminus g_1'^{-1}(C_{n_0-1})) \) is strictly less than the number of components of \( \operatorname{cl}(g_1^{-1}(C_{n_0}) \setminus g_1^{-1}(C_{n_0-1})) \).

After a finite number of steps, we have constructed a map $g_{2}$, properly homotopic rel $\partial M\cup \operatorname{cl}(M\setminus D_{n_0+1})$ to $g_1$ so that $g_{2}$ is again in case \ref{a2}. Moreover, if $g_2^{-1}(C_k) \neq D_k$ for some $k \leq n_0$, then $g_2^{-1}(C_k)\setminus g_2^{-1}(C_{k-1})$ has fewer components than $D_k\setminus D_{k-1}$, while if $g_{2}^{-1}(C_k) = D_k$ for some $k \leq n_0$, then the homotopy is constant on $D_k$. Finally, $g_{2}|S$ is homotopic rel $\partial S$ to a covering map if $S$ is
a component of $\operatorname{cl}(g_{2}^{-1}(C_k) \setminus g_{2}^{-1}(C_{k-1}))$ and $k \leq n_0$.

We now proceed by induction to improve $g_{2}$. Notice that the number of components of $\operatorname{cl}(D_k \setminus D_{k-1})$ cannot decrease indefinitely. Thus, the process stabilizes on the inverse image of $C_k$ after a finite number of steps. Therefore, we can construct a proper homotopy of $f_1'$ rel $\partial M$ to a covering map.\end{proof}

\begin{remark}\label{specialcaseremark}
    The proof of \Cref{specialcase} can easily be modified—ignoring the boundaries of the domain and codomain—to give a corresponding result for all surfaces, except when the surface is either the plane or the punctured plane. The exclusion of the plane is due to the heavy dependence on part \ref{(iii)} of \Cref{exhaustion}. The exclusion of the punctured plane arises from the fact that, in adapting the Brown--Tucker argument to dimension two, we do not consider the case where the domain is a (boundaryless) $3$-manifold that appears as the total space of a fiber bundle over a compact surface with fiber~\(\mathbb{R}\). More specifically, the proof of \Cref{specialcase} does not adapt case ($\text{b}_1$), case ($\text{b}_2$), or the three paragraphs following ($\text{b}_2$) from the proof of \cite[Theorem~4.2]{MR334225}, because their two-dimensional adaptation concerns boundaryless $2$-manifolds.  
\end{remark}

 %We conclude this section by proving \Cref{classification} using \Cref{specialcase}.

\begin{proof}[Proof of \Cref{classification}]
    By the first paragraph of the proof of \Cref{specialcase}, every component of $\partial M$ is compact. Since a \(\pi_1\)-injective map \(\mathbb{S}^1 \to \mathbb{S}^1\) is (properly) homotopic to a covering map, there exists a proper homotopy \(h \colon \partial M \times [0,1] \to \partial N\) from \(f\vert\partial M \to \partial N\), such that \(h(-,1)\) restricts to a covering map on each component of \(\partial M\) onto a component of \(\partial N\). By \cite[Theorem~10.6]{MR198479}, the boundary of a smooth manifold is a subcomplex of some CW-structure on the manifold. Therefore, applying the proper homotopy extension theorem \cite[Theorem~1.6]{MR334226}, we obtain a proper homotopy \( \widetilde{h} \colon M \times [0,1] \to N \) from \( f \), such that \( \widetilde{h}\vert\partial M \times [0,1] = h \). In particular, \(\widetilde h(-,0) = f\), \(\widetilde h(\partial M \times [0,1]) \subseteq \partial N\), and \(\widetilde h(-,1)\) restricts to a covering map on each component of \(\partial M\) onto a component of \(\partial N\). Applying \Cref{specialcase}, the result follows.
\end{proof}

\subsection{Cut-off technique}The goal of this section is to apply the theory developed earlier—specifically \Cref{specialcase} and its proof—to show that a \(\underline{\pi}_1\)-injective proper map between surfaces behaves well outside a compact subset of the domain. We also establish an analogous result in the context of \(3\)-manifolds. As an application, we provide a sufficient condition under which a degree-one map between surfaces admits a geometric kernel.

To proceed, we fix some notation for the remainder of the section. Let \(M\) and \(N\) be connected, non-compact \(2\)-manifolds such that each component of \(\partial M \sqcup \partial N\), if any, is compact, and neither \(M\) nor \(N\) is homeomorphic to \(\mathbb{R}^2\). For \Cref{deg}, we also fix orientations on \( M \) and \( N \). Let \(f \colon M \to N\) be a proper map such that the restriction of \(f\) to each component of \(\partial M\) is a covering map onto a component of \(\partial N\). We begin with two lemmas that will be used in the proof of \Cref{cut}. 

\begin{lemma}\label{cutlemma1}
    Let $e$ be an end of $M$, and let $a$ and $a'$ be two representatives of $e$. If $\underline\pi_1(f)\colon \underline \pi_1(M,\underline a)\to \underline \pi_1(N,\underline{fa})$ is a monomorphism, then  $\underline\pi_1(f)\colon \underline \pi_1(M,\underline a')\to \underline \pi_1(N,\underline{fa'})$ is also a monomorphism. 
\end{lemma}
\begin{proof}
    Let \(p = \{p_k\}\) be a path in \(M\) from \(\underline{a}\) to \(\underline{a'}\). Then the sequence \(q \coloneqq \{f p_k\}\) defines a path in \(N\) from \(\underline{f a}\) to \(\underline{f a'}\). Furthermore, the composition \(\underline{\pi}_1(M, \underline{a}) \xrightarrow{p_*} \underline{\pi}_1(M, \underline{a'}) \xrightarrow{\underline{\pi}_1(f)} \underline{\pi}_1(N, \underline{f a'})\) coincides with \(\underline{\pi}_1(M, \underline{a}) \xrightarrow{\underline{\pi}_1(f)} \underline{\pi}_1(N, \underline{f a}) \xrightarrow{q_*} \underline{\pi}_1(N, \underline{f a'})\). Since both \(p_*\) and \(q_*\) are isomorphisms, the claim follows.
\end{proof}

\begin{lemma}\label{cutlemma2}
    Let \(f' \colon M \to N\) be a proper map that is properly homotopic to \(f\), and let \(e\) be an end of \(M\) with a representative \(a \in e\). Suppose \(\underline{\pi}_1(f) \colon \underline{\pi}_1(M, \underline{a}) \to \underline{\pi}_1(N, \underline{f a})\) is a monomorphism. Then \(\underline{\pi}_1(f') \colon \underline{\pi}_1(M, \underline{a}) \to \underline{\pi}_1(N, \underline{f' a})\) is also a monomorphism.
\end{lemma}
\begin{proof}
    Let \(\alpha \colon (\underline{\mathbb{S}^1}, \underline{*}) \to (M, \underline{a})\) be an arbitrary proper map of pairs such that \(\underline{\pi}_1(f')([\alpha]) = 0\). It suffices to show that \(\underline{\pi}_1(f)([\alpha]) = 0\). Without loss of generality, we may assume that \(\alpha\vert_{\underline{*}} = a\). For each \(k \geq 0\), let \(\ell_k\) denote the loop given by restricting \(\alpha\) to the \(k\)th circle of \(\underline{\mathbb{S}^1}\). Since \(\underline{\pi}_1(f')([\alpha]) = 0\), there exists a sequence \(\{h_k \colon \mathbb{S}^1 \times [0,1] \to N : k \geq 0\}\) of homotopies such that \(\{\operatorname{im}(h_k) : k \geq 0\}\) converges to \(e^\dagger \coloneqq \underline{\pi}_0(f)(e) = \underline{\pi}_0(f')(e)\), and for all sufficiently large \(k\), the following holds: \(h_k\) is a homotopy of loops based at \(f'a(k)\), from \(f'\ell_k\) to the constant loop.

    Consider a proper homotopy \(H \colon M \times [0,1] \to N\) from \(f'\) to \(f\). Then, for each \(k \geq 0\), we can construct a homotopy \(h_k' \colon \mathbb{S}^1 \times [0,1] \to N\) of loops based at \(f a(k)\), from \(f \ell_k\) to \(p_k * f' \ell_k * \overline{p_k}\), where \(p_k\) denotes the path \(H(a(k), -)\), such that \(\operatorname{im}(h_k') \subseteq H(\operatorname{im}(\ell_k), [0,1])\). In particular, \(\{\operatorname{im}(h_k') : k \geq 0\}\) converges to \(e^\dagger\). Also, define \(h_k'' \coloneqq \{p_k * h_k(-,t) * \overline{p_k}\}_{t \in [0,1]}\) for each $k\geq0$. Then \(\{h_k'' : k \geq 0\}\) is a sequence of homotopies with \(\{\operatorname{im}(h_k'') : k \geq 0\}\) converging to \(e^\dagger\), and for all sufficiently large \(k\), each \(h_k''\) is a homotopy of loops based at \(f a(k)\), from \(p_k * f' \ell_k * \overline{p_k}\) to the constant loop. Finally, for all sufficiently large \(k\), the homotopy \(G_k \coloneqq \{h_k'(-,t) * h_k''(-,t)\}_{t \in [0,1]}\) is a homotopy of loops based at \(f a(k)\), from \(f \ell_k\) to the constant loop. Moreover, \(\{\operatorname{im}(G_k) : k \geq 0\}\) converges to \(e^\dagger\). Therefore, \(\underline{\pi}_1(f)([\alpha]) = 0\).
\end{proof}

\begin{theorem}\label{cut}
 Suppose for every end $e$ of $M$, there exists a representative $a$ of $e$ such that $\underline\pi_1(f)\colon \underline\pi_1(M,\underline a)\to \underline \pi_1(N,\underline{fa})$ is a monomorphism. Then there exist compact $2$-dimensional submanifolds \( D' \) of \( M \) and \( D \) of \( N \) such that, after a proper homotopy rel $\partial M$, the restriction of \( f \) to the closure of each unbounded component of \( M \setminus D' \) is a finite-sheeted covering map onto the closure of some unbounded component of \( N \setminus D \). 
 \end{theorem}

\begin{proof} Let \(\{C_n\}\) be an exhausting sequence for \(N\) satisfying the five properties \ref{(i)}--\ref{(v)} of \Cref{exhaustion}. Using the techniques developed in the first seven paragraphs of the proof of \Cref{specialcase}, $f$ can properly homotoped rel $\partial M$ to a map $f'$ such that each component of \(f'^{-1}(\bigcup_n \operatorname{fr}(C_n))\) is a non-trivial circle in \(M\setminus \partial M\), and that the restriction of \(f'\) to each component of \(f'^{-1}(\bigcup_n \operatorname{fr}(C_n))\) is either a constant map or a covering map. For each \(n\), let \(A_{n,1}, \ldots, A_{n,k_n}\) denote the closures (taken in $M$) of the unbounded components of \(M \setminus f'^{-1}(C_n)\), and let \(B_{n,1}, \ldots, B_{n,l_n}\) denote the closures (taken in $N$) of the components of \(N \setminus C_n\). Then, for each \(n\), there exists a map \(\sigma_n \colon \{1, \ldots, k_n\} \to \{1, \ldots, l_n\}\) such that \(f'(A_{n,i}) \subseteq B_{n,\sigma_n(i)}\). We denote the restriction \(f'\vert A_{n,i} \to B_{n,\sigma_n(i)}\) by \(f'_{n,i}\). Note that, for each \(n\) and each \(i \in \{1, \ldots, k_n\}\), the restriction of \(f'_{n,i}\) to each component of \(\partial A_{n,i}\) is a covering map onto a component of \(\partial B_{n,\sigma_n(i)}\).

Now, consider an arbitrary end \(e\) of \(M\), and let \(a\) be a representative of \(e\) such that \(\underline{\pi}_1(f) \colon \underline{\pi}_1(M, \underline{a}) \to \underline{\pi}_1(N, \underline{f a})\) is a monomorphism. Since \(\{f'^{-1}(C_n)\}\) forms an exhausting sequence for \(M\), for each \(j \in \{1, \ldots, k_{n+1}\}\), there exists \(i \in \{1, \ldots, k_n\}\) such that \(A_{n,i} \supset A_{n+1,j}\). Thus, we obtain a decreasing sequence \(A_{1,\tau_e(1)} \supset A_{2,\tau_e(2)} \supset \cdots\) such that \(\bigcap_n A_{n,\tau_e(n)} = \varnothing\), and each \(\underline{\operatorname{int}(A_{n,\tau_e(n)})}\) is a basic open neighborhood of \(e\). Choose a proper map \(a_e \colon [0, \infty) \to M\) such that \(a_e([n, \infty)) \subseteq A_{n,\tau_e(n)}\) and \(a_e(n) \in \operatorname{fr}(A_{n,\tau_e(n)})\) for each \(n\). Then \(a_e\) represents \(e\), i.e., \(a_e \in e\). For each \(n\), let \(\pi_1(A_{n+1,\tau_e(n+1)}, a_e(n+1)) \to \pi_1(A_{n,\tau_e(n)}, a_e(n))\) denote the monomorphism induced by the \(\pi_1\)-injective inclusion \(A_{n+1,\tau_e(n+1)} \hookrightarrow A_{n,\tau_e(n)}\) (see \Cref{rhofunctor}), together with the change of basepoint along the path \(a_e\vert [n,n+1]\). Similarly, for each \(n\), let \(\pi_1(B_{n+1,\sigma_n(\tau_e(n+1))}, f'a_e(n+1)) \to \pi_1(B_{n,\sigma_n(\tau_e(n))}, f'a_e(n))\) denote the monomorphism induced by the \(\pi_1\)-injective inclusion \(B_{n+1,\sigma_n(\tau_e(n+1))} \hookrightarrow B_{n,\sigma_n(\tau_e(n))}\), together with the change of basepoint along the path \(f'a_e\vert[n,n+1]\). Thus, we obtain the following ladder of commutative diagrams.
 $$\adjustbox{scale=1}{\begin{tikzcd}
\cdots  \arrow[r] & {\displaystyle\pi_1(A_{n+1, \tau_e(n+1)}, a_e(n+1))} \arrow[r] \arrow[d, "\displaystyle\pi_1(f'_{n+1,\tau_e(n+1)})"'] & {\displaystyle\pi_1(A_{n,\tau_e(n)}, a_e(n))} \arrow[r] \arrow[d, "\displaystyle\pi_1(f'_{n,\tau_e(n)})"] & \cdots \\
\cdots \arrow[r]  & {\displaystyle\pi_1(B_{n+1,\sigma_n(\tau_e(n+1))}, f'a_e(n+1))} \arrow[r]                                       & {\displaystyle\pi_1(B_{n,\sigma_n(\tau_e(n))}, f'a_e(n))} \arrow[r]                                  & \cdots
\end{tikzcd}}$$

We claim that there exists an integer \(n_e\) such that \(\pi_1(f'_{n,\tau_e(n)})\) is injective for all \(n \geq n_e\). Suppose, to the contrary, that for infinitely many \(n\), there exists a loop \(\gamma_n\) in \(A_{n,\tau_e(n)}\) based at \(a_e(n)\) such that \([\gamma_n]\) is a non-trivial element of \(\pi_1(A_{n,\tau_e(n)}, a_e(n))\), but \([f'_{n,\tau_e(n)} \gamma_n]\) is trivial in \(\pi_1(B_{n,\sigma_n(\tau_e(n))}, f'a_e(n))\). Chasing the above ladder of commutative diagrams then shows that for all \(n \in \mathbb{N}\), there exists a loop \(\delta_n\) in \(A_{n,\tau_e(n)}\) based at \(a_e(n)\) such that \([\delta_n]\) is a non-trivial element of \(\pi_1(A_{n,\tau_e(n)}, a_e(n))\), but \([f'_{n,\tau_e(n)} \delta_n]\) is trivial in \(\pi_1(B_{n,\sigma_n(\tau_e(n))}, f'a_e(n))\). Let \(\delta_0\) be the constant loop based at \(a_e(0)\). Now define a proper map of pairs \(\alpha \colon (\underline{\mathbb{S}^1}, \underline{*}) \to (M, \underline{a_e})\) by setting \(\alpha|\underline{*} \coloneqq a_e\) and \(\alpha|\mathbb{S}^1_k \coloneqq \delta_k\) for each \(k \geq 0\). Then \([\alpha]\) is a non-trivial element of \(\underline{\pi}_1(M, \underline{a_e})\), but \(\underline{\pi}_1(f')\) sends \([\alpha]\) to the trivial element of \(\underline{\pi}_1(N, \underline{f'a_e})\). By \Cref{cutlemma1} and \Cref{cutlemma2}, this contradicts our hypothesis. Therefore, there exists an integer \(n_e\) such that \(\pi_1(f'_{n,\tau_e(n)})\) is injective for all \(n \geq n_e\).

We now repeat the above argument for every \( e \in \underline{\pi}_0(M) \). Thus, for each \( e \in \underline{\pi}_0(M) \), there exists an integer \( n_e \) such that the following hold for all \( n \geq n_e \): (1) \( \underline{\operatorname{int}(A_{n,\tau_e(n)})} \) is a neighborhood of \( e \); (2) the map \( f'\vert A_{n,\tau_e(n)} \to B_{n,\sigma_n(\tau_e(n))} \) is \( \pi_1 \)-injective; and (3) the restriction of \( f' \) to each component of \( \partial A_{n,\tau_e(n)} \) is a covering map onto a component of \( \partial B_{n,\sigma_n(\tau_e(n))} \). By the compactness of \(\underline{\pi}_0(M)\), there exist finitely many points \(e_1, \dots, e_m \in \underline{\pi}_0(M)\) such that \(\bigcup_{i=1}^m \underline{\operatorname{int}(A_{n_{e_i}, \tau_{e_i}(n_{e_i})})} = \underline{\pi}_0(M)\). Let \(\underline{n} \coloneqq \max \{ n_{e_1}, \dots, n_{e_m} \}\). Since \(\operatorname{int}(A_{\underline{n}, 1}), \dots, \operatorname{int}(A_{\underline{n}, k_{\underline{n}}})\) are all the unbounded components of \(M \setminus f'^{-1}(C_{\underline{n}})\), we have \(\bigcup_{j=1}^{k_{\underline{n}}} \underline{\operatorname{int}(A_{\underline{n}, j})} = \underline{\pi}_0(M)\). Thus, for every \(j \in \{1, \dots, k_{\underline{n}}\}\), there exists \(i \in \{1, \dots, m\}\) such that \(\underline{\operatorname{int}(A_{\underline{n}, j})}\) intersects \(\underline{\operatorname{int}(A_{n_{e_i}, \tau_{e_i}(n_{e_i})})}\), and hence \(A_{n_{e_i}, \tau_{e_i}(n_{e_i})}\) must contain \(A_{\underline{n}, j}\). This follows because, for integers \(t \geq s\), each (unbounded) component of \(M \setminus f'^{-1}(C_t)\) is contained in an (unbounded) component of \(M \setminus f'^{-1}(C_s)\). Therefore, since \( f'\vert A_{n_{e_i}, \tau_{e_i}(n_{e_i})} \) is \(\pi_1\)-injective for all \( i \in \{1, \dots, m\} \), and the components of \( M \setminus f'^{-1}(C_{\underline{n}}) \) are essential (see \Cref{rhofunctor}), it follows that \( f'\vert A_{\underline{n}, j} \) is also \(\pi_1\)-injective for all \( j \in \{1, \dots, k_{\underline{n}}\} \). Recall that, for each \( j \in \{1, \ldots, k_{\underline{n}}\} \), the restriction of \( f' \) to each component of \(\partial A_{\underline{n}, j}\) is a covering map onto a component of \(\partial B_{\underline{n}, \sigma_{\underline{n}}(j)}\). By \Cref{classification}, for each \( j = 1, \dots, k_{\underline{n}} \), there exists a proper homotopy \( H_{\underline{n}, j} \colon A_{\underline{n}, j} \times [0,1] \to B_{\underline{n}, \sigma_{\underline{n}}(j)} \) from \( f' \vert A_{\underline{n}, j} \), relative to \(\partial A_{\underline{n}, j}\), to a finite-sheeted covering map. In particular, each \( H_{\underline{n}, j} \) is a proper homotopy rel \( \operatorname{fr}(A_{\underline{n}, j}) \) from \( f' \). Therefore, we obtain a proper homotopy \( H \colon M \times [0,1] \to N \) rel \( \operatorname{cl}(M \setminus \bigcup_{j=1}^{k_{\underline{n}}} A_{\underline{n}, j} ) \cup \partial M \) from \( f' \), such that \( H \) extends \( H_{\underline{n}, j} \) for each \( j = 1, \dots, k_{\underline{n}} \). Define \( g \coloneqq H(-,1) \), \( D' \coloneqq f'^{-1}(C_{\underline{n}}) \), and \( D \coloneqq C_{\underline{n}} \). Then \( f' \) is properly homotopic to \( g \) relative to \( \partial M \), and the restriction of \( g \) to the closure of each unbounded component of \( M \setminus D' \) is a finite-sheeted covering map onto the closure of some unbounded component of \( N \setminus D \). To conclude the proof, recall that \( f \) is properly homotopic to \( f' \) relative to \( \partial M \).
\end{proof}
\begin{remark}\label{unbounded}
    By definition, \( D \) is connected and essential, whereas \( D' \) may be disconnected; however, each component of \( D' \) is essential (see \Cref{exhaustion} and \Cref{rhofunctor}). Moreover, by \ref{(v)} of \Cref{exhaustion}, every component of \( N \setminus D \) is unbounded, although \( M \setminus D' \) may have bounded components.
\end{remark}

If we replace “mono” with “iso” in the statements of \Cref{cutlemma1}, \Cref{cutlemma2}, and \Cref{cut}, then the corresponding versions also hold.

\begin{corollary}\label{iso}
     Suppose that for every end \( e \) of \( M \), there exists a representative \( a \) of \( e \) such that \( \underline{\pi}_1(f) \colon \underline{\pi}_1(M, \underline{a}) \to \underline{\pi}_1(N, \underline{f a}) \) is an isomorphism. Then there exist compact, $2$-dimensional submanifolds \( D' \subset M \) and \( D \subset N \) such that \( f \) can be properly homotoped rel \( \partial M \) to send the closure of each unbounded component of \( M \setminus D' \) homeomorphically onto the closure of some unbounded component of \( N \setminus D \). 
\end{corollary}

As an application of \Cref{cut}, we now prove the following theorem. 
\begin{theorem}\label{deg}
 Suppose $\deg(f)=1$ and that the restriction $f\vert \partial M\to \partial N$ is a homeomorphism. If $\underline\pi_0(f)$ is injective,  and for every end $e$ of $M$, there exists a representative $a$ of $e$ such that $\underline\pi_1(f)\colon \underline\pi_1(M,\underline a)\to \underline \pi_1(N,\underline{fa})$ is a monomorphism, then either $f$ has a geometric kernel or $f$ is properly homotopic to a homeomorphism. 
\end{theorem}

The proof of this theorem relies on the following result of Edmonds.

\begin{theorem}[Nielsen--Edmonds classification of allowable maps of degree one\ \textnormal{\cite[Theorem~4.1]{MR541331}}]\label{allow}
    Let \(F\) and \(G\) be connected, oriented, compact \(2\)-manifolds, and let \(\varphi \colon (F, \partial F) \to (G, \partial G)\) be a degree-one map such that \(\varphi^{-1}(\partial G) = \partial F\) and \(\varphi\vert \partial F \to \partial G\) is a homeomorphism. If $\ker \pi_{1}(\varphi) = 0$, then $\varphi$ is homotopic rel $\partial F$ to a homeomorphism; otherwise, there exists an essential subsurface $S_{g,1} \subseteq F$, with $g \geq 1$, such that $\varphi$ is homotopic rel $\partial F$ to the quotient map $F \to F/S_{g,1}$ that collapses $S_{g,1}$ to a point.

\end{theorem}
Note that, without loss of generality, the condition \( \varphi^{-1}(\partial G) = \partial F \) can be omitted from \Cref{allow} by the following lemma, whose proof is given after that of \Cref{deg}.
\begin{lemma}\label{nearboundary}
    Let \( M \) and \( N \) be manifolds with collars \( [-2,0] \times \partial M \subseteq M \) and \( [-2,0] \times \partial N \subseteq N \), where \( \{0\} \times \partial M \equiv \partial M \) and \( \{0\} \times \partial N \equiv \partial N \). Suppose \( f \colon (M, \partial M) \to (N, \partial N) \) is a map. Then there exists a homotopy \( H\colon M\times [0,1]\to N \) rel \( \partial M \) from \( f \) to a map \( g \) such that \( g^{-1}((-1,0] \times \partial N) = (-1,0] \times \partial M \), and \( g(x,t) = (f(x), t) \quad \text{for all } (x,t) \in (-1,0] \times \partial M. \)
\end{lemma}We are now in a position to prove \Cref{deg}.
\begin{proof}[Proof of \Cref{deg}]
    By \Cref{cut}, there exist compact \(2\)-dimensional submanifolds \( D' \subset M \) and \( D \subset N \), along with a proper map \( g \), such that \( f \) is properly homotopic to \( g \) relative to \( \partial M \), and the restriction of \( g \) to the closure of each unbounded component of \( M \setminus D' \) is a finite-sheeted covering map onto the closure of some unbounded component of \( N \setminus D \). In particular, \( \deg(g) = 1 \), \( \underline{\pi}_0(g) \) is injective, and the map \( g\vert\partial M \to \partial N \) is a homeomorphism. Since a proper map between manifolds is closed~\cite{MR254818}, \( g \) must be surjective; otherwise, there would exist a disk \( \mathfrak{D} \subseteq N \setminus \partial N \) such that \( g^{-1}(\mathfrak{D}) = \varnothing \), which, by~\cite[Lemma~2.1(b)]{MR192475}, would imply \( \deg(g) = 0 \). Hence, \( \underline{\pi}_0(g) \) is a bijection.
    
    Let \( D_0' \) be the \(2\)-dimensional submanifold obtained by taking the union of \( D' \) and all bounded components of \( M \setminus D' \). Thus, every component of \( D_0' \) is essential, and every component of \( M \setminus D_0' \) is unbounded. On the other hand, \( D \) is connected and essential, and every component of \( N \setminus D \) is unbounded (see \Cref{unbounded}). Let \( U_1, \dots, U_n \) be all the components of \( N \setminus D \). Since \( \underline{\pi}_0(g) \) is injective, \( g^{-1}(U_i) \) must contain exactly one component, say \( U_i' \), of \( M \setminus D_0' \) for each \( i \). Thus, \( U_1', \dots, U_n' \) are all the components of \( M \setminus D_0' \).

    Consider any \( i \in \{1, \dots, n\} \). By our hypothesis, the map \( g\vert \operatorname{cl}(U_i') \to \operatorname{cl}(U_i) \) is a finite-sheeted covering map, say \( k_i \)-sheeted. Thus, there exists a small disk \( \mathfrak{D}_i \subseteq U_i \setminus (\partial N \cup g(D_0')) \) such that \( g^{-1}(\mathfrak{D}_i) \) is a pairwise-disjoint union of \( k_i \) disks in \( M \setminus (D_0' \cup \partial M) \), and \( g \) restricted to each component of \( g^{-1}(\mathfrak{D}_i) \) is an orientation-preserving homeomorphism onto \( \mathfrak{D}_i \), which means \( \deg(g) = k_i \) by \cite[Lemma~2.1(b)]{MR192475}. But \( \deg(g) = 1 \). Thus, \( k_i = 1 \), and hence \( g\vert \operatorname{cl}(U_i') \to \operatorname{cl}(U_i) \) is a homeomorphism.
    
    Let \(\{C_j\}\) be an exhausting sequence for \(N\) satisfying the five properties \ref{(i)}--\ref{(v)} of \Cref{exhaustion}. Choose \(j_0\) sufficiently large so that \(\operatorname{int}(C_{j_0})\) contains \(D \cup g(D_0')\). Then the map \(g\vert g^{-1}(\operatorname{cl}(N \setminus C_{j_0})) \to \operatorname{cl}(N \setminus C_{j_0})\) is a homeomorphism. Define \(C_{j_0}' \coloneqq \operatorname{cl}(M \setminus g^{-1}(\operatorname{cl}(N \setminus C_{j_0})))\). Then \(C_{j_0}'\) is a compact 2-dimensional submanifold of \(M\) such that \(g(C_{j_0}') \subseteq C_{j_0}\), and the map \(g\vert \operatorname{fr}(C_{j_0}') \to \operatorname{fr}(C_{j_0})\) is a homeomorphism. Hence, \(\operatorname{fr}(C_{j_0}') \cap \partial M = \varnothing\), and \(g\vert \partial C_{j_0}' \to \partial C_{j_0}\) is a homeomorphism.

    We claim that \(C_{j_0}'\) is connected and that \(g\vert C_{j_0}' \to C_{j_0}\) is a map of degree one. To prove this claim, let \(X'\) be a component of \(C_{j_0}'\). Then \(g\) maps each component of \(\partial X'\) homeomorphically onto a component of \(\partial C_{j_0}\). Now, the following commutative diagram \[
        \begin{tikzcd}
            H_2(X',\partial X') \arrow[d, "H_2(g\vert X')"'] \arrow[r] & H_1(\partial X') \arrow[d, "H_1(g\vert \partial X')"] \\
            H_2(C_{j_0}, \partial C_{j_0}) \arrow[r] & H_1(\partial C_{j_0})
        \end{tikzcd}
    \]where the horizontal maps are of the form \( \mathbb{Z} \ni 1 \mapsto \oplus_{i=1}^m 1 \in \oplus_{i=1}^m \mathbb{Z} \), implies that both \(H_2(g\vert X')\) and \(H_1(g\vert \partial X')\) are isomorphisms. In particular, \(g\vert \partial X' \to \partial C_{j_0}\) is a homeomorphism, and hence \(X' = C_{j_0}'\). Moreover, since \(H_2(g\vert X')\) is an isomorphism, the map \(g\vert C_{j_0}' \to C_{j_0}\) is of degree one.

    Suppose \( \ker \pi_1(f) = 0 \). Then \( g\vert C_{j_0}' \to C_{j_0} \) can be homotoped rel \( \partial C_{j_0}' \) to a homeomorphism \( h\colon C_{j_0}' \to C_{j_0} \). By pasting \( h \) with the homeomorphism \( g\vert \operatorname{cl}(M \setminus C_{j_0}') \to \operatorname{cl}(N \setminus C_{j_0}) \), we obtain a homeomorphism \( \overline{h} \colon M \to N \) such that \( \overline{h} \) is properly homotopic to \( f \) rel \( \partial M \). 

    On the other hand, suppose \( \ker \pi_1(f) \ne 0 \). Then, by \Cref{allow} and \Cref{nearboundary}, the map \( g\vert \partial C_{j_0}' \to \partial C_{j_0} \) can be homotoped rel \( \partial C_{j_0}' \) to send an essential handle to a point; that is, \( g \) has a geometric kernel, and hence \( f \) also has a geometric kernel.
\end{proof}
We now prove \Cref{nearboundary}, closely following the argument of \cite[Homotopic-to-Product-Near-Boundary Lemma]{friedlbook}.
\begin{proof}[Proof of \Cref{nearboundary}] Define a homotopy \( G \colon N \times [0,1] \to N \) by $$
G(q, s) \coloneqq \begin{cases}
\left(-s, \tilde{q}\right) & \text{if } q = \left(t, \tilde{q}\right) \in \left[-s, 0\right] \times \partial N, \\
q & \text{if } q \in N \setminus \left(\left(-s, 0\right] \times \partial N\right).\end{cases}$$ Using \( G \), we now define a map \( H \colon M \times [0,1] \to N \) as follows:

$$H(p, s) \coloneqq 
\begin{cases}
\left(t, f(0, \tilde{p})\right) & \text{if } p = (t, \tilde{p}) \text{ with } t \in \left[-s, 0\right], \\
G\left(f\left(\frac{2(t+s)}{2-s}, \tilde{p}\right),s\right) & \text{if } p = (t, \tilde{p}) \text{ with } t \in \left[-2, -s\right], \\
G\left(f(p),s\right) & \text{if } p \notin \left(-2, 0\right] \times \partial M.
\end{cases}$$ 

Let \( g \coloneqq H(-,1) \). Then \( H \) is a homotopy rel \( \partial M \) from \( f \) to \( g \), and \( g(t,p) = (t, f(p)) \) for all \( (t,p) \in [-1,0] \times \partial M \). Since \( (-1,0] \times \partial N \) is disjoint from \( G(N \times \{1\}) \), it follows that \( g^{-1}((-1, 0] \times \partial N) = (-1, 0] \times \partial M \).
\end{proof}

Combining \Cref{iso} with the proof of \Cref{deg}, we obtain the following.
\begin{corollary}
    Suppose \( \partial M = \varnothing = \partial N \). If \( \underline\pi_0(f) \) is injective and, for every end \( e \) of \( M \), there exists a representative \( a \) of \( e \) such that \( \underline\pi_1(f) \colon \underline\pi_1(M, \underline{a}) \to \underline\pi_1(N, \underline{f a}) \) is an isomorphism, then \( f \) is properly homotopic to a map that either collapses an essential subsurface \( S_{g,1} \subset M \), with \( g \geq 1 \), to a point, or is a homeomorphism.
\end{corollary}
The remainder of this section is devoted to proving the analogue of \Cref{cut} in the $3$-dimensional setting, where the assumption of end-irreducibility becomes essential. Unlike in dimension two—where all boundaryless connected $2$-manifolds are end-irreducible—this property does not hold in dimension three. In fact, there exists a contractible open subset of $\mathbb{R}^3$ that is not end-irreducible—not even eventually end-irreducible \cite[Figure~1]{MR511426}. A connected $3$-manifold $P$ is said to be \emph{eventually end-irreducible} \cite[p.~504]{MR511426} if there exists a connected compact $3$-dimensional submanifold $C \subset P$ such that each component of $\operatorname{cl}(P \setminus C)$ is end-irreducible. On the other hand, for example, the interior of any connected, compact, boundary-irreducible $3$-manifold is end-irreducible.

The analogue of \Cref{cut} is \Cref{cutoff3}. A special case of \Cref{cutoff3} was proved by Brown \cite[Theorem~2.4]{MR511426}. In fact, Brown’s result concerns $3$-manifolds that are boundaryless, have exactly one end, and are only eventually end-irreducible. For simplicity, we omit the word “eventually” from our hypotheses.

\begin{theorem}\label{cutoff3}
Let \(P\) and \(Q\) be connected, irreducible, boundary-irreducible, end-irreducible, non-compact $3$-manifolds such that every component of \(\partial P\) is compact. Suppose there exists a proper map \(f \colon (Q, \partial Q) \to (P, \partial P)\) with \(f\vert\partial Q\) a local homeomorphism, and for every end \(e\) of \(Q\), there exists a representative \(a\) of \(e\) such that \(\underline{\pi}_1(f) \colon \underline{\pi}_1(Q, \underline{a}) \to \underline{\pi}_1(P, \underline{f a})\) is a monomorphism (resp. an isomorphism). Then there exist compact $3$-dimensional submanifolds \(D' \subset Q\) and \(D \subset P\) such that, after a proper homotopy of \(f\) rel \(\partial Q\), for each unbounded component \(U'\) of \(Q \setminus D'\), there exists an unbounded component \(U\) of \(P \setminus D\) for which \(f\vert\operatorname{cl}(U') \to \operatorname{cl}(U)\) is a finite-sheeted covering map (resp. a homeomorphism). 
\end{theorem}
\begin{proof}
    As in the first paragraph of the proof of \Cref{specialcase}, one sees that each component of \(\partial Q\) is compact. Let \(\{C_n\}\) be an exhausting sequence for \(P\) satisfying the properties listed just before \Cref{exhaustion}, as guaranteed by \cite[Lemma~3.1]{MR334225}. Using the techniques developed in the first five paragraphs of the proof of \cite[Theorem~4.2]{MR334225}, we may properly homotope \(f\) rel \(\partial Q\) to a map \(f'\) such that each component of \(f'^{-1}(\bigcup_n \operatorname{fr}(C_n))\) is an incompressible surface in \(Q \setminus \partial Q\), and the restriction of \(f'\) to each component of \(f'^{-1}(\bigcup_n \operatorname{fr}(C_n))\) is either constant or a covering map. For each \(n\), define \(A_{n,1}, \dots, A_{n,k_n}\) as in the proof of \Cref{cut}. Then, by a similar argument, there exists a positive integer \(\underline{n}\) such that, for each \(j = 1, \dots, k_{\underline{n}}\), the map \(f'\vert A_{\underline{n},j}\) is \(\pi_1\)-injective and \(f\vert \partial A_{\underline{n},j}\) is a local homeomorphism. Since \(\partial A_{\underline{n},j} \neq \varnothing\), it follows from \cite[Theorem~4.2(a)]{MR334225} that \(f'\vert A_{\underline{n},j}\) is properly homotopic rel \(\partial A_{\underline{n},j}\) to a finite-sheeted covering map, for each \(j = 1, \dots, k_{\underline{n}}\). Finally, an argument similar to the one given at the end of the proof of \Cref{cut} completes the proof.
\end{proof}

\section{Free Homotopy Classes of Sequences of Loops Near an End}\label{freehomotopy}
This section aims to characterize the conjugacy classes in Brown’s proper fundamental group and to apply this characterization in the study of geometric kernels. A key advantage of this approach is that the characterization depends solely on the end of the surface, and not on the choice of a representing ray—unlike the proper fundamental group. For simplicity, we assume throughout that all \( 2 \)-manifolds under consideration are boundaryless. However, most of the results extend to the case where all boundary components are compact and the maps restrict to local homeomorphisms on the boundary.

Let \( S \) be a non-compact surface, and let \( e \) be an end of \( S \). For each \( k \in \mathbb{N}\cup\{0\} \), define \(\mathbb{S}_k^1 \coloneqq \{ (x,y) \in \mathbb R^2 \mid (x-k)^2+(y-1/3)^2=1/9 \}\). Suppose \(\alpha\) is a proper map from \(\mathbb S_\infty^1\coloneqq\bigcup_{k\geq 0} \mathbb{S}_k^1\) to \(S\), and let \(\alpha_k\) be the restriction of \(\alpha\) to \(\mathbb{S}_k^1\). Thus, each \( \alpha_k \) is a loop based at \( \alpha_k(k,0) \). We write \(\alpha = (\alpha_0,\alpha_1, \alpha_2, \dots)\), and call \(\alpha\) a \emph{sequence of loops}.  Given a sequence \(\{n_0, n_1, n_2, \dots\}\) of integers, the tuple \((\alpha_0^{n_0},\alpha_1^{n_1}, \alpha_2^{n_2}, \dots)\) is called the \emph{\((n_0, n_1, n_2, \dots)\)-th power of \(\alpha\)}. Here, for each \( k \), \( \alpha_k^{n_k} \) denotes the loop based at \( \alpha_k(k,0) \) obtained by concatenating \( |n_k| \) copies of \( \alpha_k \) (if \( n_k > 0 \)) or its inverse loop \( \overline{\alpha_k} \) (if \( n_k < 0 \)), with \( \alpha_k^0 \) defined to be the constant loop at \( \alpha_k(k,0) \). We say \(\alpha\) \emph{converges to \(e\)} if \(\{\operatorname{im} \alpha_n\}\) converges to \(e\). We say that \(\alpha\) \emph{bounds \(e\)} if \(\alpha\) converges to \(e\) and, for all \(k\), \(\operatorname{im}(\alpha_k)\) is a non-trivial separating circle in \(S\) with \(A_k\) the component of \(S \setminus \operatorname{im}(\alpha_k)\) satisfying \(e \in \underline{A_k}\), and \(\operatorname{cl}(A_{k+1}) \subset A_k\).

%one component of \(S \setminus \alpha_k\) contains \(\alpha_{k+1} \cup \alpha_{k+2} \cup \cdots\).

If \(\alpha\) and \(\beta\) are two sequences of loops converging to \(e\), we say they are \emph{freely homotopic near \(e\)} if there exists a proper map \(H\colon \mathbb S_\infty^1 \times [0,1] \to S\) such that \(\underline{H_0} = \underline{\alpha}\) and \(\underline{H_1} = \underline{\beta}\). This defines an equivalence relation on the set of all sequences of loops converging to \(e\), and the set of equivalence classes is denoted by \(\widehat{\underline{\pi}_1}(S,e)\), called the \emph{free homotopy classes of sequences of loops near the end \(e\)}. The phrase ``free homotopy near $e$'' instead of just ``free homotopy'' is justified by the following proposition.

\begin{proposition}\label{unnecessary}
    Let \(\alpha\) and \(\beta\) be two sequences of loops converging to \(e\). If there exists a proper map \(H\colon \mathbb S_\infty^1 \times [0,1] \to S\) such that \(\underline{H_0} = \underline{\alpha}\) and \(\underline{H_1} = \underline{\beta}\), then \(\{ H(\mathbb{S}_k^1 \times [0,1]) \}\) converges to \(e\). 
\end{proposition}
\begin{proof}
    Choose a compact subset $C$ of $S$. Let $A$ be the unbounded component of $S\setminus C$ such that $e\in \underline A$. Since $H$ is proper, there exists $k_0\in \mathbb N$ such that $H$ sends $\bigcup_{k\geq k_0}\mathbb S^1_k\times [0,1]$ into $S\setminus C$. Without loss of generality, we may assume that $H_0\vert\bigcup_{k\geq k_0}\mathbb S^1_n=\alpha\vert\bigcup_{k\geq k_0}\mathbb S^1_n$ and $H_1\vert\bigcup_{k\geq k_0}\mathbb S^1_k=\beta\vert\bigcup_{k\geq k_0}\mathbb S^1_n$.  Since $\alpha$ and $\beta$ both converge to $e$, we may further assume that $\alpha(\mathbb S^1_k)\cup \beta(\mathbb S^1_k)\subseteq A$ for all $k\geq k_0$. By the continuity of $H$, the image $H(\mathbb S^1_k\times [0,1])$  must lie within a component $U_k$ of $S\setminus C$ for each $k\geq k_0$. Since $H_0(\mathbb S^1_k)\cup H_1(\mathbb S^1_k)=\alpha(\mathbb S^1_k)\cup \beta(\mathbb S^1_k)\subseteq A$ for all $k\geq k_0$, we conclude that $U_k=A$ for all $k\geq k_0$.
\end{proof}

Notice that, without loss of generality, we may assume \( \underline{\mathbb{S}^1} = ([0,\infty) \times \mathbb{R}) \cup \mathbb{S}^1_\infty \). Let \( a\colon [0,\infty) \to S \) be a proper map representing the end \( e \). Then, every proper map of pairs \( (\underline{\mathbb{S}^1}, \underline{*}) \to (S, \underline{a}) \) when restricted to \( \mathbb{S}^1_\infty \) is a sequence of loops converging to \( e \). This determines a well-defined map
\( \Phi \colon \underline{\pi}_1(S, \underline{a}) \to \widehat{\underline{\pi}_1}(S,e) \). We show that \( \Phi \) induces a bijection between the set of conjugacy classes in \( \underline{\pi}_1(S, \underline{a}) \) and \( \widehat{\underline{\pi}_1}(S,e) \) by applying the following straightforward lemmas.

\begin{lemma}\label{lemmapre}
Let \( X \) be a path-connected space, and let \( x_0 \) and \( x_1 \) be points of \( X \). Then, for every loop \( \ell \colon (\mathbb{S}^1, 1) \to (X, x_1) \), there exists a loop \( \delta \colon (\mathbb{S}^1, 1) \to (X, x_0) \) such that \( \ell^n \) is freely homotopic to \( \delta^n \) for every integer \( n \).
\end{lemma}

\begin{lemma}\label{lemma}
Let \((X, x_0)\) be a based space, and let \(\alpha, \beta \colon (\mathbb{S}^1, 1) \to (X, x_0)\) be loops. Then \(\alpha\) and \(\beta\) are freely homotopic if and only if there exists a loop \(\gamma \colon (\mathbb{S}^1, 1) \to (X, x_0)\) such that \(\beta\) is homotopic to \(\gamma * \alpha * \overline{\gamma}\) relative to the basepoint.
\end{lemma}

We now use the preceding lemmas to establish the intended result.

\begin{theorem}\label{conjugacy}
    The map $\Phi\colon\underline\pi_1(S,\underline a)\to \widehat {\underline{\pi}_1}(S,e)$ is a surjection. Moreover, \(\Phi(x) = \Phi(y)\) for \(x, y \in \underline{\pi}_1(S, \underline{a})\) if and only if there exists \(z \in \underline{\pi}_1(S, \underline{a})\) such that \(y = z x z^{-1}\).
\end{theorem}
\begin{proof}
    Let \( a \) be an element of \( \widehat{\underline{\pi}_1}(S, e) \). Suppose \( \alpha = (\alpha_0, \alpha_1, \alpha_2, \ldots) \) is a sequence of loops converging to \( e \) such that \( a = [\alpha] \). Choose a decreasing sequence \( A_0 \supseteq A_1 \supseteq A_2 \supseteq \cdots \) of path-connected, unbounded open subsets of \( S \) such that \( \bigcap_k A_k = \varnothing \), \( e \in \underline{A_k} \) for all \( k \), and \( \{a(k)\} \cup \operatorname{im}(\alpha_k) \subseteq A_k \) for all sufficiently large \( k \). This is possible because the sequence \( \{\{a(k)\} \cup \operatorname{im}(\alpha_k)\} \) converge to \( e \). By \Cref{lemmapre}, for all sufficiently large \( k \), there exists a homotopy \( f_k \colon \mathbb{S}^1_k \times [0,1] \to A_k \) from \( \alpha_k \) to a loop \( \delta_k \colon (\mathbb{S}^1_k, (k,0)) \to (A_k, a(k)) \). Let \( \Delta \colon (\underline{\mathbb{S}^1}, \underline{*}) \to (S, \underline{a}) \) be a proper map of pairs such that \( \Delta\vert\mathbb{S}^1_k = \delta_k \) for all sufficiently large \( k \). Then, the homotopies \( f_k \) collectively give a free homotopy near \( e \) from \( \alpha \) to \( \Delta\vert\mathbb{S}^1_\infty \). Hence, \( \Phi \) sends \( [\Delta] \in \underline{\pi}_1(S, \underline{a}) \) to \( a = [\alpha] \in \widehat{\underline{\pi}_1}(S, e) \). Since \( a \) was arbitrary, \( \Phi \) is surjective.

    Now, let \(x\) and \(y\) be two elements of \(\underline{\pi}_1(S, \underline{a})\). Choose representatives \(\alpha_x\colon (\underline{\mathbb{S}^1}, \underline{*}) \to (S, \underline{a})\) and \(\alpha_y\colon (\underline{\mathbb{S}^1}, \underline{*}) \to (S, \underline{a})\) for \(x\) and \(y\), respectively. We will show that \(\Phi(x) = \Phi(y)\) if and only if \(y = z x z^{-1}\) for some \(z \in \underline{\pi}_1(S, \underline{a})\).

    First, suppose \( y = z x z^{-1} \) for some \( z \in \underline{\pi}_1(S, \underline{a}) \). Choose a representative \( \alpha_z \colon (\underline{\mathbb{S}^1}, \underline{*}) \to (S, \underline{a}) \) of \( z \). Then the relation \( y = z x z^{-1} \) implies that for all sufficiently large \( k \), there exists a homotopy \( h_k \colon \mathbb{S}^1_k \times [0,1] \to U_k \) rel \( (k,0) \) from \( \alpha_y \vert \mathbb{S}^1_k \) to \( (\alpha_z \vert \mathbb{S}^1_k) * (\alpha_x \vert \mathbb{S}^1_k) * \overline{(\alpha_z \vert \mathbb{S}^1_k)} \), where \( U_k \) is a path-connected, unbounded open subset of \( S \) with \( e \in \underline{U_k} \). Moreover, we may assume \( U_k \supseteq U_{k+1} \) for all sufficiently large \( k \) and \( \bigcap_k U_k = \varnothing \). Thus, for all sufficiently large \( k \), by \Cref{lemma}, there exists a free homotopy \( \widehat{h_k} \colon \mathbb{S}^1_k \times [0,1] \to U_k \) from \( \alpha_x \vert \mathbb{S}^1_k \) to \( \alpha_y \vert \mathbb{S}^1_k \). Now, these homotopies \( \widehat{h_k} \) collectively provide a free homotopy near \( e \) from \( \alpha_x \vert \mathbb{S}^1_\infty \) to \( \alpha_y \vert \mathbb{S}^1_\infty \). Therefore, \( \Phi(x) = \Phi(y) \). This completes the proof of the ``if" direction.

    We now prove the ``only if'' direction. So, assume \(\Phi(x) = \Phi(y)\). Then there exists a proper map \(H \colon \mathbb{S}^1_\infty \times [0,1] \to S\) such that the germ of \(H(-,0)\) coincides with the germ of \(\alpha_x \vert \mathbb{S}^1_\infty\), and the germ of \(H(-,1)\) coincides with the germ of \(\alpha_y \vert \mathbb{S}^1_\infty\). Choose a decreasing sequence \(V_0 \supseteq V_1 \supseteq V_2 \supseteq \cdots\) of path-connected, unbounded open subsets of \(S\) such that \(\bigcap_k V_k = \varnothing\), \(e \in \underline{V_k}\) for all \(k\), and \(H(\mathbb{S}^1_k \times [0,1]) \subseteq V_k\) for all sufficiently large \(k\). This is possible because the sequence \(\{H(\mathbb{S}^1_k \times [0,1])\}\) converges to \(e\); see \Cref{unnecessary}. Hence, by \Cref{lemma}, for all sufficiently large \(k\), there exists a homotopy \(g_k \colon \mathbb{S}^1_k \times [0,1] \to V_k\) rel \((k,0)\) from \(\gamma_k * (\alpha_x \vert \mathbb{S}^1_k) * \overline{\gamma_k}\) to \(\alpha_y \vert \mathbb{S}^1_k\), where \(\gamma_k \colon \mathbb{S}^1_k \to V_k\) is a loop with \(\gamma_k(k,0) = a(k,0)\). Let \(\Gamma\) and \(\overline{\Gamma}\) be proper maps of pairs \((\underline{\mathbb{S}^1}, \underline{a}) \to (S, \underline{a})\) such that \(\Gamma \vert \mathbb{S}^1_k = \gamma_k\) and \(\overline{\Gamma} \vert \mathbb{S}^1_k = \overline{\gamma_k}\) for all sufficiently large \(k\). Then the homotopies \(g_k\) collectively give a germ homotopy rel \(\underline{*}\) from \(\Gamma \cdot \alpha_x \cdot \overline{\Gamma}\) to \(\alpha_y\). Thus, \([\Gamma] x [\Gamma]^{-1} = y\) in the group \(\underline{\pi}_1(S, \underline{a})\). This completes the proof of the ``only if'' direction.
\end{proof}

In the remainder of this section, we use \( \widehat{\underline{\pi}_1} \) to study geometric kernels of proper maps. The first theorem provides a sufficient condition under which a degree one map between non-planar surfaces with the same finite number of ends can be properly homotoped to pinch at least one handle—possibly infinitely many—and, in particular, can have a geometric kernel. An example of a map satisfying this sufficient condition is provided in the introduction, immediately following \Cref{infinitepinch}.

\begin{theorem}\label{infinitepinchproof}
     Let \( S' \) and \( S \) be non-compact surfaces, each with finitely many ends, such that \( S' \) is non-planar and \( S \neq \mathbb{R}^2 \). Suppose \( f \colon S' \to S \) is a proper map satisfying the following: 
\begin{enumerate}[\((\mathrel{\star})\)]
    \item\label{tc} For every end \( e \) of \( S \), there exists a unique end \( e' \) of \( S' \) such that \( \underline{\pi}_0(f)(e') = e \), and there exists a sequence \( \alpha \) of separating circles bounding \( e \) such that the preimage of the free homotopy class of every power of \( \alpha \) under the induced map \( \widehat{\underline{\pi}_1}(f) \colon \widehat{\underline{\pi}_1}(S', e') \to \widehat{\underline{\pi}_1}(S, e) \) is a singleton.
\end{enumerate}
If \( \ker \pi_1(f) \neq 0 \), then there exists a pairwise disjoint collection \( \{ h_i' \colon i \in \mathscr{A} \} \) of essential handles in \( S' \), with \( 1 \leq |\mathscr{A}| \leq \aleph_0 \), and a proper map \( g \colon S' \to S \) properly homotopic to \( f \) such that \( g(h_i') \) is a point for every \( i \in \mathscr{A} \). In fact, this collection is infinite, i.e. \( |\mathscr{A}| = \aleph_0 \), if and only if there exists an end \( e' \) of \( S' \) for which the induced map \( \widehat{\underline{\pi}_1}(f) \colon \widehat{\underline{\pi}_1}(S', e') \to \widehat{\underline{\pi}_1}(S, e) \) is not injective, where \( e = \underline{\pi}_0(f)(e') \).
\end{theorem}
The following lemma will be applied several times, including in the proof of the above theorem. Although its hypothesis can be weakened without affecting the conclusion, we state it here in the restricted form sufficient for our purposes.
\begin{lemma}\label{boundaryhomeo}
   Let \( M \) and \( N \) be oriented, non-compact, bordered surfaces such that \( \partial M \) and \( \partial N \) are each homeomorphic to \( \mathbb{S}^1 \). Suppose \( f \colon (M, \partial M) \to (N, \partial N) \) is a proper map whose restriction \( f\vert\partial M \to \partial N \) is a homeomorphism. Then the degree of \( f \) is \( +1 \) (resp.\ \( -1 \)) if \( f\vert\partial M \) preserves (resp.\ reverses) the induced orientations on the boundaries.
\end{lemma}
\begin{proof}
   Observe that the homotopy \( H \) constructed in the proof of \Cref{nearboundary} is relative to the complement of the compact subset \( f^{-1}([-2,0]\times \partial N) \cup ([-2,0]\times \partial M) \) of \( M \). Thus, \( H \) is proper. Moreover, \( H \) is relative to \( \partial M \), and the restriction \( g\vert g^{-1}([-1/2,0]\times \partial N)=[-1/2,0]\times \partial M \to [-1/2,0]\times \partial N \) is a homeomorphism. Thus, if $\mathfrak D$ is a disk in $(-1/2,0)\times \partial N$, then $g\vert g^{-1}(\mathfrak D)\to \mathfrak D$ is homeomorphism. It then follows from \cite[Lemma~2.1(b)]{MR192475} that \( \deg(g) = \pm 1 \). Hence, \( \deg(f) = \pm 1 \), since degree is invariant under proper homotopy relative to the boundary of a manifold. Consider the induced orientations on the annuli \( A' \coloneqq [-1/2,0] \times \partial M \) and \( A \coloneqq [-1/2,0] \times \partial N \) obtained from the orientations of \( M \) and \( N \), respectively. By the naturality of the homology long exact sequence, the homeomorphism \( f\vert\partial M = g\vert\partial M \) preserves (resp.\ reverses) the induced orientations on the boundaries if and only if the degree of the map \( g\vert A' \to A \) is \( +1 \) (resp.\ \( -1 \)). Moreover, the degree of \( g\vert A' \to A \) is \( +1 \) (resp.\ \( -1 \)) if and only if the restriction \( g\vert g^{-1}(\mathfrak{D}) \to \mathfrak{D} \) is orientation-preserving (resp.\ orientation-reversing).
\end{proof}

\begin{proof}[Proof of \Cref{infinitepinchproof}]
    Let $\{e'_1,\ldots,e'_n\}$ and $\{e_1,\ldots,e_n\}$ denote the sets of all ends of $S'$ and $S$, respectively, where $\underline{\pi}_0(f)(e'_i)=e_i$ for all $i=1,\ldots,n$. For each $i=1,\ldots,n$, fix a sequence $\alpha_i=(\alpha_{i0},\alpha_{i1},\alpha_{i2},\ldots)$ of separating non-trivial circles in $S$ bounding $e_i$, such that for every sequence of integers $\{n_0,n_1,n_2,\ldots\}$, the preimage under $\widehat{\underline{\pi}_1}(f)$ of $[(\alpha_{i0}^{n_0},\alpha_{i1}^{n_1},\alpha_{i2}^{n_2},\ldots)] \in \widehat{\underline{\pi}_1}(S,e_i)$ is a singleton subset of $\widehat{\underline{\pi}_1}(S',e'_i)$. There exists an exhausting sequence $\{C_k : k=0,1,2,\ldots\}$ for $S$ satisfying properties \ref{(i)}--\ref{(v)} of \Cref{exhaustion}, such that $\partial C_k=\bigcup_{i=1}^n \operatorname{im}(\alpha_{ik})$ for all sufficiently large $k$. Hence, for all sufficiently large $k$, the set $\operatorname{cl}(C_{k+1}\setminus C_k)$ has exactly $n$ components, say $D_{1k},\ldots,D_{nk}$, where $\partial D_{ik}=\operatorname{im}(\alpha_{ik}) \cup \operatorname{im}(\alpha_{i(k+1)})$ for $i=1,\ldots,n$. As in the proof of \Cref{specialcase}, after a proper homotopy we may assume that each component of $\bigcup_k f^{-1}(\partial C_k)$ is a non-trivial circle in $S'$, and that the restriction of $f$ to each component of $\bigcup_k f^{-1}(\partial C_k)$ is either constant or a finite-sheeted covering. Moreover, $\{f^{-1}(C_k)\}$ is an exhausting sequence for $S'$.

    Fix \(i\in\{1,\ldots,n\}\). Because \(\underline{\pi}_0(f)\) is surjective and \(\{\operatorname{im}(\alpha_{ik})\}\) is a sequence of separating non-trivial circles converging to \(e_i\), we have \(f^{-1}(\operatorname{im}(\alpha_{ik}))\neq\varnothing\) for all sufficiently large \(k\). Hence there exists a sequence \(\alpha_i'=(\alpha_{i0}',\alpha_{i1}',\alpha_{i2}',\ldots)\) of loops converging to \(e'_i\) with the property that, for all sufficiently large \(k\), \(\alpha_{ik}'\) is simple and \(\operatorname{im}(\alpha_{ik}')\) is a component of \(f^{-1}(\operatorname{im}(\alpha_{ik}))\). By the previous paragraph, there exists a sequence $\{n_{i0},n_{i1},n_{i2},\ldots\}$ of integers such that, for all sufficiently large $k$, possibly after re-parametrizing $\alpha_{ik}'$, we have $f\alpha_{ik}'=\alpha_{ik}^{\,n_{ik}}$. Using \ref{tc}, there exists a sequence of loops $\beta'_i=(\beta_{i0}',\beta_{i1}',\beta_{i2}',\ldots)$ converging to $e'_i$ such that $\widehat{\underline{\pi}_1}(f)([\beta'_i])=[\alpha_i]$. Then $\widehat{\underline{\pi}_1}(f)$ sends both $[(\beta_{i0}'^{\,n_{i0}},\beta_{i1}'^{\,n_{i1}},\beta_{i2}'^{\,n_{i2}},\ldots)]$ and $[\alpha'_i]$ to $[(\alpha_{i0}^{\,n_{i0}},\alpha_{i1}^{\,n_{i1}},\alpha_{i2}^{\,n_{i2}},\ldots)]$. Again, by \ref{tc}, $\beta_{ik}'^{\,n_{ik}}$ must be freely homotopic to $\alpha_{ik}'$ for all sufficiently large $k$. Using \Cref{lemmapre} and \Cref{lemma}, we see that for all sufficiently large $k$, the nontrivial element of $\pi_1(S',*)$ represented by $\alpha_{ik}'$ is the $n_{ik}$th power of some element of $\pi_1(S',*)$. By \cite[Theorems~1.7 and~4.2]{MR214087}, it follows that $n_{ik}=\pm 1$, and hence $f\vert\operatorname{im}(\alpha_{ik}')\to\operatorname{im}(\alpha_{ik})$ is a homeomorphism for all sufficiently large $k$. Thus, after re-parametrizing $\alpha_{ik}'$, we may assume $f\alpha_{ik}'=\alpha_{ik}$ for all sufficiently large $k$. Moreover, for all sufficiently large $k$, any two distinct components of $f^{-1}(\operatorname{im}(\alpha_{ik}))$ must co-bound an annulus in $S'$; otherwise, there would exist a sequence $\{\gamma_{im_j}':m_1<m_2<\cdots\}$ of simple loops in $S'$ such that for all $j$, $f\gamma_{im_j}'=\alpha_{im_j}$ but $\gamma'_{im_j}$ is not freely homotopic to $\alpha_{im_j}'$. This would imply that $\widehat{\underline{\pi}_1}(f)^{-1}([\alpha])$ contains at least two elements, contradicting \ref{tc}.

    We repeat the above process for all $i$. Since $S$ has finitely many ends, there exists a positive integer $k_0$ such that for any $i\in \{1,\ldots,n\}$ and any $k\geq k_0$, the following hold: $f^{-1}(\operatorname{im}(\alpha_{ik}))\neq \varnothing$, any two distinct components of $f^{-1}(\operatorname{im}(\alpha_{ik}))$ co-bound an annulus in $S'$, the restriction of $f$ to every component of $f^{-1}(\operatorname{im}(\alpha_{ik}))$ is a homeomorphism onto $\operatorname{im}(\alpha_{ik})$, and $\partial C_k=\bigcup_{i=1}^n\operatorname{im}(\alpha_{ik})$.

    \begin{claim}\label{claim}
        There exists \(k_0' \geq k_0\) such that, for every \(i \in \{1, \dots, n\}\) and every \(k \geq k_0'\), each component of \(f^{-1}(\operatorname{im}(\alpha_{ik}))\) separates \(S'\). Hence, without loss of generality, we henceforth assume that \(k_0 = k_0'\).
    \end{claim}
    \begin{proof}[Proof of \Cref{claim}]
        If \(n \geq 2\), then by taking \(k_0' = k_0\), the claim follows, since \(\underline{\pi}_0(f)\) is bijective and, for every \(i \in \{1, \dots, n\}\) and every \(k \geq k_0'\), \(\operatorname{im}(\alpha_{ik})\) separates \(S\) and any two distinct components of \(f^{-1}(\operatorname{im}(\alpha_{ik}))\) co-bound an annulus in \(S'\).

       Now assume \(n=1\). For any \(k \geq k_0\), let \(X'_{1k}\) be the union of all annuli in \(S'\) whose boundary components are components of \(f^{-1}(\operatorname{im}(\alpha_{1k}))\); in particular, \(X'_{1k}\) contains every component of \(f^{-1}(\operatorname{im}(\alpha_{1k}))\). Since \(\{f^{-1}(\operatorname{im}(\alpha_{1k}))\}\) converges to \(e_1'\), there exists \(k_0' > k_0\) such that, for every \(k \geq k_0'\), \(X'_{1k} \cap f^{-1}(\operatorname{im}(\alpha_{1k_0})) = \varnothing\). Hence, for every \(k \geq k_0'\), each component of \(f^{-1}(\operatorname{im}(\alpha_{1k}))\) separates \(S'\), since \(f^{-1}(\operatorname{im}(\alpha_{1k_0})) \neq \varnothing\) and the unbounded component of \(S \setminus \operatorname{im}(\alpha_{1k})\) does not contain \(\operatorname{im}(\alpha_{1k_0})\) (recall that \(\alpha_1\) bounds \(e_1\)).
\end{proof} 
       
       For \(i \in \{1, \dots, n\}\), let \(A'_i\) be a properly embedded essential bordered subsurface of \(S'\) such that \(\partial A'_i\) is a component of \(f^{-1}(\operatorname{im}(\alpha_{ik_0}))\) and \(\underline{\operatorname{int}(A'_i)} = \{e'_i\}\) (i.e., \(A'_i\) is either a punctured disk or a one-holed Loch Ness monster surface). Since, for each \(i \in \{1, \dots, n\}\), the sequence \(\{f^{-1}(\operatorname{im}(\alpha_{ik})) : k \geq k_0\}\) converges to \(e'_i\), by choosing a sufficiently large \(k_1 \geq k_0\) we may assume that \(\operatorname{int}(A'_i)\) contains \(\bigcup_{k \geq k_1} f^{-1}(\operatorname{im}(\alpha_{ik}))\) for all \(i\in \{1,...,n\}\). Moreover, since \(S'\) is non-planar, we may further assume that \(\partial A'_i\) does not co-bound an annulus in \(S'\) with \(\partial A'_j\) for \(i \neq j\).

       Suppose \(\ker \pi_1(f) \neq 0\). Choose a loop \(\ell\) in \(S'\) representing a non-trivial element of \(\ker \pi_1(f)\). Without loss of generality, we may assume \(\ell \subset f^{-1}(C_{k_1})\). Let \(X'\) be the component of \(f^{-1}(C_{k_1})\) containing \(\ell\). Then \(X'\) cannot be an annulus, since \(f\) maps each component of \(\partial X'\) homeomorphically onto a component of \(\partial C_{k_1}\) and \(\ell \subset X'\). Therefore, if \(\mu'\) and \(\nu'\) are two distinct components of \(\partial X'\), then \(\mu'\) and \(\nu'\) lie in \(A'_i\) and \(A'_j\), respectively, for some distinct \(i,j \in \{1, \dots, n\}\), and thus \(\mu'\) does not co-bound an annulus in \(S'\) with \(\nu'\); that is, \(f(\mu') \cap f(\nu') = \varnothing\). This follows from the properties of the \(A'_i\)s given above and the fact that each component of \(\partial X'\) separates \(S'\). Now, an argument similar to that in the fifth paragraph of the proof of \Cref{deg} shows that \(f\vert\partial X' \to \partial C_{k_1}\) is a homeomorphism and \(f\vert X' \to C_{k_1}\) is a map of degree \(\pm 1\). By \Cref{allow} and \Cref{nearboundary}, \(f\vert X'\to C_{k_1}\) can be homotoped rel \(\partial X'\) to send an essential handle to a point. Thus, \(f\) can be properly homotoped to pinch at least one essential handle to a point.

      We now prove the remaining part of the theorem. Suppose that, for some \(i_0 \in \{1, \dots, n\}\), the induced map \(\widehat{\underline \pi_1}(f) \colon \widehat{\underline \pi_1}(S', e'_{i_0}) \to \widehat{\underline \pi_1}(S, e_{i_0})\) is not injective. We first show that \(e'_{i_0}\) is a non-planar end of \(S'\). Assume, for contradiction, that \(e'_{i_0}\) is planar. Then, for some sufficiently large \(k \ge k_1\), the closure \(B'_*\) of the unbounded component of \(A'_{i_0} \setminus f^{-1}(\operatorname{im}(\alpha_{i_0 k}))\) is a punctured disk. Denote by \(B_*\) the closure of the unbounded component of \(S \setminus \operatorname{im}(\alpha_{i_0 k})\) corresponding to \(e_{i_0}\). Then \(f(B'_*) \subseteq B_*\) and \(f\vert \partial B'_* \to \partial B_*\) is a homeomorphism. Orient both \(B'_*\) and \(B_*\). By \Cref{boundaryhomeo}, \(\deg(f\vert B'_* \to B_*) = \pm 1\), so \(f\vert B'_* \to B_*\) is \(\pi_1\)-surjective \cite[Corollary 3.4]{MR192475}. Since \(\pi_1(B'_*) \cong \mathbb{Z}\), it follows that \(\pi_1(B_*) \cong \mathbb{Z}\), and hence \(B_*\) is also a punctured disk. Identifying \(B'_*\) and \(B_*\) with \(\{x \in \mathbb{R}^2 : 0 < |x| \le 1\}\), one can directly use the definition of the homotopy in Alexander's trick \cite[Proof of Lemma 2.1]{MR2850125} to construct a proper homotopy \(F \colon B'_* \times [0,1] \to B_*\) rel \(\partial B'_*\) from \(f\vert B'_*\to B_*\) to the map \(B_*'\ni x\mapsto |x| f(x/|x|)\in B_*\). In particular, \(f\) can be properly homotoped rel \(S' \setminus \operatorname{int}(B'_*)\) to send \(B'_*\) homeomorphically onto \(B_*\), which implies that \(\widehat{\underline \pi_1}(f) \colon \widehat{\underline \pi_1}(S', e'_{i_0}) \to \widehat{\underline \pi_1}(S, e_{i_0})\) is injective—a contradiction. Hence \(e'_{i_0}\) must be non-planar.

    We next show that there exists a pairwise disjoint collection $\{h_m':m\in \mathbb N\}$ of essential handles in $S'$ and a proper map $g\colon S'\to S$ such that this collection converges to $e_{i_0}'$, $f$ is properly homotopic to $g$, and $g(h_m')$ is a point for each $m\in \mathbb N$. 
    
    Let $\{D_p':p\in \mathbb N\}$ be the set of closures of all components of $A_{i_0}'\setminus \bigcup_{k\geq  k_1}f^{-1}(\operatorname{im}(\alpha_{i_0k}))$. Then each $D_p'$ is an essential bordered subsurface of $S'$ and  $\operatorname{int}(D_p')\cap \operatorname{int}(D_q')=\varnothing$ if $p\neq q$, since every component of $\bigcup_{k\geq  k_1}f^{-1}(\operatorname{im}(\alpha_{i_0k}))$ is a non-trivial separating circle on $S'$ (see \Cref{claim}). Choose $p_0\in \mathbb N$ so that $\partial D_p'\cap f^{-1}(\operatorname{im}(\alpha_{i_0k_1}))=\varnothing$ for all $p\geq p_0$.

    Let $D'\in \{D_p':p\geq p_0\}$. Then $\partial D'$ contains one of the components of $f^{-1}(\operatorname{im}(\alpha_{i_0l}))$ for some $l\geq k_1+1$. Hence, the connected subset $f(\operatorname{int}(D'))$ is contained in a component $G$ of $S\setminus \bigcup_{k\geq  k_1}\operatorname{im}(\alpha_{i_0k})$ such that one of the two components of $\partial G$ is $\operatorname{im}(\alpha_{i_0l})$; that is, following the notation of the first paragraph, $\operatorname{cl}(G)$ is either $D_{i_0l}$ or $D_{i_0(l-1)}$. So $f(D')\subseteq \operatorname{cl}(G)$. In particular, $D'$ is compact. Let $D'$ be homeomorphic to $S_{g,b}$ for some $g\geq 0$ and $b\geq 1$. Since $f$ sends each component of $\partial D'$ homeomorphically onto a component of $\partial G$, we have $b\geq 2$ (otherwise, $b$ would equal $1$, in which case the diagram $$\begin{tikzcd}
{H_2(D',\partial D')=\mathbb Z} \arrow[r, "1\mapsto 1"] \arrow[d, "H_2(f\vert D')"'] & \mathbb Z=H_1(\partial D') \arrow[d, "H_1(f\vert \partial D')\colon 1\mapsto \pm 1\oplus 0"] \\
{H_2(G,\partial G)=\mathbb Z} \arrow[r, "1\mapsto 1\oplus 1"]                        & \mathbb Z\oplus \mathbb Z=H_1(\partial G)                                               
\end{tikzcd}$$ would fail to commute). In fact, $b=2$, since for any $k\geq k_0$, any two distinct components of $f^{-1}(\operatorname{im}(\alpha_{i_0k}))$ co-bounds an annulus in $S'$. Now, two cases arise depending on whether $f\vert \partial D'\to \partial G$ is a homeomorphism or not.

Suppose $f\vert \partial D'\to \partial G$ is not a homeomorphism, that is, $f$ sends both the components of $\partial D'$ homeomorphically onto $\operatorname{im}(\alpha_{i_0l})$. In this case, $D'$ is an annulus co-bounded by two components of $f^{-1}(\operatorname{im}(\alpha_{i_0l}))$. By \cite[Theorem 2.1 (b)]{MR2619608}, there exists a homotopy $H_{D'}\colon D'\times [0,1]\to \operatorname{cl}(G)$ rel  $\partial D'$ from $f\vert D'$ to a map $D'\to\operatorname{im}(\alpha_{i_0l})$. In this situation, we call $D'$ a \emph{compressible annulus} and $H_{D'}$ a \emph{compression homotopy for $D'$}. Considering all compression homotopies, $f$ can be properly homotoped (relative to the complement in $S'$ of the union of the interiors of all the compressible annuli) to a map $f_1$ such that $f_1$ sends each compressible annulus onto a component of $\bigcup_{k\geq k_1+1} \operatorname{im}(\alpha_{i_0k})$.

Since $e_{i_0}'$ is a non-planar end, $D_p'$ must be non-planar for infinitely many $p$. Let $\{L_r':r\in \mathbb N\}$ be a set of compact bordered subsurfaces of $S'$ such that: each $L_r'$ is the union of a finitely many elements of $\{D_p':p\geq p_0\}$; each $L_r'$ contains exactly one non-planar element of $\{D_p':p\geq p_0\}$, denoted $D_{p_r}'$; and $\bigcup_r L_r'=\bigcup_{p\geq p_0} D_p'$. In particular, each $L_r'$ is homeomorphic to $S_{g,2}$ for some $g\geq 1$. Moreover, for any two distinct positive integers $r$ and $s$, $L_r' \cap L_s' = \partial L_r' \cap \partial L_s'$ is either empty or a single circle. In fact, we may further assume that $L_r'\cap L_s'\subset \partial D_{p_r}'\cup \partial D_{p_s}'$ for distinct $r$ and $s$.

Recall that for each $D'\in \{D_p':p\geq p_0\}$, there is $k\geq k_1$ such that $f_1(D')\subseteq f(D')\subseteq D_{i_0k}$. Hence, there exists a set $\{L_r:r\in \mathbb N\}$ of essential compact bordered subsurfaces of $S$ such that, for every $r\geq 1$, the following hold: $L_r$ is the union of finitely many elements of $\{D_{i_0k}:k\geq k_1\}$;  $f_1(L_r')\subseteq L_r$; and $f_1\vert \partial L_r'\to \partial L_r$ is a homeomorphism. Since any two components of $f^{-1}(\operatorname{im}(\alpha_{ik}))$, where $k\geq k_1$, co-bounds an annulus in $S'$, and since $L_r'\cap L_s'\subset \partial D_{p_r}'\cup \partial D_{p_s}'$ for $r\neq s$, we can say that $L_r\cap L_s=\partial L_r\cap \partial L_s$ is either empty or a single circle for $r\neq s$. 

By \Cref{allow} and \Cref{nearboundary}, for every $r\geq 1$, there exists a homotopy $H_r\colon L_r'\times [0,1]\to L_r$ rel $\partial L_r'$ from $f_1\vert L_r'\to L_r$ to a map $g_r\colon L_r'\to L_r$ such that $g_r$ is either a homeomorphism or a quotient map that collapses an essential handle in $L_r'$ to a point. This gives a proper homotopy $H\colon S'\times [0,1]\to S$ rel $S'\setminus \operatorname{int}(\bigcup_r L_r')$ from $f_1$ to $g$ such that $H\vert L_r'\times [0,1]=H_r$ for all $r\in \mathbb N$. If $g_r$ were a homeomorphism for all sufficiently large $r$, then $g$ would send a properly embedded essential one-holed Loch Ness monster surface corresponding to $e_{i_0}'$ homeomorphically onto its image, that is, $\widehat{\underline \pi_1}(g)=\widehat{\underline \pi_1}(f)\colon \widehat{\underline \pi_1}(S',e'_{i_0})\to \widehat{\underline \pi_1}(S,e_{i_0})$ would be injective. Hence, there exists a pairwise disjoint collection $\{h_m':m\in \mathbb N\}$ of essential handles in $S'$ converging to $e_{i_0}'$ such that $g(h_m')$ is a point for each $m\in \mathbb N$. 
\end{proof}

\begin{corollary}
    \label{optional}
    If the surfaces \( S' \) and \( S \) in \Cref{infinitepinchproof} are oriented, and the proper map \( f \colon S' \to S \) satisfies \ref{tc}, then \( \deg(f) = \pm 1 \).
\end{corollary}
\begin{proof}
    We follow the notations introduced in the proof of \Cref{infinitepinchproof}. Let $e'$ be an end of $S'$ and $e$ an end of $S$ such that $\underline \pi_0(f)(e')=e$. 
    
    First, suppose that $\widehat{\underline \pi_1}(f)\colon \widehat{\underline \pi_1}(S',e')\to \widehat{\underline \pi_1}(S,e)$ is injective. Then, the proof of \Cref{infinitepinchproof} provides properly embedded essential bordered subsurfaces $X'$ of $S'$ and $X$ of $S$ with $\underline{\operatorname{int}(X')}=\{e'\}$ and $\underline{\operatorname{int}(X)}=\{e\}$ such that, after a proper homotopy, we may assume that $f\vert X'\to X$ is a homeomorphism. Since $\underline \pi_0(f)^{-1}(e)=\{e'\}$, there exists a disk $\mathfrak D$ in $X$ for which $f\vert f^{-1}(\mathfrak D)\to \mathfrak D$ is a homeomorphism. Thus, by \cite[Lemma~2.1(b)]{MR192475}, $\deg(f)=\pm 1$.

    Now, assume that $\widehat{\underline \pi_1}(f)\colon \widehat{\underline \pi_1}(S',e')\to \widehat{\underline \pi_1}(S,e)$ is not injective. Then, the proof of \Cref{infinitepinchproof} shows that $e'$ must be non-planar. Moreover, the last paragraph of the proof of \Cref{infinitepinchproof} shows that, for all $r$, the map $g_r\colon L_r'\to L_r$ has degree $\pm1$, as it sends $\partial L_r'$ homeomorphically onto $\partial L_r$. Using Hopf's geometric realization of degree \cite[Theorem 4.1]{MR192475}, for all $r$, after a homotopy rel $\partial L_r'$, we may assume that $g_r\vert g_r^{-1}(\mathfrak D_r)\to \mathfrak D_r$ is a homeomorphism for some disk $\mathfrak D_r$ in $L_r\setminus \partial L_r$.  Since $L_r\cap L_s=\partial L_r\cap \partial L_s$ either empty or a single circle when $r\neq s$, and since $\underline \pi_0(g)^{-1}(e)=\{e'\}$, for all sufficiently large $r$ we have $g^{-1}(\mathfrak D_r)=g_r^{-1}(\mathfrak D_r)$. Therefore, $g\vert g^{-1}(\mathfrak D_r)\to \mathfrak D_r$ is a homeomorphism for all sufficiently large $r$. By \cite[Lemma~2.1(b)]{MR192475}, $\deg(g)=\pm 1$, and hence $\deg(f)=\pm1$.
\end{proof}

For the remainder of the paper, we focus on planar surfaces. In contrast to the non-planar case—where proper maps are typically taken to be end-allowable—this assumption in the planar setting forces the degree of the map to vanish. 
\begin{theorem}\label{randomproposition}
   Let \(S'\) and \(S\) be oriented planar surfaces, each with at least three ends. Suppose \(f\colon S'\to S\) is a proper map such that \(\underline{\pi}_0(f)\) is injective. Then the following holds:
\begin{enumerate}[(1)]
    \item\label{randomproposition1} If \(f\) has no geometric kernel, then \(f\) is properly homotopic to a finite-sheeted covering map. In particular, \(\deg(f) \neq 0\).
    \item\label{randomproposition2} If \(f\) has a geometric kernel, then \(\deg(f) = 0\).
\end{enumerate}
\end{theorem}
    \begin{proof}
        Suppose that $f$ has no geometric kernel. We show that $f$ is $\pi_1$-injective. To prove this, let $\{C_n\}$ be an exhausting sequence for $S$ satisfying the properties given in \Cref{exhaustion}. Since $S$ has at least three ends, the first part of the proof of \Cref{exhaustion} allows us, without loss of generality, to assume that $\partial C_n$ has at least three components for every $n$. Next, as in the proof of \Cref{specialcase}, after a proper homotopy we may further assume that every component of $\bigcup_n f^{-1}(\partial C_n)$ is a nontrivial circle in $S'$, and that $f$ restricts to a finite-sheeted covering on every component of $\bigcup_n f^{-1}(\partial C_n)$ (since $f$ has no geometric kernel). Finally, observe that $\{f^{-1}(C_n)\}$ is an exhausting sequence for $S'$.
      
        Consider a non-trivial element $x$ of $\pi_1(S')$. We show $x\not \in \ker\pi_1(f)$. So, choose a loop $\ell'$ representing $x$. Fix a positive integer $m$ such that $\ell'\subset f^{-1}(C_m)$. Let $D'$ be the component of $f^{-1}(C_m)$ containing $\ell'$. Denote the components of $\partial D'$ by $\gamma_1'$,..., $\gamma_k'$. For each $i=1,...,k$, let $A_i'$ be the unique component of the disconnected set $S'\setminus \gamma_i'$ such that $A_i'\cap \operatorname{int}(D')=\varnothing$. If $i\neq j$, then $\underline \pi_0(f)(\underline{A_i'})$ and $\underline \pi_0(f)(\underline{A_j'})$ does not intersect (since $\underline \pi_0(f)$  is injective). Hence $f(\gamma_i')$ and $f(\gamma_j')$ must be two distinct components of $\partial C_m$. Recall that $f$ restricted to each component of $\partial D'$ is a covering map onto a component of $\partial C_m$. By naturality of the long exact sequence in homology, it follows that $\partial C_m$ has exactly $k$ components and that $\deg(f\vert (D',\partial D')\to (C_m,\partial C_m))\neq 0$. In particular, $k\geq 3$.  By \Cref{tool}, $(2-k)=\chi(D')\leq |\deg(f)|\cdot \chi(C_m)=|\deg(f)|\cdot (2-k)$, which implies $\deg(f)=\pm 1$. Thus, by \cite[Corollary 3.4]{MR192475}, $\pi_1(f\vert D')$ is an epimorphism. Since a free group of finite rank is Hopfian, $\pi_1(f\vert D')$ is also a monomorphism. By \Cref{rhofunctor}, $D'$ is an essential subsurface of $S'$. Therefore, $\pi_1(f)(x)\neq 0$. Since $x$ is an arbitrary non-trivial element of $\pi_1(S')$, we conclude that $f$ is $\pi_1$-injective. Finally, by \Cref{specialcaseremark}, $f$ can be properly homotoped to a finite sheeted covering map. In particular, $\deg(f)\neq 0$. This completes the proof of \ref{randomproposition1}.

        Now we prove \ref{randomproposition2}. Assume that $f$ has a geometric kernel. We show that $\deg(f)= 0$. Let $\gamma'$ be a non-trivial circle in $S'$ such that $f(\gamma')$ is null-homotopic loop. Choose a tubular neighborhood $\gamma'\times [1,2]$ with $\gamma'\times \{3/2\}\equiv \gamma'$ in $S'$, and let $M$ be the $2$-manifold obtained by removing $\gamma'\times(1,2)$ from $S'$ and gluing a disk $B_j$ along each $\gamma'\times\{j\}$ for $j=1,2$. Thus $M$ has exactly two components, say $M_1$ and $M_2$, where $\gamma'\times\{j\}=\partial B_j\subset B_j\subset M_j$ for $j=1,2$. Notice that both $M_1$ and $M_2$ are non-compact. Moreover, since $f(\gamma')$ is null-homotopic, the restriction $f\vert S'\setminus(\gamma'\times(1,2))$ extends to a proper map $\overline f\colon M\to S$. Denote by $\overline f_j$ the restriction $\overline f\vert M_j\to S$ for $j=1,2$. We consider two cases: either $\deg(\overline f_1)=\deg(\overline f_2)=0$, or at least one of them is nonzero.

        First, suppose $\deg(\overline f_1)=\deg(\overline f_2)=0$. Then, by Hopf's geometric realization of degree \cite[Theorem~4.1]{MR192475}, for $j=1,2$ there exists a disk $\mathfrak D_j \subset S$ and a proper homotopy $H_j \colon M_j \times [0,1] \to S$ from $\overline f_j$ to a proper map $g_j$ such that $g_j^{-1}(\mathfrak D_j)=\varnothing$. The Palais disk theorem \cite[Theorem~B]{MR117741} gives a diffeomorphism $\varphi \colon S \to S$ with $\varphi(\mathfrak D_1)=\mathfrak D_2$ such that $\varphi$ is homotopic through diffeomorphisms to $\operatorname{id}_S$. Moreover, it follows from \cite[Theorem~1.3]{MR181989} that a homotopy through self-homeomorphisms of a manifold is proper. Therefore, replacing $g_1$ with $\varphi g_1$, we may assume there is a disk $\mathfrak D \subset S$ such that $g_1^{-1}(\mathfrak D)=\varnothing=g_2^{-1}(\mathfrak D)$. Choose a compact bordered subsurface $C \subset S$ with $f(\gamma' \times [1,2]) \cup H_1(\gamma' \times \{1\} \times [0,1]) \cup H_2(\gamma' \times \{2\} \times [0,1]) \subseteq C.$ Using the Palais disk theorem, we may further assume that $\mathfrak D \cap C = \varnothing$. By the homotopy extension theorem, there exists a homotopy $G \colon \gamma' \times [1,2] \times [0,1] \to C$ from $f\vert\gamma' \times [1,2]\to C$ extending the homotopies $H_j \vert \gamma' \times \{j\} \times [0,1]\to C$ for $j=1,2$.  Finally, pasting $G$ with $H_j \vert(M_j \setminus \operatorname{int}(B_j)) \times [0,1]$ for $j=1,2$ yields a proper homotopy from $f$ to a proper map $g \colon S' \to S$ with $g^{-1}(\mathfrak D)=\varnothing$. Hence $\deg(f)=0$. This completes the first case.

       Now, assume that one of $\deg(\overline f_1)$ or $\deg(\overline f_2)$ is nonzero. Without loss of generality, suppose $\deg(\overline f_1)\neq 0$. Since a proper map between manifolds is closed \cite{MR254818}, we have $\operatorname{im}(\overline f_1)=S$. Hence $\underline \pi_0(\overline f_1)\colon \underline \pi_0(M_1)\to \underline \pi_0(S)$ is surjective. Because $M_2$ is non-compact, the image of $\underline \pi_0(\overline f_2)$ intersects the image of $\underline \pi_0(\overline f_1)$. This contradicts the injectivity of $\underline \pi_0(f)$, since for each $j=1,2$ we have $f\vert M_j\setminus \operatorname{int}(B_j)=f_j\vert M_j\setminus \operatorname{int}(B_j)$. Therefore, $\deg(f)=0$.
    \end{proof}
    Here is an example of a proper map that satisfies the hypothesis of \ref{randomproposition2} of \Cref{randomproposition}. 
\begin{example}
Let $E$ be a compact, totally disconnected subset of $[-1,1]\times \{0\}$. Define $f\colon\mathbb R^2\setminus E\to \mathbb R^2\setminus E$ by $f(x,y)\coloneqq (x,|y|)$. Then, $\deg(f)=0$, since $f$ is a non-surjective proper map. Moreover, $\underline \pi_0(f)$ is bijective and $f$ has a geometric kernel (for example, any simple loop whose image is $\{z\in \mathbb{R}^2 : |z|=2\}$ represents a nontrivial element of $\ker \pi_1(f)$).
\end{example}

The following theorem does not assert the existence of a geometric kernel, but instead describes the situation when the end-allowability assumption in \Cref{randomproposition} is omitted. %It can be interpreted as a step toward detecting a geometric kernel for a degree-one, non–end-allowable map between two oriented planar surfaces.

\begin{theorem}\label{geokerplanar}
    Let \( S' \) and \( S \) be oriented planar surfaces, each with at least three ends, and let \( f \colon S' \to S \) be a proper map of degree one. Suppose there exists an end \( e \) of \( S \) such that the preimage \( \underline{\pi}_0(f)^{-1}(e) \) is a finite set of cardinality at least two, and that for each \( e' \in \underline{\pi}_0(f)^{-1}(e) \), the induced map \(\widehat{\underline{\pi}_1}(f) \colon \widehat{\underline{\pi}_1}(S', e') \to \widehat{\underline{\pi}_1}(S, e)\) is an isomorphism. Then there exists a simple loop in \( S' \) representing a non-trivial element of \( \ker  H_1(f) \).
\end{theorem}
\begin{proof}
      Consider an exhausting sequence $\{C_i : i = 0,1,2,\dots\}$ for $S$ satisfying the properties given in \Cref{exhaustion}. Any proper map that is properly homotopic to $f$ must be surjective. Indeed, a proper map between manifolds is closed \cite{MR254818}, and the non-zero integer $\deg(f)$ is preserved under proper homotopy. Therefore, as in the proof of \Cref{specialcase}, after a proper homotopy we may assume that for each $i \geq 0$, every component of the non-empty set $f^{-1}(\partial C_i)$ is a non-trivial circle in $S'$, and the restriction of $f$ to each component of $f^{-1}(\partial C_i)$ is either constant or a finite-sheeted covering map. Moreover, the collection $\{f^{-1}(C_i) : i = 0,1,2,\dots\}$ forms an exhausting sequence for $S'$.

      Recall that the induced map between the spaces of ends remains unchanged under proper homotopy. Let $e'$ be an end of $S'$ with $\underline{\pi}_0(f)(e') = e$. Consider a sequence of circles $\alpha = (\alpha_0, \alpha_1, \alpha_2, \dots)$ converging to $e$, where $\operatorname{im}(\alpha_i)$ is a component of $\partial C_i$ for each $i \geq 0$. Then there exists a sequence of circles $\alpha' = (\alpha_0', \alpha_1', \alpha_2', \dots)$ converging to $e'$ such that $\operatorname{im}(\alpha_i')$ is a component of $f^{-1}(\operatorname{im}(\alpha_i))$ for each $i \geq 0$. By the previous paragraph, there exists a sequence $\{n_0, n_1, n_2, \dots\}$ of integers such that, possibly after re-parametrizing each $\alpha_i'$, we have $f\alpha_i' = \alpha_i^{n_i}$ for each $i \geq 0$. Since $\widehat{\underline{\pi}_1}(f)$ is surjective, there exists a sequence of loops $\beta' = (\beta_0', \beta_1', \beta_2', \dots)$ converging to $e'$ such that $\widehat{\underline \pi_1}(f)$ sends $[\beta']\in \widehat{\underline{\pi}_1}(S', e')$ to $[\alpha]\in \widehat{\underline{\pi}_1}(S, e)$. Then $\widehat{\underline{\pi}_1}(f)$ sends both $[(\beta_0'^{n_0}, \beta_1'^{n_1}, \beta_2'^{n_2}, \dots)]$ and $[\alpha']$ to $[(\alpha_0^{n_0}, \alpha_1^{n_1}, \alpha_2^{n_2}, \dots)]$. Since $\widehat{\underline{\pi}_1}(f)$ is injective, it follows that $\beta_i'^{n_i}$ is freely homotopic to $\alpha_i'$ for all sufficiently large $i$. By \Cref{lemmapre} and \Cref{lemma}, for all sufficiently large $i$, the non-trivial element of $\pi_1(S',*)$ represented by the circle $\alpha_i'$ is the $n_i$th power of some element of $\pi_1(S',*)$. Now, \cite[Theorems~1.7 and~4.2]{MR214087} imply that for all sufficiently large $i$ we have $n_i = \pm 1$, and hence $f\vert\operatorname{im}(\alpha_i') \to \operatorname{im}(\alpha_i)$ is a homeomorphism. For each $i\geq 0$, let $A_i'$ (resp. $A_i$) denote the closure of the component of $S' \setminus \operatorname{im}(\alpha_i')$ (resp. $S \setminus \operatorname{im}(\alpha_i)$) that corresponds to a neighborhood of $e'$ (resp. $e$). Then, for all sufficiently large $i$, the restriction $f\vert A_i' \to A_i$ sends $\partial A_i'$ homeomorphically onto $\partial A_i$. Recall that $\partial A_i' = \operatorname{im}(\alpha_i')$ (resp. $\partial A_i = \operatorname{im}(\alpha_i)$), which is a component of $f^{-1}(\partial C_i)$ (resp. $\partial C_i$) for all $i \geq 0$.

      We repeat the above process for every $e' \in \{e_1', \dots, e_l'\} = \underline{\pi}_0(f)^{-1}(e)$, where $l \geq 2$. Thus, there exists a positive integer $k_0$ such that for some unbounded component $U$ of $S \setminus C_{k_0}$ with $e \in \underline{U}$ and some unbounded components $U_1', \dots, U_l'$ of $S' \setminus f^{-1}(C_{k_0})$ with $e_1' \in \underline{U_1'}, \dots, e_l' \in \underline{U_l'}$, the map $f_j \coloneqq f\vert \operatorname{cl}(U_j') \to \operatorname{cl}(U)$ sends the circle $\partial U_j'$ homeomorphically onto the circle $\partial U$ for each $j = 1, \dots, l$. Moreover, we may assume that the sets $\operatorname{cl}(U_1'), \dots, \operatorname{cl}(U_l')$ are pairwise disjoint. By \Cref{boundaryhomeo}, for each $j = 1, \dots, l$ we have $\varepsilon_j \coloneqq \deg(f_j) = \deg(f_j\vert \partial U_j' \to \partial U) = \pm 1$, where the orientations of $\operatorname{cl}(U_j')$ and $\operatorname{cl}(U)$ are induced from the orientations of $S'$ and $S$, respectively.

      \begin{claim}\label{claimsum}
          $1=\varepsilon_1+\cdots+\varepsilon_l$.
      \end{claim} 
      \begin{proof}[Proof of \Cref{claimsum}]
          By Kerékjártó's classification theorem \cite[Theorem~1]{MR143186}, $S'$ (resp. $S$) is homeomorphic to $\mathbb{S}^2 \setminus E'$ (resp. $\mathbb{S}^2 \setminus E$) for some compact, totally disconnected subset $E'$ (resp. $E$) of $\mathbb{S}^2$, with $E'$ (resp. $E$) homeomorphic to $\underline{\pi}_0(S')$ (resp. $\underline{\pi}_0(S)$). Thus, by a slight abuse of notation, we may identify $S' = \mathbb{S}^2 \setminus E'$, $E' = \underline{\pi}_0(S')$, $S = \mathbb{S}^2 \setminus E$, and $E = \underline{\pi}_0(S)$. Since $f$ is proper, it extends to a proper map $\widetilde{f} \colon \mathbb{S}^2 \setminus \{e_1', \dots, e_l'\} \to \mathbb{S}^2 \setminus \{e\}$.

          We first show that $\deg(\widetilde f)=1$, where the orientations of the domain and codomain of $\widetilde f$ are induced from those of $\mathbb S^2\setminus E'$ and $\mathbb S^2\setminus E$, respectively. By Hopf's geometric realization of degree \cite[Theorem~4.1]{MR192475}, there exists a proper homotopy $H\colon (\mathbb S^2\setminus E')\times [0,1]\to \mathbb S^2\setminus E$, relative to $(\mathbb S^2\setminus E')\setminus K'$ for some compact subset $K'\subset \mathbb S^2\setminus E'$, from $f$ to a proper map $g$ such that, for some disk $\mathfrak D\subset \mathbb S^2\setminus E$, the restriction $g\vert g^{-1}(\mathfrak D)\to \mathfrak D$ is an orientation-preserving homeomorphism. Consequently, $H$ extends to a proper homotopy $\widetilde H\colon (\mathbb S^2\setminus \{e_1',\dots,e_l'\})\times [0,1]\to \mathbb S^2\setminus \{e\}$, relative to $(\mathbb S^2\setminus \{e_1',...,e_l'\})\setminus K'$, from $\widetilde f$ to an extension $\widetilde g$ of $g$. Since $\widetilde g\vert \widetilde g^{-1}(\mathfrak D)\to \mathfrak D$ is an orientation-preserving homeomorphism, it follows from \cite[Lemma~2.1(b)]{MR192475} that $\deg(\widetilde g)=1$, and hence $\deg(\widetilde f)=1$.

          Next, define $V_j'\coloneqq U_j'\sqcup (E'\setminus \{e_j'\})$ for $j=1,\dots,l$. Then each $\operatorname{cl}(V_j')$ is a properly embedded punctured disk in $\mathbb S^2\setminus \{e_1',\dots,e_l'\}$ corresponding to the end $e_j'$. Similarly, let $V\coloneqq U\sqcup (E\setminus \{e\})$; then $\operatorname{cl}(V)$ is a properly embedded punctured disk in $\mathbb S^2\setminus \{e\}$ corresponding to the end $e$. Consider the orientation of $\operatorname{cl}(V_j')$ (resp. $\operatorname{cl}(V)$) induced from the orientation of the domain (resp. codomain) of $\widetilde f$. Then, for each $j=1,\dots,l$, by \Cref{boundaryhomeo}, $\deg(\widetilde f\vert\operatorname{cl}(V_j')\to \operatorname{cl}(V))= \deg(\widetilde f\vert\partial V_j'\to \partial V)= \deg(f\vert\partial U_j'\to \partial U)= \deg(f_j)= \varepsilon_j,$ where the second equality follows from the facts that the orientation of $\operatorname{cl}(V_j')$ (resp. $\operatorname{cl}(V)$) matches that of $\operatorname{cl}(U_j')$ (resp. $\operatorname{cl}(U)$), and that both $\widetilde f$ and $f$, when restricted to $\partial V_j'=\partial U_j'$, are the same homeomorphism onto $\partial V=\partial U$.

          The homotopy from Alexander's trick \cite[Proof of Lemma~2.1]{MR2850125} can be used directly to construct, for each $j=1,\dots,l$, a proper homotopy $h_j\colon \operatorname{cl}(V_j')\times [0,1]\to \operatorname{cl}(V)$ rel $\partial V_j'$ from $\widetilde f\vert\operatorname{cl}(V_j')\to \operatorname{cl}(V)$ to a homeomorphism. Since the sets $\operatorname{cl}(V_1'),\dots,\operatorname{cl}(V_l')$ are pairwise disjoint, the proper homotopies $h_1,\dots,h_l$, taken together, give a proper homotopy of $\widetilde f$, relative to the compact set $X'\coloneqq (\mathbb S^2\setminus \{e_1',\dots,e_l'\})\setminus \bigcup_{j=1}^l V_j'$, to a proper map $\widetilde f_{\mathrm{At}}$ such that $\widetilde f_{\mathrm{At}}\vert\operatorname{cl}(V_j')\to \operatorname{cl}(V)$ is a homeomorphism for each $j=1,\dots,l$. Since $\widetilde f_{\mathrm{At}}$ is properly homotopic to $\widetilde f$, we have $\deg(\widetilde f_{\mathrm{At}})=\deg(\widetilde f)=1$. Also, for each $j=1,\dots,l$, the degree of $\widetilde f_{\mathrm{At}}\vert\operatorname{cl}(V_j')\to \operatorname{cl}(V)$ is $\varepsilon_j$, since it is properly homotopic rel $\partial V_j'$ to $\widetilde f\vert\operatorname{cl}(V_j')\to \operatorname{cl}(V)$. Choose a disk $\mathfrak D\subset V$ with $\mathfrak D\cap \widetilde f_{\mathrm{At}}(X')=\varnothing$. Then, for each $j=1,\dots,l$, there exists a disk $\mathfrak D_j'\subset V_j'$ such that $\widetilde f_{\mathrm{At}}^{-1}(\mathfrak D)=\bigsqcup_{j=1}^l\mathfrak D_j'$. By \cite[Lemma~2.1(b)]{MR192475}, for each $j=1,\dots,l$, the homeomorphism $\widetilde f_{\mathrm{At}}\vert\mathfrak D_j'\to \mathfrak D$ is orientation-preserving or orientation-reversing according as $\varepsilon_j=+1$ or $\varepsilon_j=-1$. Hence $\deg(\widetilde f_{\mathrm{At}})=\varepsilon_1+\dots+\varepsilon_l$ \cite[Lemma~2.1(b)]{MR192475}. Thus, $1=\deg(\widetilde f)=\deg(\widetilde f_{\mathrm{At}})=\varepsilon_1+\dots+\varepsilon_l$.
      \end{proof}
    Since $l\geq 2$ and each $\varepsilon_j$ is either $+1$ or $-1$, by \Cref{claimsum} it follows that $l\geq 3$, and hence there exist $r,s\in \{1,\dots,l\}$ such that $\varepsilon_r\varepsilon_s=-1$. Fix an orientation of $\mathbb S^1$. Recall that for each $j=1,\dots,l$, the boundary $\partial U_j'$ of $\operatorname{cl}(U_j')$ inherits its orientation from $\operatorname{cl}(U_j')$, and similarly, the boundary $\partial U$ of $\operatorname{cl}(U)$ inherits its orientation from $\operatorname{cl}(U)$. Let $\gamma_r'\colon \mathbb S^1\to \partial U_r'$ and $\gamma_s'\colon \mathbb S^1\to \partial U_s'$ be orientation-preserving homeomorphisms. Since the degrees of the homeomorphisms $f\gamma_r'\colon \mathbb S^1\to \partial U$ and $f\gamma_s'\colon \mathbb S^1\to \partial U$ are $\varepsilon_r$ and $\varepsilon_s$, respectively, the map $f\gamma_r'$ is freely homotopic to $\overline{f\gamma_s'}$.

    Choose an essential pair of pants \(P'\subset S'\) whose boundary has \(\partial U_r'\) and \(\partial U_s'\) as two of its components. This is possible because \(\operatorname{cl}(U_r')\cap \operatorname{cl}(U_s')=\varnothing\) and $l\geq 3$. Let \(\eta'\colon \mathbb S^1\to \partial P'\setminus(\partial U_r'\cup \partial U_s')\) be a homeomorphism. Then, in \(H_1(P';\mathbb Z)\), we can write either \([\eta']=[\gamma_r']+[\gamma_s']\) or \(-[\eta']=[\overline{\eta'}]=[\gamma_r']+[\gamma_s']\). We show \(H_1(f)\colon H_1(S';\mathbb Z)\to H_1(S;\mathbb Z)\) sends $[\eta']$ to the trivial element. So, without loss of generality, assume \([\eta']=[\gamma_r']+[\gamma_s']\) in \(H_1(P';\mathbb Z)\). Applying the inclusion-induced homomorphism \(H_1(P';\mathbb Z)\to H_1(S';\mathbb Z)\), we obtain \([\eta']=[\gamma_r']+[\gamma_s']\) in \(H_1(S';\mathbb Z)\). But \(H_1(f)\) sends \([\gamma_r']+[\gamma_s']\) to \([\overline{f\gamma_s'}]+[f\gamma_s']=-[f\gamma_s']+[f\gamma_s']=0\). Therefore, to complete the proof, it is enough to show that \([\eta']\) is a non-trivial element of \(H_1(S';\mathbb Z)\). Since either component of \(S'\setminus\operatorname{im}(\eta')\) is unbounded, this follows if we use simplicial homology theory (alternatively, one may note that if \(j\in\{1,\dots,l\}\setminus\{r,s\}\), then \(\eta'\) corresponds to a generator of the infinite cyclic group \(H_1(\mathbb S^2\setminus\{e_r',e_j'\};\mathbb Z)\)).
\end{proof}
\section*{Acknowledgments}
The author thanks Prof. Siddhartha Gadgil for suggesting the question of whether Edmonds’ theory can be extended to non-compact surfaces. The author also acknowledges support from the Institute Postdoctoral Fellowship at IIT Bombay and thanks Prof. Rekha Santhanam for hosting the postdoctoral position.
\bibliography{bibliography}
\bibliographystyle{plain}
\end{document}